\documentclass[a4paper, 11pt]{article}

\input xy
\xyoption{all}

\newcounter{ExacSeq}
\newcommand{\ES}{{(ES\theExacSeq)\stepcounter{ExacSeq}}}

\usepackage{amsmath,amsthm,amssymb}
\usepackage[page,header]{appendix}
\usepackage{morefloats}
\usepackage{color}
\usepackage{makeidx}
\usepackage{fancybox}
\usepackage{mathabx}
\usepackage{tikz}
\usetikzlibrary{matrix}
\usepackage{arcs}
\usepackage{hyperref}

\makeindex 

\usepackage{latexsym, amsmath, amssymb, amsfonts, amscd}
\usepackage{amsthm}
\usepackage{t1enc}
\usepackage[mathscr]{eucal}
\usepackage{indentfirst}
\usepackage{graphicx, pb-diagram}
\usepackage{enumerate}
\usepackage[all,poly,web,knot]{xy}
\usepackage{float}
\restylefloat{figure}

\newcommand{\Title}{Title}

\numberwithin{equation}{section}

{\theoremstyle{definition}\newtheorem{definition}{Definition}[section]

\newtheorem{defititle}[definition]{\Title}
\newtheorem{notation}[definition]{Notation}

\newtheorem{remark}[definition]{Remark}
\newtheorem{remarks}[definition]{Remarks}
\newtheorem{assumption}[definition]{Assumption}

\newtheorem{ex}[definition]{Example}

\newtheorem{exs}[definition]{Examples}}
\newtheorem{prop}[definition]{Proposition}
\newtheorem{proposition-definition}[definition]{Proposition-Definition}
\newtheorem{lemma}[definition]{Lemma}
\newtheorem{thm}[definition]{Theorem}
\newtheorem{cor}[definition]{Corollary}

\newtheorem*{prop*}{Proposition}
\newtheorem*{theorem*}{Theorem}

\newcommand{\cB}{\mathcal{B}}
\newcommand{\cG}{\mathcal{G}}

\newcommand{\cF}{\mathcal{F}}
\newcommand{\cE}{\mathcal{E}}

\newcommand{\cK}{\mathcal{K}}
\newcommand{\cT}{\mathcal{T}}
\newcommand{\gT}{\mathfrak{T}}
\newcommand{\gG}{\mathfrak{G}}
\newcommand{\cJ}{\mathcal{J}}

\newcommand{\cI}{\mathcal{I}}

\newcommand{\bA}{\mathbb{A}}
\newcommand{\bG}{\mathbb{G}}
\newcommand{\bH}{\mathbb{H}}
\newcommand{\bK}{\mathbb{K}}

\newcommand{\cH}{\mathcal{H}}
\newcommand{\cM}{\mathcal{M}}
\newcommand{\bE}{\mathbb{E}}
\newcommand{\resp}{{\it resp.}\/ }
\newcommand{\id}{{\hbox{id}}}

\newcommand{\ie}{{\it i.e.}\/ }
\newcommand{\eg}{{\it e.g.}\/ }
\newcommand{\cf}{{\it cf.}\/ }

\newcommand{\X}{\mathcal{X}}

\newcommand{\DNC}{DNC}
\def\gpd{\,\lower1pt\hbox{$\longrightarrow$}\hskip-.24in\raise2pt
             \hbox{$\longrightarrow$}\,}

\renewcommand\theenumi{\alph{enumi}}
\renewcommand\labelenumi{\rm {\theenumi})}

\renewcommand{\latticebody}{\drop@{ }}

\newcommand{\rra}{\rightrightarrows}

\newcommand{\N}{\ensuremath{\mathbb N}}
\newcommand{\Z}{\ensuremath{\mathbb Z}}
\newcommand{\C}{\ensuremath{\mathbb C}}
\newcommand{\R}{\ensuremath{\mathbb R}}

\newcommand{\cL}{{\mathcal L}}

\newcommand{\cC}{\mathcal{C}}            

\newcommand{\cZ}{\mathcal{Z}}

\DeclareMathOperator{\coker}{coker}

\newcommand{\pih}{\pi}

\def\act{\mathbin{\hbox{$<\kern-.4em\mapstochar\kern.4em$}}}
\def\ract{\mathbin{\hbox{$\mapstochar\kern-.3em>$}}}

\def\PB(#1,#2,#3,#4){\left\{\begin{matrix}#1&\!\!\!\stackrel{?}{\longrightarrow}&\!\!\!#2\\
\downarrow&&\!\!\!\downarrow\\
#3&\!\!\!\stackrel{?}{\longrightarrow}&\!\!\!#4\end{matrix}\right\}}

\def\pb(#1,#2,#3,#4){ \hom(#1 \to #3, #2 \to #4)}

\begin{document}


\begin{center}
{\Large\bf A Baum-Connes conjecture for singular foliations\footnote{AMS subject classification: 46L87,~ Secondary 58J22, 19K45, 53C22, 22A25. Keywords: Baum-Connes conjecture, singular foliation, singularity height.} 

\bigskip

{\sc by Iakovos Androulidakis\footnote{I. Androulidakis was partially supported by SCHI 525/12-1 (DFG, Germany) and FP7-PEOPLE-2011-CIG, Grant No PCI09-GA-2011-290823 (Athens).} and Georges Skandalis\footnote{G. Skandalis was partially supported by ANR-14-CE25-0012-01 (France)}}
}
 
\end{center}

{\footnotesize
Department of Mathematics 
\vskip -4pt National and Kapodistrian University of Athens
\vskip -4pt Panepistimiopolis
\vskip -4pt GR-15784 Athens, Greece
\vskip -4pt e-mail: \texttt{iandroul@math.uoa.gr}
  
\vskip 2pt Universit\'e Paris Diderot, Sorbonne Paris Cit\'e
\vskip-4pt  Sorbonne Universit\'es, UPMC Paris 06, CNRS, IMJ-PRG
\vskip-4pt  UFR de Math\'ematiques, {\sc CP} {\bf 7012} - B\^atiment Sophie Germain 
\vskip-4pt  5 rue Thomas Mann, 75205 Paris CEDEX 13, France
\vskip-4pt e-mail: \texttt{georges.skandalis@imj-prg.fr}
}
\bigskip
\everymath={\displaystyle}

\date{today}

\begin{abstract}\noindent 
We consider singular foliations whose holonomy groupoid may be nicely decomposed using Lie groupoids (of unequal dimension). We construct a $K$-theory group and a natural assembly-type morphism to the $K$-theory of the foliation $C^*$-algebra generalizing to the singular case the Baum-Connes assembly map. This map is shown to be an isomorphism under assumptions of amenability. We examine several examples that can be described in this way and make explicit computations of their $K$-theory.
\end{abstract}
 
\setcounter{tocdepth}{2} 
\tableofcontents

\section*{Introduction}
\addcontentsline{toc}{section}{Introduction}

Let $(M,\cF)$ be a singular Stefan-Sussmann foliation \cite{Stefan, Sussmann}. In \cite{AS1}, we constructed its holonomy groupoid  and the foliation $C^*$-algebra $C^*(M,\cF)$. In \cite{AS2,AS3} we showed that the $K$-theory of $C^*(M,\cF)$ is a receptacle for natural index problems along the leaves. It is then important to try to give some insight to this $K$-theory, \ie to construct a kind of Baum-Connes assembly map. Of course we cannot hope in general for such a map to be an isomorphism (since it is not always an isomorphism in the regular case, as shown in \cite{HLS}), and even to be defined for every kind of singular foliation. However, in this paper we manage to construct such a map for a quite general class of singular foliations.

\subsection{Some examples}

In order to formulate the assembly map, let us examine a few natural examples. Consider the foliation given by a smooth action of a connected Lie group on a manifold $M$:\begin{enumerate}
\item the action of $SO(3)$ on $\R^3$;
\item the action of $SL(2,\R)$ on $\R^2$;
\item any action of $\R$ (given by a vector field $X$).
\end{enumerate}
In all of these cases, we can compute the $K$-theory thanks to an exact sequence $$0\to C^*(\Omega_0,\cF_{|\Omega_0})\to C^*(M,\cF)\to C^*(M,\cF)|_{Y_1}\to 0.$$  
Here $\Omega_0$ corresponds to ``most regular points'' of the foliation and $Y_1=M\setminus\Omega_0$: in example (a), $\Omega_0=\R^3\setminus \{0\}$, in example (b), $\Omega_0=\R^2\setminus \{0\}$; in example (c) $\Omega_0$ is the interior of the set of points where $X$ vanishes.

In these examples, the connecting map $\partial$ of the $K$-theory exact sequence is easily computed and we can describe precisely $K_*(C^*(M,\cF))$.

In other examples that we discuss here, the ``regularity'' of points variates even more. For instance:
\begin{enumerate}
\item[d)] the action on $\R^n$ of a parabolic subgroup $G$ of $GL(n,\R)$; \eg the minimal parabolic subgroup of upper triangular matrices. 
\item[e)] the action of $PG = G/\R^*$ on $\R P^{n-1}$.
\item[f)] the action of $G\times G$ by left and right multiplication on $M_n(\R)$. (Orbits give the well-known Bruhat decomposition.)
\end{enumerate}

\subsection{Nicely decomposable foliations. Height of a nice decomposition}  Let $(M,\cF)$ be a singular foliation. Its holonomy groupoid may be very singular. On the other hand, this singularity gives rise to open subsets which are  \emph{saturated} for $\cF$ (\ie  union of leaves of $\cF$). We thus obtain ideals of $C^*(M,\cF)$ that we may use to compute the $K$-theory.

For instance, recall that the source fibers of the holonomy groupoid of the foliation as defined in \cite{AS1} were shown by Debord to be smooth manifolds in \cite{Debord2013}. On the other hand, the dimension of these manifolds varies. Let us denote by $\ell_0<\ell_1<\ldots<\ell_k$ the various dimensions occurring (note that $k$ may be infinite as shown in \cite{AZ1}). Let $\Omega_j$ denote the set of points with source fiber dimension $\le \ell_j$. We find an ascending sequence  $\Omega_0 \subseteq \Omega_{1} \subseteq \ldots \subseteq \Omega_{k-1} \subseteq \Omega_k = M$ of saturated open subsets of $M$. This decomposition yields a sequence of two sided ideals $J_j=C^*(\Omega_j,\cF_{|\Omega_j})$ of $C^{\ast}(M,\cF)$.   The quotient $C^*$-algebra $J_j/J_{j-1}$ is the $C^*$-algebra of the restriction of the holonomy groupoid $H(\cF)$ to the locally closed saturated set $Y_j=\Omega_j\setminus \Omega_{j-1}$. The module $\cF$, when restricted to $Y_j$ is finitely generated and projective, and the restriction of $H(\cF)$ to $Y_j$ is a Lie groupoid (when $Y_j$ is a submanifold) so that we may expect a Baum-Connes map for it.

Our computation of the $K$-theory requires an ingredient, which we add as an extra assumption (it is satisfied in the above examples). Let $(M,\cF)$ a singular foliation. We say that $(M,\cF)$ is \emph{nicely decomposable with height $k$} if there is a cover of $M$ by open subsets $(W_j)_{j\in \N,\ j\le k}$, such that for every $j\in \N$, $j\le k$ the restriction of the foliation $\cF$ to each $W_j$ is defined by a Hausdorff Lie groupoid $\cG_j$, the open subset $\Omega_j=\bigcup _{i\le j}W_i$ is saturated and  $\cG_j$ coincides with the holonomy groupoid $H(\cF)$ on the (locally closed) set $Y_j=\Omega_j\setminus \Omega_{j-1}$ (we set $Y_0=W_0=\Omega_0$).  Moreover, we assume that the groupoids $\cG_j$ are linked via morphisms which are submersions $\cG_j \mid_{\Omega_{j-1}\cap W_j} \to \cG_{j-1}$.

If $(M,\cF)$ is nicely decomposable, the quotients $J_j/J_{j-1}$ are given by (restriction to closed sets of) Lie groupoids, for which a Baum-Connes conjecture does exist. This makes the calculation of the $K$-theory of $C^*(M,\cF)$ possible, at least in terms of a spectral sequence.

Singularity height $0$ corresponds to foliations whose holonomy groupoid is a Lie groupoid. Examples (a), (b), (c) are all of singularity height $1$. We will use the decomposition given by the dimensions of the fibers. In examples (a) and (b), the dimensions of the fibers are $\ell_0=2$ and $\ell_1=3$; in example (c), these dimensions are $\ell_0=0$ and $\ell_1=1$. Examples (d), (e) and (f) have higher singularity height.

\subsection{The left hand side}\label{sec:introlhs}

We construct the left hand side in two steps. 
\begin{itemize}
\item The first one consists of replacing the holonomy groupoid $H(\cF)$ by a slightly more regular one $G$ whose (full) $C^*$-algebra is $E$-equivalent to the foliation one. This groupoid is constructed via a mapping cone construction in the height one case and via a telescope construction in the higher singularity case.
\item In the second step we construct a left hand side for the ``telescopic groupoid'' $G$ which is the $K$-theory of a proper $G$-algebra in a generalized sense, together with a Dirac type construction.
\end{itemize}

\subsubsection{A telescopic construction}

\paragraph{A mapping cone construction in the height 1 case.}
Let us explain our strategy more explicitly in the case of a foliation admitting a singularity height 1 decomposition: In this case, we obtain a diagram of \emph{full} $C^{\ast}$-algebras (with $\cG=\cG_1$):
\[
\xymatrix{
0\ar[r] & C^{\ast}(\cG_{\Omega_0}) \ar[r]^{i_{\cG}} \ar[d]_{\pi_{\Omega_0}} & C^{\ast}(\cG) \ar[r] \ar[d]_{\pi} & C^{\ast}(\cG_{Y_1})  \ar@{=}[d] \ar[r]&0 \\
0\ar[r] &C^{\ast}(\Omega_0,\cF_{|\Omega_0}) \ar[r]^{i_{\cF}} & C^{\ast}(M,\cF)  \ar[r] & C^{\ast}(M,\cF)|_{Y_1} \ar[r]&0
}
\]
The singularity height one assumption means that the holonomy groupoid of the restriction $\cF_{|\Omega_0}$ of the foliation $\cF$ to $\Omega_0$ is a Lie groupoid $\cG_0$ and $C^{\ast}(\Omega_0,\cF_{|\Omega_0})=C^*(\cG_0)$. The lines of this diagram  are exact at  the level of \emph{full} $C^*$-algebras.

 Since $\cG$ defines $\cF$, it is an atlas in the sense of \cite{AS1}, so $H(\cF)$ is a quotient of $\cG$. Whence the two extensions are connected by the map $\pi$ and its restriction $\pi_{\Omega_0}$, which is integration along the fibers of this quotient map $\cG\to H(\cF)$. From this diagram, we conclude that the algebra $C^*(M,\cF)$ is equivalent in $E$-theory (up to a shift of degree) with the mapping cone of the morphism $$(i_{\cG},\pi_{\Omega_0}):C^{\ast}(\cG_{\Omega_0})\to C^{\ast}(\cG)\times C^{\ast}(\Omega_0,\cF_{|\Omega_0}).$$
 
\paragraph{Foliations of height $\ge2$.} As far as singular foliations with nice decompositions of arbitrary (bounded or not) singularity height is concerned, we show that the strategy developed for the singularity height-1 case can be generalized. In particular, $C^{\ast}(M,\cF)$ is $E$-equivalent to a ``telescopic'' $C^{\ast}$-algebra whose components are Lie groupoids. In fact we see that these telescopes can just be treated as mapping cones. 

\bigskip Now let us see how the above apparatus can be used to formulate the Baum-Connes assembly map for singular foliations. It suffices to explain the idea for the height $1$ case.

\paragraph{Longitudinally smooth groupoids.} The above mapping cone and the telescopic algebra constructed here, are based on morphisms of Lie groupoids which are smooth submersions and open inclusions at the level of objects. These $C^*$-algebras are immediately seen to be the $C^*$-algebras of a kind of groupoids which generalize both Lie groupoids and singular foliation groupoids: \emph{longitudinally smooth groupoids.}

\subsubsection{A left hand side for the telescopic groupoid}

\paragraph{Setting of the problem.} 
Before we outline our construction of the left hand side, let us make a remark.  Recall that in \cite{Tu}, Jean-Louis Tu defined a left hand side and a Baum-Connes morphism for Lie groupoid $C^*$-algebras of the form $K_*^{top}(\cG)\to K_*(C^*(\cG))$. In order to construct a left hand side for this mapping cone, we need to find a ``left hand side for the morphism'' $(i_{\cG},\pi_{\Omega_0})$. We will in fact not only need it as a morphism at the level of groups $K_*^{top}$, but we really need to construct it as a $KK$-element.

The difficulty lies with the understanding of the left-hand side of the mapping cone of the surjective homomorphism $\pi_{\Omega_0}:C^{\ast}(\cG_{\Omega_0}) \to C^{\ast}(\Omega_0,\cF_{|\Omega_0})$. We treat this by deploying the Baum-Douglas formulation given in \cite{BaumConnes}, \cite{BaumDouglas1}, \cite{BaumDouglas2}. At this point we will need further assumptions on the groupoids $\cG$ and $\cG_1$, namely that their classifying spaces of proper actions are smooth manifolds, to make sure that the Baum-Connes morphisms are naturally given by $KK$-elements. (In appendix \ref{app:Enotmanifold} we show how this assumption can be weakened.)

\paragraph{Actions of the telescopic groupoid.} In order to define the left hand side for the telescopic groupoid, we follow the Lie groupoid case:
\begin{itemize}
\item For every longitudinally smooth groupoid $G$, one  defines $G$-algebras very much in the spirit of \cite{AS1}: algebraic conditions are stated at the level of the groupoid; topological ones at the level of bisubmersions which can be thought of as ``smooth local covers'' of $G$ (\cf \cite{AS1}). We define the (full and reduced) crossed product for every $G$-algebra. 
\item One may define a notion of \emph{proper $G$-algebra}: a $G$-algebra is said to be proper if its restriction to the groupoids corresponding to the various strata is proper in the usual sense. In particular, one may define actions on spaces and proper actions on spaces.
\item We define Le Gall's equivariant $KK$-theory (\cite{LeGall}) in the context of longitudinally smooth $G$-algebras, despite the topological pathology of the holonomy groupoid $G$. We extend the equivariant Kasparov product to this case.
\item We may then construct the left hand side and the assembly map for the telescopic algebras of a nice decomposition of a singular foliation. To that end we still need to assume for $(M,\cF)$ that the Lie groupoids of its decomposition admit smooth manifolds as classifying spaces for proper actions.
\item Actually this point of view, allows to construct a Baum-Connes map \emph{with coefficients} for every $G$ algebra. It is easily seen that, in the case of nicely decomposable foliations, our Baum-Connes map with coefficients in proper spaces or algebras is an isomorphism.
\end{itemize}

\paragraph{The main result.}
We show then that in cases as above the Baum-Connes map can be constructed canonically. Namely, we prove the following:
\begin{thm}\label{thm:SingBC}
\begin{enumerate}
\item If $(M,\cF)$ admits a nice decomposition by Lie groupoids whose classifying space for proper action is a manifold, then there is a well defined left hand side and one may construct a Baum-Connes assembly map.
\item If moreover the groupoids of the nice decomposition are  amenable and Hausdorff, then the Baum-Connes map is an isomorphism.
\end{enumerate}
\end{thm}

Note that examples (a) and (c) above are amenable; although example (b) is not, it is `strongly $K$-amenable' and the Baum-Connes conjecture (for the \emph{full} version) holds for it.

Note also that example (c) is not exactly covered by our theorem since the groupoid $\cG_0$ is not assumed to be Hausdorff. However, the Baum-Connes conjecture holds also in this case

For the examples of larger singularity height, described in examples (d), (e) and (f) note that: As the minimal parabolic subgroup of $GL(n,\R)$ is amenable, Theorem \ref{thm:SingBC} implies that the Baum-Connes conjecture holds.

Let us point out that our constructions of the equivariant $KK$-theory  could in a way be bypassed, but may have its own interest. 
In particular, we give a simple quite general formulation and proof for the existence of the Kasparov product, which applies in all known equivariant contexts (groups, group actions \cite{Kasparov1988}, groupoids \cite{LeGall}, Hopf algebras \cite{BaajSkandalis}).

\paragraph{Trying to weaken our assumptions.}
The assumption on the classifying spaces is quite natural. All the groupoids given by Lie group actions admit manifolds as classifying spaces for proper actions - and this assumption is stable by Morita equivalence. In this way it is satisfied by all  the (Hausdorff) groupoids that appear in the examples that we discuss in this work. Nevertheless, it is quite tempting to try to get rid of it. 
In appendix  \ref{app:Enotmanifold} we explain how it can be replaced by a quite weaker rather technical one: assumption  \ref{ass:homotopy}, which could be true in general.

\paragraph{Structure of the paper.}\begin{itemize}
\item In section \ref{sec:prel} we introduce the notion of singularity height for a singular foliation and define nicely decomposable foliations. We also explain the examples mentioned in the beginning of this introduction. 
\item Section \ref{sec:convalgheightone} focuses on nicely decomposable foliations with singularity height one. We give the construction of the associated mapping cone $C^{\ast}$-algebra and prove that it is $E$-equivalent to the foliation $C^{\ast}$-algebra. 
\item In section \ref{sec:higherheight} we extend this construction and result to foliations of arbitrary singularity height, replacing mapping cones with telescopes. 
\item Section \ref{sec:longsmoothKK} defines longitudinally smooth groupoids, their actions and constructs the associated $KK$-theory. 
\item The crucial section is \S \ref{sec:BCtelescope}, where we formulate the Baum-Connes conjecture (left-hand side and Baum-Connes map) for the telescopic algebra, assuming the classifying spaces of proper actions of the groupoids associated with the nice decomposition of $(M,\cF)$ are smooth manifolds. The proof of Theorem \ref{thm:SingBC} can be found there.
\item  The rest of the paper gives the explicit calculation of the $K$-theory for examples (a), (b) and (c). Specifically, section \ref{sec:computelinearactions} is concerned with the linear actions in examples (a) and (b). 
\item Flows of vector fields (example (c)) are treated in section \ref{sec:X}; in particular, we treat separately vector fields with periodic points. 
\item Last, in appendix \ref{appendix} we explain how to remove the assumption that the classifying spaces of proper actions are smooth manifolds.
\end{itemize}

\begin{notation}
Let $(M,\cF)$ be a foliation. We will denote the  (minimal  \ie the groupoid associated with the path holonomy atlas - \cf \cite{AS1}) holonomy groupoid by $H(\cF)$ (or $H(M,\cF)$ when needed). We will denote by $C^*(M,\cF)$ and $C_{red}^*(M,\cF)$ its \emph{full} and \emph{reduced} $C^*$-algebras.

We will mainly use the \emph{full} $C^*$-algebra. This is justified by the two following reasons:
\begin{itemize}
\item Constructing a Baum-Connes map for the full foliation algebra automatically gives the one for the reduced version. Recall that the Baum-Connes map, in the regular case, factors through the full version of the foliation algebra.
\item All our constructions are based on sequences of groupoid $C^*$-algebras, which are always exact at the full $C^*$-algebra level, and may fail to be exact at the reduced level (see \S \ref{sec:exactness}).
\end{itemize}
\end{notation}

\noindent{\bf Acknowledgements:}
We would like to thank Nigel Higson and Thomas Schick for various discussions and suggestions.

\section{Nicely decomposable foliations}\label{sec:prel}

\subsection{Notations and remarks}

Let $M$ be a smooth manifold and $\X_c(M)$ the $C^{\infty}(M)$-module of compactly supported vector fields. In \cite{AS1},  we defined a singular foliation on $M$ to be a $C^{\infty}(M)$-submodule $\cF$ of $\X_c(M)$ which is locally finitely generated and satisfies $[\cF,\cF] \subseteq \cF$.

Given a point $x \in M$ let $I_x = \{f \in C^{\infty}(M) : f(x)=0\}$ and recall from \cite{AS1} the fiber $\cF_{x} = \cF/I_x\cF$. The map $M \ni x \mapsto  \dim(\cF_x)$ is upper semicontinuous (see \cite[Prop. 1.5]{AS1}).

When this dimension is constant (continuous if $M$ is not assumed to be connected), \ie when the module $\cF$ is projective, the foliation is said to be \emph{almost regular} and the holonomy groupoid $H(\cF)$ was proved to be a Lie groupoid by Debord in \cite{DebordJDG}. 

In the present paper, we will deal with cases where the dimension of $\cF_x$ is not constant. The number of possible dimensions measures the singularity of the foliation. We will give a definition of this \emph{singularity height}, more appropriate for our purposes in definition \ref{dfn:goodgpd}.

By semicontinuity, the subsets $O_\ell=\{x\in M;\ \dim(\cF_x)\le \ell\}$ are open. They are \emph{saturated}, \ie unions of leaves of $\cF$.

We will deal with restrictions of the foliation to open sets. We will use the following remark:

\begin{remark}
Let $(M,\cF)$ be a foliation. Let $V$ be an open subset of $M$.\begin{itemize}
\item The holonomy groupoid of the restriction $\cF_{|V}$ to $V$ is the $s$-connected component of the restriction $H(\cF)_V^V=\{z\in H(\cF);\ t(z)\in V,\ \hbox{and}\ s(z)\in V\}$ to $V$.
\item If $V$ is saturated, then $H(\cF_{|V})=H(\cF)_V^V$.
\end{itemize}
Actually an analogous statement holds for the pull back foliation $f^{-1}(\cF)$ by a smooth map $f:V\to M$ \emph{transverse to $\cF$} (\cf \cite[\S 1.2.3]{AS1}):  $H(f^{-1}(\cF))$ is the $s$-connected component of $H(\cF)_f^f=\{(v,z,w)\in V\times H(\cF)\times V;\ t(z)=f(v),\ \hbox{and}\ s(z)=f(w)\}$. If moreover $f$ is a submersion whose image is saturated with connected fibers, then $H(f^{-1}(\cF))=H(\cF)_f^f$.
\end{remark}

Now let us discuss the notation for $C^*$-algebras we will be using in this sequel as far as restrictions are concerned. If $\cG$ is a locally compact groupoid (with Haar measure) and $Y$ is a locally closed saturated subset of $\cG_0$, then $\cG_Y=\{x\in \cG;\ s(x)\in Y\}$ is also a locally closed groupoid and we can define its $C^*$-algebra. We put $C^*(\cG)_{|Y}=C^*(\cG_Y)$. The same construction for foliation algebras will be useful in our context:

\begin{notation} Let $(M,\cF)$ be a (singular) foliation.
\begin{enumerate}
\item Let $\Omega \subset M$ be a \emph{saturated} open subset. Then $$C^{\ast}(M,\cF)|_{\Omega} : = C_0(\Omega)C^{\ast}(M,\cF)=C^*(\Omega,\cF_{|\Omega})$$ is the foliation $C^*$-algebra of the restriction of $\cF$ to $\Omega$. The same holds for the reduced $C^*$-algebras.
\item If $Y \subset M$ is a saturated closed subset  then the \emph{full} $C^{\ast}(M,\cF)|_{Y}$ is the quotient of $C^{\ast}(M,\cF)$ by $C^{\ast}(M,\cF)|_{M\setminus Y}$.\\ Note that the natural definition for the reduced one is to take the quotient of $C^{\ast}(M,\cF)$ corresponding to the regular representations at points of $Y$, \ie representation on $L^2(H(M,\cF)_y)$ for $y\in Y$.
\item If $Y \subset M$ is a saturated locally closed subset then $Y$ is open in its closure $\overline{Y}$ and the closed subset $\overline{Y}\setminus Y$ is saturated. Let $U=M\setminus(\overline{Y}\setminus Y)$. We denote $C^{\ast}(M,\cF)|_Y$ the quotient of $C_0(U)C^{\ast}(M,\cF)$ by $C^{\ast}(M,\cF)|_{M\setminus \overline{Y}}$. In other words, $C^*(M,\cF)_{|Y}=(C^*(M,\cF)_{|U})_{|Y}$.
\end{enumerate}
\end{notation}

\subsection{Foliations associated with Lie groupoids}

In this sequel we will consider foliations defined from Lie groupoids (at least locally - \cf \S\ref{sec:goodgpd}).

Let us make a few observations regarding singular foliations defined by Lie groupoids. 

Every Lie algebroid $A$ with base $M$ and therefore every Lie groupoid $(t,s):\cG\rra M$ defines a foliation: Indeed, the anchor map $\sharp : A \to TM$ is a morphism of Lie algebroids, whence $\sharp(\Gamma_c A) \subset \X_c(M)$ is a singular foliation.

Let $\cG$ be  a (locally Hausdorff) Lie groupoid over a manifold $M$ and $\cF$ the associated foliation. Up to replacing $\cG$ by its $s$-connected component (which is an open subgroupoid of $\cG$ with the same algebroid - and thus defines the same foliation on $M$) we may assume that $\cG$ is $s$-connected, \ie the fibers of the source map $s:\cG\to M$ are connected.
Then the groupoid $\cG$ is an atlas for our foliation (in the sense of \cite[Def. 3.1]{AS1}). As $\cG$ is assumed $s$-connected, it defines the path holonomy atlas (\cite[example 3.4.3]{AS1}). The holonomy groupoid $H(M,\cF)$ is a quotient of $\cG$ by the equivalence relation defined in \cite[Prop. 3.4.2]{AS1}.

In order to compute this quotient, we will use a Lemma from \cite{AZ1}. 

Let $\gamma\in \cG$ and write $x=s(\gamma)$. Note that if $q(\gamma)$ is a unit, then $t(\gamma)=x$. Choosing a bi-section through $\gamma$ we obtain a local diffeomorphism $g$ of $M$ which acts on the tangent bundle $T_{x}M$ and fixes the tangent to the leaf $F_{x}$. It therefore acts on $N_{x}=T_xM/F_x$. This action only depends on $\gamma$. Denote it by $\nu(\gamma)\in GL(N_x)$.

Now, it was shown in \cite{AZ1} that there is an action of $H(\cF)$ on this ``bundle'' of normal spaces. As an immediate consequence, we find:

\begin{lemma}\label{lem:holgpd}
If $q(\gamma)$ is a unit, then $\nu(\gamma)=\id_{N_x}$.\hfill$\square$
\end{lemma}

\subsection{Nicely decomposable foliations}\label{sec:goodgpd}

We now present the constraints that we will put on our foliations: We say that the foliation is \emph{nicely decomposable} if it admits a nice decomposition in the following sense.

\begin{definition}\label{dfn:goodgpd}
Let $(M,\cF)$ be a singular foliation and let $k\in \N\cup \{+\infty\}$. A \emph{nice decomposition} of $(M,\cF)$ of \emph{singularity height $k$} is given by
\begin{enumerate}
\item a sequence $(W_j)_{0\le j< k+1}$ of open sets of $M$ such that the open set $\Omega_j=\bigcup_{\ell\le j}W_\ell$ is saturated and $\bigcup_{j< k+1}W_j=M$ (with the convention $+\infty+1=+\infty$).\\
Put $Y_0=\Omega_0$ and, for $j\ge 1$,  $Y_j=\Omega_j\setminus \Omega_{j-1}$. 
\item a sequence of Lie groupoids $\cG_j \gpd W_j$ defining the restriction of $\cF$ to $W_j$, and such that $\cG_j |_{Y_j} = H(\cF) |_{Y_j}$;
\item morphisms of Lie groupoids $q_j : \cG_{j}|_{\Omega_{j-1}\cap W_j} \to \cG_{j-1}$  (for $j>0$) which are submersions, and which  at the level of objects, are just the inclusion $\Omega_{j-1}\cap W_j\to W_{j-1}$.
\end{enumerate}
\end{definition}

\begin{remarks}\label{rem:nicedecom}
\begin{enumerate}
\item  If $(M,\cF)$ is an almost regular foliation then $H(\cF)$ is a Lie groupoid as shown in \cite{AS1} (it coincides with the one constructed by Debord in \cite{DebordJDG}). In our current context, the decomposition sequence of such a foliation has singularity height zero; its realization is $H(\cF)$ itself. We will not be concerned with situations as such in this sequel. Truly singular examples of nicely decomposable singular foliations arise when the singularity height of the decomposition is 1 or larger. 

\item  By definition $W_0=\Omega_0$ and the restriction of $H(\cF)$ coincides with $\cG_0$. It follows that the restriction of $\cF$ to $\Omega_0$ is almost regular, which means that $\Omega_0$ is contained in the (open) set of points where $\dim \cF$ is continuous, \ie has a local minimum.

\item In all our examples, $W_j=\Omega_j$ and $\Omega_j$ may be constructed using the dimension of the fibers:

For $\ell\in \N$, put $$O_\ell=\{x\in M,\ \dim (\cF_x)\le \ell\}$$ 

Denote by $\ell_0<\ell_1<...<\ell_j$ for $j<k+1$ the various possible dimensions. For $j=0,1,\ldots,k$ put  $\Omega_j= O_{\ell_j}$.

Note that in \cite{AZ1} is given an example of a foliation where this $k$ is infinite.
\end{enumerate}
\end{remarks}

\subsection{Examples of nicely decomposable foliations}

We now give a few examples of nice decompositions of foliations.

\subsubsection{Examples of height one}\label{subsubsec:height1}

\begin{remark}\label{rem:length1}
 In the case of height one, we have $W_0=\Omega_0$ and the restriction  and $\cG_0$ is the holonomy groupoid of the restriction of $\cF$ to $\Omega_0$. We therefore just need to specify the set $\Omega_0$ and the Lie groupoid $\cG_1\gpd W_1$ defining the foliation $\cF$ on an open subset $W_1$ containing the complement $Y_1=M\setminus \Omega_0$ of $\Omega_0$ and such that the restriction of $\cG$ to $Y_1$ coincides with that of $H(\cF)$. 
 
 Actually, in our examples $W_1=M$.
\end{remark}

\begin{exs}\label{exs:basic}
We will give here examples of singularity height $1$ associated with Lie group actions. Some examples of larger singularity height are computed in \cite{AH}. In the sequel of this paper, we will calculate the associated $K$-theory explicitly for the following examples.
\begin{enumerate}
\item \label{examp:b}Let $M = \R^3$ and consider the foliation $\cF$ defined by the image of the (infinitesimal) action of $SO(3)$ on $\R^3$ by rotations. The leaves are concentric spheres in $\R^3$ with one singularity at $\{0\}$. Let $\cG$ be the action groupoid $\R^3 \rtimes SO(3) \gpd \R^3$. Since $SO(3)$ is simple, the restriction of $H(\cF)$ to $0$, which is a quotient of $SO(3)$ has to be $SO(3)$ (we may also use Lemma \ref{lem:holgpd} to prove this result). The restriction of $\cF$ to $\R^{3} \setminus \{0\}$ is really a regular foliation - and in fact the fibration $S^2 \times \R_{+}^{\ast} \to \R_{+}^{\ast}$, whence the holonomy groupoid of $\cF$ is $$H(\cF)=(S^2 \times S^2 \times \R_{+}^{\ast})\bigcup \{0\}\times SO(3).$$  It follows that the foliation has a nice decomposition of singularity height 1, namely $W_1 = \R^3$, $\cG_1=\cG$ and $\Omega_0 = \R^3 \setminus \{0\}$.

\item \label{examp:c}Let $M=\R^2$ and consider the action of $SL(2,\R)$. It has two leaves, namely $\{0\}$ and $\R^2 \setminus \{0\}$. Using again Lemma  \ref{lem:holgpd}, the associated holonomy groupoid is seen to be $$H(\cF)=(\R^2\setminus\{0\} \times \R^2\setminus\{0\}) \bigcup \{0\} \times SL(2,\R)$$  Considering the action groupoid $\cG = \R^2 \rtimes SL(2,\R)$ we obtain the singularity height 1 nice decomposition $\Omega_1 = \R^2$, $\cG_1=\R^2\rtimes SL_2(\R)$ and $\Omega_0 = \R^2 \setminus \{0\}$, $\cG_0=\Omega_0\times \Omega_0$.

\item There are many singular foliations of singularity height 1 arising from group actions which have nice decompositions. For instance, take $n\ge4$ instead of $3$ in  example (\ref{examp:b}) or $n\ge3$ instead of $2$ in  example (\ref{examp:c}). 

We may also consider the action of $GL(2,\R)$ on $\R^2$. The associated holonomy groupoid is $$H(\cF)=(\R^2\setminus\{0\} \times \R^2\setminus\{0\}) \bigcup \{0\} \times GL^+(2,\R)$$ where $GL^+(2,\R)$ denotes $2\times 2$ matrices with positive determinant. Considering the action groupoid $\cG = \R^2 \rtimes GL^+(2,\R)$ we obtain $\Omega_1 = \R^2$ and $\Omega_0 = \R^2 \setminus \{0\}$.  We can of course replace $2$ by $n$ also in this situation.

Another example as such comes from the action of $SL(n,\C)$ on $\C^n$: Its holonomy groupoid is $$H(\cF)=(\C^n\setminus\{0\} \times \C^n\setminus\{0\}) \bigcup \{0\} \times SL(n,\C)$$ Considering the action groupoid $\cG = \C^n \rtimes SL(n,\C)$ we have $\Omega_1 = \C^n$, $\Omega_0 = \C^n \setminus \{0\}$.

\item \label{example:actionofR}We end with an example of quite different flavor. 

Let $M$ be a manifold endowed with a smooth action $\alpha$ of $\R$. Let $\cG_1=M\rtimes_\alpha \R$ be the associated action groupoid, and $\cF$ the associated foliation. 

Denote by $Fix (\alpha)$ the set of fixed points of $\alpha$, $W=\text{Int}{(Fix (\alpha))}$ its interior and $V=M\setminus Fix (\alpha)$ its complement. Let $x\in M$.\begin{itemize}
\item If $x\in W$, then $\cF_x=0$;
\item For $x\in V$, the dimension of $\cF_x$ is $1$. By semi-continuity, $\dim \cF_x=1$ for $x\in \overline V$.
\end{itemize}
Let $\Omega_0$ be the set of continuity points of $\dim \cF$. Its complement $Y_1$ is the boundary $\partial W$ of $W$. The restriction of $\cF$ to the open set $\Omega_0$ is almost regular. 
 
We show that the morphism $M\rtimes _\alpha\R\to H(\cF)$ is injective over $Y_1$. We thus have a nice decomposition $H(\Omega_0,\cF_{|\Omega_0})\gpd \Omega_0$,  and $M\rtimes_\alpha  \R\gpd M$.

This is done using classical facts based on the period bounding lemma (\cf \cite{AbrahamRobbin}) that we recall here:
\begin{lemma}[Period bounding]\label{lem:periodbound}
Let $X$ be a compactly supported $C^r$-vector field on a $C^r$-manifold $M$ with $r \geq 2$. There is a real number $\eta > 0$ such that, for any $x \in M$, either $X(x)=0$ or the prime period $\tau_x$ of the integral curve of $X$ passing through $x$ is $\tau_x > \eta$.\hfill $\square$
\end{lemma}

Put $P=\{(x,u) \in M \times \R;\ \alpha_u(x)=x\}$. It is  obviously a closed subset of $M \times \R$ and the restrictions of the source and target maps to $P$ coincide. By definition of the holonomy groupoid, an element $(x,u)\in \cG_1=M\rtimes \R$ is a trivial element in $H(\cF)$ if and only if there is an identity bisection through it \ie if there exists an open neighborhood $U$ of $x$ and a smooth function $f:U\to \R$ such that $f(x)=u$ and $(z,f(z))\in P$ for all $z\in U$.

Let $Per(\alpha)$ be the set of \emph{stably periodic points}, \ie the set of $x\in M$ such that there exists an open neighborhood $U$ of $x$ and a smooth function $f:U\to \R^*$ such that $(y,f(y))\in P$ for all $y\in U$. It is the set of $x\in M$ such that $\{(x,u);\ u\in \R\}\to H(\cF)$ is not injective.

Obviously $W\subseteq Per(\alpha)$. 

\begin{prop}\label{prop:nicedecompoactionofR}
The set  $Y_1\cap Per(\alpha)$ is empty.
\end{prop}
\begin{proof}
Let $x \in \overline{W}\cap Per(\alpha)$. We need to show that $x\not\in Y_1$, \ie that $x\in W$.

Up to changing $X$ far from $x$, we may assume that $X$ has compact support.

Since $x \in \overline{W}$, it follows that $X$ vanishes as well as all its derivatives at $x$. We may then write $X=qY$ where $q$ is a smooth nonnegative function such that $q(x)=0$ and $Y$ is a smooth vector field with compact support (take for instance $q$ to be a smooth function which coincides near $x$ to the square of the distance to $x$ for some riemannian metric). Let then $U$ be an open relatively compact neighborhood of $x$ and  $f:U\to \R^*$ a smooth bounded function such that $(y,f(y))\in P$ for all $y\in U$. It follows that all the points in $U$ are periodic for $X$ and therefore for $Y$. When $y\to x$, $f(y)\to f(x)$ whence the $Y$ period of $y$ tends to $0$. By the period bounding lemma, it follows that any $y$ close enough to $x$ satisfies $Y(y)=0$, whence $x\in W$.
\end{proof}

It follows that $(H(\Omega_0)\gpd \Omega_0,M\rtimes_{\alpha}\R\gpd M)$ is a nice decomposition for $\cF$.

It is worth noticing that the holonomy groupoid $\cG_0=H(\Omega_0,\cF_{|\Omega_0})$ is a disjoint union of clopen subgroupoids $W\coprod H(V',\cF_{|V'})$ where $V'$ is the interior of $\overline V$, and that its $C^*$-algebra $C^*(\Omega_0,\cF_{|\Omega_0})$ is a direct sum $C_0(W)\oplus C^*(V',\cF_{|V'})$. 

Remark that, in presence of periodic points, the groupoid $H(V,\cF_{|V})$ and therefore $\cG_0$ need not be Hausdorff.
\end{enumerate}
\end{exs}

\subsubsection{An example of larger singularity height}\label{sec:exlengthk}

We start by giving a natural family of examples of nicely decomposable foliations with singularity height larger than $1$. Some of them will be studied in \cite{AH}. 

If a subgroup $G\subset GL_n(\R)$ has more than two orbits in its action on $\R^n$, then the transformation groupoid $\R^n\rtimes G$ may give rise to interesting nicely decomposable  foliations of singularity height $\ge 2$.

A typical example is given by a parabolic subgroup of $GL(n,\bK)$, where $\bK=\R$ or $\C$: given a flag $\{0\}=E_k\subset E_{k-1}\subset E_{k-2}\subset \ldots \subset E_1\subset E_0=\bK^n$ (with $k\le n$ and $E_k$ pairwise different), let $G$ be the group of (positive if $\bK=\R$) automorphisms of this flag, \ie $G$ is the subgroup of $GL(n,\bK)$ of elements fixing the spaces $E_k$; if $\bK=\R$ we further impose that their restriction to $E_j$ has positive determinant (in order to fulfill connectedness). 


For $0\le j\le k$, let $\Omega_j=\bK^n\setminus E_{j+1}$ and $Y_j=E_j\setminus E_{j+1}$ (with the convention $E_{k+1}=\emptyset$). The set $Y_j$ consists of one or two $G$ orbits (depending on whether $\dim E_j\ge 2+\dim E_{j+1}$ or $\dim E_j=1+\dim E_{j+1}$ - in the complex case the $Y_j$ consists of a single orbit).

For every $j\in\{0,\ldots,k\}$, let $F_j$ be the quotient space $F_j=\bK^n/E_j$ endowed with the flag $\{0\}\subset E_{j-1}/E_j,\ldots \subset E_{0}/E_j$ and let $G_j$ be the group of positive automorphisms of this flag. The quotient map $\bK^n\to F_j$ induces a group homomorphism $q_j:G\to G_j$.

Let also $p_j:\Omega_j\to F_j$ be the restriction of the quotient map to $\Omega_j$. Let then $\widetilde{\cG_j}$ be the pull back groupoid of $F_j\rtimes G_j$ by the map $p_j$. In other words $$\widetilde{\cG_j}=\{(x,g,y)\in \Omega_j\times G_j\times \Omega_j;\ p_j(x)=gp_j(y)\}$$

The map $(x,g,y)\mapsto (x,q_j(g),y)$ is a submersion and a groupoid morphism from $\Omega_j\rtimes G=\{(x,g,y)\in  \Omega_j\times G\times \Omega_j;\ x=gy\}$ into $\widetilde{\cG_j}$. Its image is the $s$-connected component $\cG_j$ of $\widetilde{\cG_j}$.

It follows from  the following obvious Lemma that $\cG_j\rra \bK^n$ is a bisubmersion. 

\begin{lemma}\label{lem:bisub}
Let $M,U,V$ be manifolds, $(M,\cF)$ a foliation, $p:U\to V$ a surjective submersion and $t_V,s_V:V\rightrightarrows M$ two submersions. Then $(U,t_V\circ p,s_V\circ p)$ is a bi-submersion for $\cF$ if and only if $(V,t_V,s_V)$ is a bi-submersion for $\cF$. \hfill$\square$
\end{lemma}
It follows then from lemma \ref{lem:holgpd} that $H(\cF)|_{Y_j}=(\cG_j)|_{Y_j}$. We deduce:

\begin{prop}
The foliation of $\bK^n$ by the action of $G$ is nicely decomposed by the groupoids $\cG_j\rra \Omega_j$. Its holonomy groupoid is a union $\coprod_{j=0}^k(\cG_j)|_{Y_j}$.
\end{prop}

\begin{remarks}\label{rem:Bruhat}
 \begin{enumerate}
\item \label{item:proj}One may write a projective analogue of this example: let $PG$ be the projective analogue of $G$ acting on $\bK P^{n-1}$, namely $PG$ is the quotient of $G$ by its center: the group of similarities in $G$. It has $k$ orbits: the images $Y_j=PE_j\setminus PE_{j+1}$ of $E_j\setminus E_{j+1}$ by the quotient map $p:\bK^n\setminus\{0\}\to \bK P^{n-1}$ (for $j>0$). This foliation is nicely decomposed by the projective analogues $P\cG_j$ of the $\cG_j$. Note that the map $p:E_j\setminus \{0\}\to p(E_j)$ induces  a morphism $p_j:\cG_j\to P\cG_j$ which is a Morita equivalence in the complex case. In the real case, it is almost a Morita equivalence: the morphism $p_j$ induces an isomorphism  of the stabilizer of $x\in E_j\setminus \{0\}$ in $\cG_j$ with the stabilizer of $p(x)\in PE_j$ in $P\cG_j$, but for $0<i\le j$ and $\dim(E_i)=\dim (E_{k+1})+1$, the set $E_i\setminus E_{i-1}$ consists of two orbits of the groupoid $\cG_j$ which become equivalent in $P\cG_j$.  
The corresponding foliation $C^*$-algebra is (almost) Morita equivalent to $C^*(\Omega_1,\cF)$

\item \label{item:Bruhat} There are many other interesting examples of the same flavor. A typical one is given in the following way: let $P_1,P_2\subset GL(n,\bK)$ be two parabolic subgroups, and let $P_1\times P_2$ act on $GL(n,\bK)$ by left and right multiplication. If $P_1=P_2$ is the minimal parabolic subgroup - consisting of upper triangular matrices, the orbits of this action are labeled by the symmetric group $\mathfrak{S}_n$ (Bruhat decomposition). In this example,  the decomposition to be taken into account is more complicated than just the dimension of the fibers. One may need to use the partial ordering of the orbits given by the inclusion of the closures.
\end{enumerate}

\end{remarks}

\section{Foliations with singularity height one}\label{sec:convalgheightone}


Let $(M,\cF)$ be a foliation admitting a nice decomposition of height one. In this section our purpose is to show that the full foliation $C^*$-algebra $C^*(M,\cF)$ can be replaced by a mapping cone of Lie groupoid $C^*$-algebras associated with a nice decomposition of $\cF$. We will generalize this construction to higher length in the next section, but before this, we will make some comments on the difficulties with dealing with reduced $C^*$-algebras.

\subsection{A mapping cone construction}\label{sec:MappingConeGeneral}

In the length one case, as noted in remark \ref{rem:length1}, we just need to specify the saturated open subset $\Omega=\Omega_0$ and the Lie groupoid $\cG=\cG_1\gpd W_1=W$ which defines the foliation on an open set $W$ containing $Y=M\setminus \Omega$ and whose restriction to $Y$ coincides with that of $H(\cF)$.

The open subset $\Omega$ gives rise (at the level of the \emph{full} $C^*$-algebras) to a short exact sequence  $$
\xymatrix{
 0\ar[r] &C^{\ast}(\Omega,\cF_{|\Omega}) \ar[r]^{\iota_{\cF}} & C^{\ast}(M,\cF)  \ar[r]^{\pi_{\cF}} & C^{\ast}(M,\cF)|_{Y} \ar[r]&0
}
$$
which, in principle will allow us to compute its $K$-theory. This will actually be the case in our examples (sections \ref{sec:computelinearactions} and \ref{sec:X}).

In order to only use Lie groupoids (note that $Y$ needs not be a manifold), and also to be able to extend our construction to a more general setting (\cf section \ref{sec:Higherlength}), we will also make use of the somewhat more elaborate diagram which appears in figure \ref{fig:general} below.

Restricting $\cG$ to the open subset $W\cap \Omega$ and $H(\cF)$ to the open subset $\Omega$, the integration along fibers (\cf \cite{AS1}) of the quotient map $\cG \to H(\cF)$ induces the following diagram of half-exact sequences of \emph{full} $C^{\ast}$-algebras:
\begin{figure}[H]
\[
\xymatrix{
0\ar[r] & C^{\ast}(\cG_{W\cap \Omega}) \ar[r]^{\iota_{\cG}} \ar[d]_{\pi_{\Omega}} & C^{\ast}(\cG) \ar[r]^{p_{\cG}} \ar[d]_{\pi_M} & C^{\ast}(\cG_{Y})  \ar@{=}[d] \ar[r]&0 \\
0\ar[r] &C^{\ast}(\Omega,\cF_{|\Omega}) \ar[r]^{\iota_{\cF}} & C^{\ast}(M,\cF)  \ar[r]^{p_{\cF}} & C^{\ast}(M,\cF)_{|Y} \ar[r]&0
}
\]
\caption{Exact sequences for a nicely decomposable foliation of singularity height 1.\label{fig:general}}
\end{figure}

Let $\cF$ be a nicely decomposable foliation of singularity height one. We may use the diagram in figure \ref{fig:general} in order to compute the $K$-theory of $C^*(M,\cF)$ via a Mayer-Vietoris exact sequence. 

We explain here how one may replace $C^*(M,\cF)$ by a mapping cone of Lie groupoid $C^*$-algebras. 

We will use the following notation: 
\begin{itemize}
\item For any $C^{\ast}$-algebra $Z$ and a locally compact space $X$ put $Z(X) = C_0(X;Z)$.

\item Recall that the mapping cone of a morphism $u:A\to B$ of $C^*$-algebras is $$\cC_u=\{(a,\phi)\in A\times B([0,1));\ \phi(0)=u(a)\}.$$
\end{itemize}

With the notation of the diagram in figure \ref{fig:general}, consider the morphism of $C^{\ast}$-algebras $$(i_{\cG},\pi_{\Omega}) : C^{\ast}(\cG_{W\cap \Omega}) \to C^{\ast}(\cG) \oplus C^{\ast}(W\cap \Omega,\cF_{|{\Omega}})$$ 

\begin{prop}\label{prop:Eequiv}
With the notation of figure \ref{fig:general}, the (full) foliation $C^*$-algebra $C^{\ast}(M,\cF)$ is canonically $E^1$-equivalent to the mapping cone $\cC_{(\iota_{\cG},\pi_{\Omega})}$.
\end{prop}
\begin{proof}
We show that given a diagram  of exact sequences of $C^{\ast}$-algebras and morphisms:
\[
\xymatrix{
0 \ar[r]  & I \ar[r]^{i}\ar[d]^{{\pi}} & B_1 \ar[r]\ar[d]  &Q\ar[r]\ar@{=}[d]& 0 \\
0 \ar[r] & B_0 \ar[r]^{i'}             & A \ar[r]  & Q\ar[r] &0 \\
 }
\]
the mapping cone $\cC_{(i,\pi)}$ of the map $(\pi,i) : I \to B_0 \oplus B_1$ is canonically $E^1$-equivalent to $A$.

Indeed we have canonical morphisms $\cC_i\to \cC_{i'}\to Q(0,1)$. Since both $\cC_i\to Q(0,1)$ and $\cC_{i'}\to Q(0,1)$ are onto with contractible kernels ($I[0,1)$ and $B_0[0,1)$ respectively), it follows that the morphism $\cC_i\to \cC_{i'}$ induces an equivalence in $E$ theory. Now, using the diagram:
\[
\xymatrix{
0 \ar[r]  & B_0(0,1) \ar[r] \ar@{=}[d]  & \cC_{(i,\pi)} \ar[r]\ar[d]  &\cC_i\ar[d]\ar[r]  & 0 \\
0 \ar[r] & B_0(0,1) \ar[r]    & \cC_{(i',\id_{B_0})}   \ar[r]      & \cC_{i'} \ar[r]    &0 \\
 }
\]
we find that the morphism $\cC_{(i,\pi)}\to \cC_{(i',\id_{B_0})}$ induces an  an equivalence in $E$ theory. Finally the (split) exact sequence 
\[\xymatrix{0 \ar[r] & A(0,1) \ar[r]    & \cC_{(i',\id_{B_0})}   \ar[r]      & B_0[0,1) \ar[r]    &0 }\]
yields the desired $E^1$-equivalence.
\end{proof}

\begin{remark}
We may note that we have just shown that the morphism $\cC_{(i,\pi)}\to \cC_{(\id_A,\id_A)}\simeq A(0,1)$ is invertible in $E$ theory.
\end{remark}

\subsection{\texorpdfstring{Difficulties at  the level of reduced $C^*$-algebra}{}}\label{sec:exactness}

Let us discuss the reduced version of this diagram:
\begin{itemize}
\item If the restriction $\cG_{|Y}$ is an amenable groupoid we also have horizontal exactness at the level of reduced $C^{\ast}$-algebras.
\item If $\cG_{|W\cap \Omega}$ is not amenable then the integration along fibers may not exist at the level of the kernels. We discuss an example as such in \ref{ex:notamen}.
\end{itemize}

In view of examples \ref{exs:basic} we focus now on foliations $(M,\cF)$ arising from an action of a Lie group $G$ on manifold $M$. We assume that $W=M$, the action groupoid $\cG = M \rtimes G$ realises a nice decomposition of singularity height 1 for $(M,\cF)$ and the complementary set $Y$ is a point.

If the group $G$ is amenable then integration along fibers of the quotient map $\cG \to H(\cF)$ gives the following diagram:
\begin{figure}[H]
\[
\xymatrix{
0 \ar[r] & C_0(\Omega)\rtimes G \ar[r]^{\iota_{\cG}} \ar[d]_{\pi_{\Omega}} & C_0(M)\rtimes G \ar[r]^{\pi_{\cG}} \ar[d]_{\pi} & C^{\ast}(G) \ar[r] \ar@{=}[d]  & 0 \\
0 \ar[r] & C^{\ast}(\Omega,\cF_{|\Omega}) \ar[r]^{\iota_{\cF}} & C^{\ast}(M,\cF)  \ar[r]^{\pi_{\cF}} & C^{\ast}(G) \ar[r] & 0
}
\]
\caption{Exact sequences for a nicely decomposable foliation of singularity height 1 arising from the action of an amenable Lie group.\label{fig:amenable}}
\end{figure}



If $G$ is not amenable, the sequences are exact at the level of full $C^{\ast}$-algebras. At the reduced $C^{\ast}$-algebra level\begin{itemize}
\item the sequences need not be exact;
\item the morphism $C_0(\Omega_1)\rtimes G \to C_r^{\ast}(M,\cF)$ obtained as a composition of $\pi$ with the morphism $C^{\ast}(M,\cF)\to C_r^{\ast}(M,\cF)$ doesn't need to pass to the quotient $C_0(\Omega_1)\rtimes_r G$ of $C_0(\Omega_1)\rtimes G$.
\end{itemize}
Note however that :

\begin{itemize}
\item In most cases that we consider the top sequence in figure \ref{fig:amenable} is exact since the groups we consider are exact.
\item We allways have some completely positive splittings (see prop. \ref{splittingred}).
\item In the example of the action of $GL(2,\R)$ on $\R^2$, since the stabilizers are amenable, the morphism $\pi_\Omega:C_0(\R^2\setminus\{0\})\rtimes GL(2,\R)\to \cK$ is defined at the reduced $C^*$-algebra level. As the group $GL(2,\R)$ is $K$-amenable, we find that in this case the full and reduced $C^*$-algebra of $\cF$ are $KK$-equivalent.
\end{itemize}

\begin{prop}\label{splittingred}
\begin{sloppypar}
Let $\cG$ be the action groupoid in figure \ref{fig:amenable}. Then the morphisms $C^{\ast}_r(\cG) \to C^{\ast}_r(G)$, $C^{\ast}(\cG) \to C^{\ast}(G)$ and $ C^{\ast}(M,\cF) \to C^{\ast}(G)$ have completely positive splittings.
\end{sloppypar}
\end{prop}
\begin{proof}
This is due to the fact $C^*(G)$ sits in the multiplier algebra of a crossed product $A\rtimes G$ - and the same for reduced ones:

We construct a completely positive splitting for the map $C^{\ast}(\cG) \to C^{\ast}(G)$: Take a function $f \in C_0(M)$ such that $||f||=1$ and $f(x_0)=1$.  

Given $\zeta \in C^{\ast}(G)$ put $\sigma(\zeta)=f^{\ast}\zeta f$. This is obviously a completely positive (and contractive) splitting of the top sequence. (The same is true for the reduced algebra and crossed products.)

Composing the completely positive splitting  $C^*(G)\to C^*(\cG)$  with the morphism $\pi:C^{\ast}(\cG) \to C^{\ast}(M,\cF)$ (given by integration along the fibers) we obtain a completely positive splitting of the second sequence.
\end{proof}

We now give an example where the morphism $\pi_{\Omega}$ is not defined at the reduced $C^*$-algebra level:
\begin{ex}\label{ex:notamen}
Consider the action of $G=SL(n,\R)$ on $\R^n$ for $n \geq 3$. This action has two orbits: $\{0\}$ and $ \Omega=\R^{n}\setminus \{0\}$. The stabilizer of a nonzero point for this action is isomorphic to $H=\R^{n-1} \rtimes SL(n-1,\R)$ which is not amenable if $n \geq 3$. The full crossed product $C_0(\R^n \setminus\{0\})\rtimes  SL(n,\R)$ is Morita equivalent to $C^{\ast}(H) $. Therefore, the full $C^*$-algebra of this foliation is the quotient of $C_0(\R^n)\rtimes SL(n,\R)$ sitting in a diagram \[
\xymatrix{
0 \ar[r] & C_0(\Omega)\rtimes G\simeq \cK\otimes C^*(H) \ar[r]  \ar[d]_{\id_\cK\otimes \varepsilon_{H}} & C_0(\R^n)\rtimes G \ar[r]^{\pi_{\cG}} \ar[d]_{\pi} & C^{\ast}(G) \ar[r] \ar@{=}[d]  & 0 \\
0 \ar[r] & \cK \ar[r]& C^{\ast}(\R^n,\cF)  \ar[r]^{\pi_{\cF}} & C^{\ast}(G) \ar[r] & 0
}
\] where $\varepsilon_H$ denotes the trivial representation of $H=\R^{n-1} \rtimes SL(n-1,\R)$.\\
The reduced crossed product $C_0(\R^n \setminus\{0\})\rtimes_{r} SL(n,\R)$ is Morita equivalent to $C_r^{\ast}(H)$. 

Note that the trivial representation $C^{\ast}(H) \to \C$ is not defined at the level of $C^{\ast}_r(H)$ when the group $H$ is not amenable.

The reduced $C^*$-algebra $C_r^*(\R^n,\cF)$ of this foliation is the quotient of $C_0(\R^n)\rtimes G$ corresponding to the sum of the two covariant representations on $L^2(\Omega)=L^2(\R^n)$ and $\{0\}\times G$.
\end{ex}
\begin{remark}\label{rem:exseq}
In the sequel we will use (almost) only full $C^{\ast}$-algebra to ensure that our sequences are exact and the trivial representation exists. This is legitimate from the point of view of the Baum-Connes conjecture, since the assembly map factorizes through the $K$-theory of the full $C^{\ast}$-algebra anyway. 
\end{remark}

\section{Larger singularity height and telescope}\label{sec:higherheight}
\label{sec:Higherlength}

In this section we extend the constructions of section \ref{sec:convalgheightone} to singular foliations of arbitrary singularity height. The mapping cone of section \ref{sec:convalgheightone} is replaced by a telescope. We start by recalling telescope constructions.

\subsection{Mapping telescopes}

Let us recall the following construction of  $C^{\ast}$-algebras:

\begin{definition}\label{dfn:telescope} Let $n\in \N\cup\{+\infty\}$.
Given $C^*$-algebras $(B_k)_{0\le k< n}$ and $(I_k)_{1\le k< n}$ and morphisms $\alpha_k:I_{k}\to B_{k-1}$ and $\beta_k:I_{k}\to B_{k}$, we define the associated \emph{telescopic $C^{\ast}$-algebra}  $$\cT((\alpha_k)_{1\le k<n},(\beta_k)_{1\le k<n})$$ to be the  $C^{\ast}$-algebra comprising of $$((\phi_k)_{0\le k< n},(x_k)_{1\le k<n}) \in \prod_{0\le k<n} B_k[k,k+1] \times \prod_{1\le k<n} I_k $$ such that \begin{itemize}
\item  for $1\leq k < n$ we have $\phi_{k}(k )=\beta_k(x_{k})$ and $\phi_{k-1}(k)=\alpha_{k-1}(x_{k})$.
\item $\phi_0(0)=0$,
\item $\begin{cases}
\phi_{n-1}(n)=0&$if $\quad n\ne +\infty \\
\lim_{k\to +\infty}\|\phi_k\|=\lim_{k\to +\infty}\|x_k\|=0 &$if $\quad n= +\infty
\end{cases}$
\end{itemize}
\end{definition}

\begin{remark}\label{rem:teldominateslimit}
A particular case of a telescope is when $I_k=B_{k-1}$ and $\alpha_k=\id_{I_k}$. We will denote just by $\cT(\beta)$ the associated mapping telescope $\cT(\id,\beta)$. In that case, if $n=\infty$, let us also denote by $B_\infty$ the inductive limit of the system $(B_k,\beta_k)$. We then have an exact sequence 
\begin{equation}0\to \cT(\beta)\to \cT'(\beta)\to B_\infty\to 0\label{eq1}\end{equation}
 where $\cT'(\beta)$ is the set of elements that have a limit at $\infty$: it is the inductive limit of $\cT'_k(\beta)$ (\ie the closure in $\cM(\cT(\beta))$ of the increasing union of $\cT'_k(\beta)$) where $\cT'_k(\beta)$ is the algebra of functions that become constant after $k$: \ie such that $\phi_\ell$ is constant for $\ell \ge k$ - and of course equal to the image in $B_\ell$ of the element $\phi_k(k)\in B_k$. Note that we have a diagram $$\xymatrix{0\ar[r]& \cT(\beta)\ar[r]\ar[d]& \cT'(\beta)\ar[r]\ar[d]& B_\infty\ar[r]\ar@{=}[d]& 0\\0\ar[r]& B_\infty(0,+\infty)\ar[r]& B_\infty(0,+\infty]\ar[r]& B_\infty\ar[r]& 0}$$
It follows that the composition of the element in $E^1(B_\infty,\cT(\beta))$ given by the exact sequence (\ref{eq1}) with the morphism $\cT(\beta)\to B_\infty(0,+\infty)$ is the unit element of $E^1(B_\infty,B_\infty(0,+\infty))=E(B_\infty,B_\infty)$.
\end{remark}

Using this remark, one obtains the following results (\cf \cite{Rosenberg-Schochet}):

\begin{prop}
\begin{enumerate}
\item If $I_k$ and $B_k$ are $E$-contractible, then $\cT(\alpha,\beta)$ is also $E$-contractible.
\item If $(B_k,\beta_k)$ is an inductive system of $E$-contractible $C^*$-algebras, then their inductive limit $B_\infty$ is $E$-contractible.
\item If $(B_k,\beta_k)$ is an inductive system of   $C^*$-algebras then $\cT'(\beta)$ is $E$-contractible and the element in $E^1(B_\infty,\cT(\beta))$ given by the exact sequence (\ref{eq1}) is invertible. 
\end{enumerate}
\begin{proof}
 \begin{enumerate}
\item Indeed, we have a unital ring morphism $\prod E(B_k,B_k)\to E(\bigoplus B_k,\bigoplus B_k)$ and it follows that if the $B_k$ are $E$-contractible, then $\bigoplus B_k$ is $E$-contractible. Since $\bigoplus B_k$ and $\bigoplus I_k$ are $E$-contractible, then by the exact sequence $$0\to \bigoplus B_k(0,1)\to \cT(\alpha,\beta)\to \bigoplus I_k\to 0,$$ the telescope $\cT(\alpha,\beta)$ is $E$-contractible.
\item Since the telescope $\cT(\beta)$ is $E$-contractible, the algebra $B_\infty$ is $E$-contractible since, by remark \ref{rem:teldominateslimit}, it is $E$-sub-equivalent to $\cT(\beta)$.

\item We have (split) exact sequences $0\to B_k(k,k+1]\to \cT'_{k+1}(\beta)\to \cT'_{k}(\beta)\to 0$ and it follows by induction that, for all $k$, $\cT'_{k}(\beta)$ is $K$-contractible - and therefore $E$-contractible (note that $\cT'_{0}(\beta)=0$). It follows that the inductive limit $\cT'(\beta)$ is $E$-contractible and therefore the exact sequence (\ref{eq1}) induces an $E^1$-equivalence. \qedhere
\end{enumerate}

\end{proof}
\end{prop}

\medskip In fact a telescope can be expressed as a mapping torus:

\begin{remarks} \label{rem:telescope=cone}
\begin{enumerate}
\item Recall that given $C^*$-algebras $A,B$ and morphisms $u,v:A\to B$ the torus $C^*$-algebra $\gT(u,v)$ is $$\{(a,\phi)\in A\times B[0,1];\ u(a)=\phi(0),\ v(a)=\phi(1)\}.$$

In fact the telescopic $C^*$-algebra $\cT(\alpha,\beta)$ identifies with the torus $C^*$-algebra $\gT(\widecheck\alpha,\widecheck\beta)$ of the morphisms $\widecheck\alpha,\widecheck\beta:\bigoplus_{k=1}^n I_k\to \bigoplus_{k=0}^{n} B_k$ defined by $$\widecheck \alpha((x_k)_k)=(0,\alpha_1(x_1),\ldots,\alpha_k(x_k),\ldots)$$ and $$\widecheck \beta((x_k)_k)=
\begin{cases}
(\beta_0(x_1),\ldots,\beta_{n-1}(x_n),0)& \hbox{if}\ n\in\N\\
(\beta_k(x_{k+1}))_{k\in \N}& \hbox{if}\ n=+\infty
\end{cases}
$$

\item In turn, a mapping torus is easily seen to be $K$-equivalent to a mapping cone:

Let $A,B$ be $C^*$-algebras and $j_{\pm}:A\to B$ $*$-homomorphisms. Let $j:A(\R_+^*)\to B(\R)$ be the $*$-homomorphism defined by $$j(\phi)(t)=\begin{cases}j_+(\phi(t))&\hbox{if}\ t>0\\0&\hbox{if}\ t=0\\j_-(\phi(-t))&\hbox{if}\ t<0\end{cases}$$
Then $\gT(j_+,j_-)(\R_+^*)$ is canonically isomorphic with $C_j$.

Indeed, $$\gT(j_+,j_-)(\R_+^*)=\{(\phi,\psi)\in A(\R_+^*)\times B(\R_+^*\times [0,1]);\ \psi(t,0)=j_+(\phi(t)), \ \psi(t,1)=j_-(\phi(t))\}$$ and 
\begin{eqnarray*}
 C_j&=&\left\{(\phi,\psi)\in A(\R_+^*)\times B(\R\times \R_+);\  \psi(t,0)=\begin{cases}j_+(\phi(t))&\hbox{if}\ t>0\\0&\hbox{if}\ t=0\\j_-(\phi(-t))&\hbox{if}\ t<0\end{cases}\quad \right\}\\
&=&\left\{(\phi,\psi)\in A(\R_+^*)\times B(\R\times \R_+\setminus \{(0,0)\});\  \psi(t,0)=\begin{cases}j_+(\phi(t))&\hbox{if}\ t>0\\j_-(\phi(-t))&\hbox{if}\ t<0\end{cases}\quad \right\}
\end{eqnarray*}
are  isomorphic through the homeomorphism $(r,\theta)\mapsto (r\cos \pi\theta,r\sin\pi\theta)$ from $\R_+^*\times [0,1]$ onto $(\R\times \R_+)\setminus \{(0,0)\}$.
\end{enumerate}
\end{remarks}

\subsection{The telescope of nicely decomposable foliations}\label{subsec:telescope}

Let $(M,\cF)$ be a \emph{nicely decomposable} foliation of height $n\in \N\cup\{\infty\}$ in the sense of definition \ref{dfn:goodgpd}. Generalizing the case $n=1$, we construct a $C^{\ast}$-algebra which is $E$-equivalent with the (full)  foliation $C^{\ast}$-algebra. 

We are thus given open subsets $(W_k)_{k<n+1}$ groupoids $\cG_k\gpd W_k$ and morphisms $\cG_k|_{W_k\cap W_{k-1}}\to \cG_{k-1}$ satisfying the conditions of def. \ref{dfn:goodgpd}.

Let $\Omega_k=\bigcup_{j\le k} W_j$ be the sequence of strata of this decomposition and $Y_k=\Omega_k\setminus \Omega_{k-1}$.

Since $\cF$ is assumed to be nicely decomposable, we are given Lie groupoids  $\cG_k \gpd W_k$ and morphisms of Lie groupoids $q_k:\cG_{k}|_{\Omega_{k-1}} \to \cG_{k-1}$ such that $\cG_k |_{Y_k} = H(Y_k,\cF)$.

For every $0\leq k < n+1$ consider the full $C^{\ast}$-algebras $A_k = C^{\ast}(\Omega_k,\cF)$ and  $B_k = C^{\ast}(\cG_k)$ and the morphism, obtained by integration along the fibers $p_k:B_k\to A_k$. Put also $Q_k = C^{\ast}(\cG_k |_{Y_k})$.  We have a diagram:

\begin{figure}[H]
\[\xymatrix @!0 @R=3pc @C=3pc {
   0 \ar[rr] &&      I_{k+1} \ar[dd]_{\pi_{k}}\ar[rrr]^{j_{k+1}} \ar[rd]^{q_{k+1}}            &&& B_{k+1} \ar[dd]^{p_{k+1}} \ar[rrr]                 &&& Q_{k+1} \ar@{=}[dd] \ar[rr]     && 0 \\
   &&&B_k\ar[ld]^{p_k}
   \\
0 \ar[rr]             &&  A_{k} \ar[rrr]^{\iota_{k}} \ar[rrd]_{\tilde \iota_k}       &&&     A_{k+1} \ar[rrr]  \ar[ld]^{\tilde \iota_{k+1}}     &&& A_{k+1} / A_{k} \ar[rr]     &&  0 \\
&&&&C^*(M,\cF) }\]
\caption{Short exact sequences of strata.\label{fig:telescope}}
\end{figure}
Here the map $q_k$ is integration along the fibers of the groupoid morphism $q_k:\cG_{k}|_{\Omega_{k-1}} \to \cG_{k-1}$ and $\pi_k=p_k\circ q_{k+1}:I_k\to A_k$. The quotient algebras $B_{k}/I_{k-1}$ and $A_k/A_{k-1}$ coincide (with $Q_k$).

Denote also  by $\tilde \iota_k:A_k\to C^*(M,\cF)=A_n$ the inclusion. 

As, for every $k< n$ we have $\tilde \iota  _{k}\circ \pi_{k}=\tilde \iota_k\circ p_k\circ q_k=\tilde \iota_{k+1}\circ p_{k+1}\circ j_{k+1}$, we obtain a morphism $\Psi:\cT(q,j)\to A_n(0,n+1)$ defined by $\Psi((\phi_k)_{0\le k<n+1},(x_k)_{1\le k<n+1})(t)=\tilde \iota_k\circ p_k(\phi_k(t))$ for $t\in [k,k+1]$.

\begin{thm}\label{thm:telescope} With the above notation, the class in $E(\cT(q,j),A_n(0,n+1))$ of the morphism $\Psi$ is invertible.
\end{thm}
\begin{proof} Let $\cB=\{f\in A_n(0,n+1];\ \forall t\in \R_+^*, \forall k\in \N;\  t-1\le k\le n\Rightarrow f(t)\in A_k\}$ and put $\cJ=\{f\in B;\ f(n+1)=0\}$ 

The inclusion $\cJ\to A_n(0,n+1)$ is an $E$-equivalence (\cf \cite{Rosenberg-Schochet} - if $n<+\infty$, it is a $KK$-equivalence). Its inverse is given by the exact sequence $0\to \cJ\to \cB\to A_n\to 0$.

For $\ell\in \N$, $\ell\le n$, put $\cJ_\ell=\{f\in \cJ;\ f(t)=0 \hbox{ if } t\ge \ell+1\}$ and let $\cT_\ell$ be the ideal $$\cT_\ell=\{((\phi_k)_{0\le k<n+1},(x_k)_{1\le k<n+1})\in \cT; \ \forall k>n,\ \phi_k=0 \ \hbox{ and }\ x_k=0\}.$$

Let us show by induction that the morphism $\Psi_\ell:\cT_\ell\to \cJ_\ell$ induced by $\Psi$ is an $E$-equivalence: \begin{itemize}
\item $\Psi_0$ is an isomorphism (and the case $\ell=1$ follows from the proof of prop. \ref{prop:Eequiv}).
\item We have an exact sequence 
$$\xymatrix{0\ar[r]&\cT_{\ell-1}\ar[r]\ar[d]^{\Psi_{\ell-1}}&\cT_{\ell}\ar[r]\ar[d]^{\Psi_{\ell}}&\cC_{j_\ell}\ar[r]\ar[d]^{\tilde p_\ell}&0\\
0\ar[r]&\cJ_{\ell-1}\ar[r]&\cJ_{\ell}\ar[r]&\cC_{\iota_{\ell-1}}\ar[r]&0}
$$
where $\tilde p_\ell:\cC_{j_\ell}\to \cC_{\iota_{\ell-1}}$ is the morphism induced by $p_\ell :B_\ell\to A_\ell$ at the cone level. 

Examining figure \ref{fig:telescope}, as $j_\ell$ and $\iota_{\ell-1}$ are inclusions of ideals and $p_\ell$ induces an isomorphism $B_\ell/I_\ell\to A_{\ell}/A_{\ell-1}$, we deduce that $\tilde p_\ell$ is $E$-invertible. We thus obtain the induction step.
\end{itemize}

If $n$ is finite, the proof is complete. 

If $n=+\infty$, the mapping cone $\cC_{\Psi_\ell}$ is $E$-contractible for all $\ell$ and it follows that their inductive limit $\cC_{\Psi}$ is $E$-contractible.
\end{proof}


\section{\texorpdfstring{Longitudinally smooth groupoids and equivariant $KK$-theory}{Longitudinally smooth groupoids and equivariant KK-theory}}\label{sec:longsmoothKK}

We proved in theorem \ref{thm:telescope} that the telescopic algebra $\cT(q,j)$ has the same $K$-theory as $C^*(M,\cF)$. In the next section, we will build a Baum-Connes map for the telescopic algebra $\cT(q,j)$, which will give us a Baum-Connes map for $C^*(M,\cF)$.

The telescopic algebra $\cT(q,j)$ associated with a nicely decomposable foliation, as well as the foliation algebra itself (thanks to \cite{Debord2013}), is the $C^*$-algebra of a \emph{longitudinally smooth groupoid} in a sense that we briefly describe here. 

Note that all the constructions we will give below generalize easily to groupoids that are covered by $C^{\infty,0}$ manifolds or even locally compact spaces with Haar measures.

\subsection{Longitudinally smooth groupoids}

A \emph{longitudinally smooth groupoid} is a groupoid $G\overset{t,s}\gpd G^{(0)}$ such that:
\begin{itemize}
\item its set of objects is endowed with a structure of smooth manifold (possibly with boundary or corners); 
\item for every $x \in G^{(0)}$, the set $G^x=t^{-1}(\{x\})$  caries also a smooth structure (without boundary) and the source map $s:G^x\to G^{(0)}$ is smooth with (locally) constant rank;
\item the ``smooth structure'' of $G$ itself is given by an \emph{atlas} which is a family of smooth manifolds $(U_i)_{i\in I}$ (possibly with boundary or corners) and maps $q_i:U_i\to G$.
\end{itemize}

We assume that these smooth structures satisfy:
\begin{description}
\item[Compatibility.] For every $i\in I$, the maps $t\circ q_i$ and $s\circ q_i$ are smooth submersions; for every $i\in I$ and every $x\in G^{(0)}$ the map $q_i$ induces a smooth submersion $q_i^{-1}(G^x)\to G^x$.
\item [Minimal elements.] For every $\gamma\in G$, there exists $i\in I$ and $z\in q_i$ such that $q_i(z)=\gamma$ and the map $q_i^{-1}(G^{t(\gamma)})\to G^{t(\gamma)}$ is a local diffeomorphism near $z$. If $j\in I$ and $z'\in U_j$ are such that $q_j(z')=\gamma$, then there is an open neigborhood $V'\subset U_j$ of $z'$ and a submersion $\varphi:V'\to U_i$ such that $q_i\circ \varphi=(q_j)_{|V'}$.
\item [Inverse is smooth.] For every $i\in I$, there exists $j\in I$ and a diffeomorphism $\kappa:U_i\to U_j$ such that $q_j\circ \kappa(z)=q_i(z)^{-1}$ for every $z\in U_i$.
\item[Composition is smooth.] For every $i,j\in I$, let $U_i\circ U_j$ be the fibered product $U_i\times_{s\circ q_i,t\circ q_j}U_j$. For every $(z_i,z_j)\in U_i\circ U_j$, there is a $k\in I$, a neighborhood $W$ of $(z_i,z_j)$ in $U_i\circ U_j$ and a submersion $\varphi:W\to U_k$ such that for all $(w_i,w_j)\in W$ we have $q_i(w_i)q_j(w_j)=q_k\circ \varphi (w_i,w_j)$.
\end{description}

Exactly as in \cite{AS1}, we may associate to a longitudinally smooth groupoid a $C^*$-algebra $C^*(G)$ (as well as a reduced one, since the $s$-fibers are assumed to be manifolds).

\begin{exs}\begin{enumerate}
\item A Lie groupoid is of course a  longitudinally smooth groupoid. The atlas is the groupoid itself!

\item The holonomy groupoid of a singular foliation is such a longitudinally smooth groupoid, the atlas is given by bisubmersions (see \cite{AS1}).

\item The telescopic algebra $\cT(q,j)$ constructed in the previous section is associated with the groupoid $G=\bigcup_{k=0}^n\cG_k\times (k,k+1)\cup \bigcup_{k=1}^n(\cG_k)_{\Omega_{k-1}\cap W_k}\times \{k\}$. Its set of objects is the open subset $\Big(\bigcup _{k=0}^nW_k\times (k,k+1)\Big)\cup \Big(\bigcup _{k=1}^n(\Omega_{k-1}\cap W_k)\times (k-1,k+1)\Big)$ of $M\times \R_+^*$.

It is endowed with the atlas formed by the Lie groupoids $(\cG_k\times (k,k+1))_{k\in \N,\ k\le n}$ and $((\cG_k)_{\Omega_{k-1}\cap W_k}\times (k-1,k+1))_{k\in \N,\ 1\le k\le n}$.
\end{enumerate}
\end{exs}

\subsection{\texorpdfstring{Action of a longitudinally smooth groupoid on a $C^*$-algebra}{Action of a longitudinally smooth groupoid on a C*-algebra}}

We now fix a longitudinally smooth groupoid $G\overset{t,s}\gpd M$ with an atlas $(U_i,q_i)_{i\in I}$. We put $s_i=s\circ q_i$ and $t_i=t\circ q_i$.

\subsubsection{\texorpdfstring{Action of a locally compact groupoid on a $C^*$-algebra \cite{Kasparov1988, LeGall}}{Action of a locally compact groupoid on a C*-algebra}}

For the convenience of the reader we recall some definitions on $C(X)$-algebras and actions of locally compact groupoids from \cite{Kasparov1988, LeGall}:

Let $M$ be a locally compact space.
\begin{enumerate}
\item A $C_0(M)$-algebra is a pair $(A,\theta)$ where $A$ is a $C^{\ast}$-algebra and $\theta : C_0(M) \to \cZ\cM(A)$ a $\ast$-homomorphism such that $\theta(C_0(M))A=A$. (Here $\cZ\cM(A)$ is the center of the multiplier algebra of $A$.)
\item Put $A_b = \{a \in \cM(A) : \phi a \in A \text{ for all } \phi \in C_0(X)\}$ and $A_c = A C_c(X)$.
\item Let $A, B$ be $C_0(M)$-algebras. A homomorphism of $C_0(M)$-algebras $\phi : A \to B$ is a  $C_0(M)$-linear homomorphism of $C^*$-algebras.
\item \label{ex:CMalgebra} Let $N$ be a locally compact space and $p : N \to M$ a continuous map. Then $C_0(N)$ is a $C_0(M)$-algebra thanks to the map $\theta = p^{\ast} : C_0(M) \to C_b(N)=\cM(C_0(M))$.

\bigskip Let $A$ be a $C_0(M)$ algebra.
\item  \label{a} For every $x \in M$ there is a fiber $A_x = A/C_x A$ where $C_x = \{h \in C_0(M) : h(x)=0\}$. The natural map $A \to \prod_{x \in M} A_x$ induced by the quotient maps $\pi_x : A \to A_x$ is injective. For instance, in example \ref{ex:CMalgebra}, given $x \in M$ the fiber $C_0(N)_x$ is $C_0(N_x)$ where $N_x = p^{-1}(x)$.
\item \label{b} A homomorphism of $C_0(M)$-algebras $\phi : A \to B$ induces a homomorphism of $C^{\ast}$-algebras $(\phi_x)_{x \in M} : \prod_{x \in M} A_x \to \prod_{x \in M} B_x$. The homomorphism $\phi$ is injective (surjective) if and only of $\phi_x$ is injective (surjective) for every $x \in M$.
\item \label{c} There are natural operations of restriction to open and closed subsets of $M$: If $U$ is an open subset of $M$ and $F=X\setminus U$; the algebra $C_0(U)$ identifies with the ideal $C_0(U)=\{f \in C_0(M) : f(y)=0 \text{ for all } y \in F\}$ of $C_0(M)$. Then $A_U$ denotes the $C_0(U)$-algebra $C_0(U)A$ and $A_F$ the $C_0(F)$-algebra $A/A_U$. If $Y\subset X$ is locally closed, then $Y$ is open in $\overline Y$ and $A_Y$ denotes the $C_0(Y)$-algebra $(A_{\overline Y})_Y$.
\item \label{d} Let $A$ be a $C_0(M)$ and $B$ a $C_0(N)$-algebra. Then $A \otimes_{max} B$ is a $C_0(M \times N)$-algebra. When $M = N$, the restriction of $A \otimes_{max} B$ to the diagonal $\{(x,x) : x \in M\}$ (which is a closed subset of $M \times M$) is a $C_0(M)$-algebra denoted $A \otimes_{C_0(M)} B$.
\item \label{e} Again let $A$ be a $C_0(M)$-algebra and consider a smooth map $p : N \to M$. We denote $p^{\ast}A$ the $C_0(N)$-algebra obtained by restricting $A \otimes C_0(N)$ to the graph $\{(p(y),y) : y \in N\}$ which is a closed subset of $M \times N$ (here $C_0(N)$ is regarded as a $C_0(N)$-algebra). It is easy to see that this construction has the following properties:
\begin{itemize}
\item $(p^{\ast}A)_y = A_{p(y)}$ for every $y \in N$;
\item If $A, B$ are $C_0(M)$-algebras then $p^{\ast}A \otimes_{C_0(Y)} p^{\ast}B = p^{\ast}(A \otimes_{C_0(M)} B)$.
\item If $q : Z \to N$ is a smooth map then $q^{\ast}(p^{\ast}A)=(p\circ q)^{\ast} A$.
\end{itemize}

\item \label{f} With the previous notation, for every $a \in A$ we put $p^{\ast}a = a \otimes 1 \in (p^{\ast}A)_b$. If $\phi : A \to B$ is a homomorphism of $C_0(M)$-algebras we put $p^{\ast}\phi = \phi \otimes id_{C_0(N)} : p^{\ast}A \to p^{\ast}B$. 
\item \label{g} An action of a Lie groupoid $\cG \gpd M$ on a $C_0(M)$-algebra $A$ is defined in \cite{LeGall} by an isomorphism $\alpha$ of $C_0(\cG)$ algebras $s^*A\to t^*A$. This isomorphism is given by a family  of isomorphisms $\alpha_{\gamma} : A_{s(\gamma)} \to A_{t(\gamma)}$ for $\gamma \in \cG$. The isomorphism $\alpha$ is required to be a representation of $\cG$, \ie to satisfy  $\alpha_{\gamma \circ \gamma'} = \alpha_{\gamma} \circ \alpha_{\gamma'}$ for all $(\gamma,\gamma') \in \cG^{(2)}=\cG \times_{s,t}\cG$.
\end{enumerate}

\subsubsection{Action of a longitudinally smooth groupoid}

Let $G\overset{t,s}\gpd M$ be a longitudinally smooth groupoid with an atlas $(U_i,q_i)_{i\in I}$. We put $s_i=s\circ q_i$ and $t_i=t\circ q_i$.

\begin{definition}\label{dfn:Galg}
A $G$-algebra is a $C_0(M)$-algebra $A$ together with an isomorphism of $C_0(U_i)$-algebras $\alpha^i: s_i^{\ast}A \to t_i^{\ast}A$ for every $i\in I$. 
\begin{enumerate}
\item The isomorphism $\alpha^{i}$ is a family $(\alpha^{i}_u)_{i\in I}$ of isomorphisms $\alpha^{i}_u : A_{s_i(u)} \to A_{t_i(u)}$. We require that if $\gamma\in G$ is represented by two elements $u_i\in U_i$ and $u_j\in U_j$ (with $i,j\in I$), then $\alpha^{i}_{u_i}=\alpha^{j}_{u_j}$.
\item By (a), we get a well defined isomorphism $\alpha_\gamma:A_{s(\gamma)}\to A_{t(\gamma)}$. We require that for every composable $\gamma,\gamma'\in G$, we have $\alpha_{\gamma\gamma'}=\alpha_\gamma\circ \alpha_{\gamma'}$.
\end{enumerate}
\end{definition}

\begin{definition}
Let $(A,\alpha)$ and $(B,\beta)$ be $G$-algebras.\label{dfn:equivt}
\begin{enumerate}
\item  A morphism $\phi : A \to B$ is said to be $G$-equivariant if  it is $C_0(M)$-linear and for every $\gamma\in G$ we have $\phi_{t(\gamma)}\circ \alpha(\gamma) = \beta(\gamma)\circ \phi_{s(\gamma)}$.

\item\label{dfn:equivtMult} More generaly let $\phi : A \to \cM(B)$ be a morphism. Let $D=G(\phi)\oplus (0\oplus B)\subset A\oplus \cM(B)$ where $G(\phi)=\{(x,\phi(x));\ x\in A\}$ is the graph of $\phi$. We say that $\phi $ is equivariant if there is an action of $G$ on $D$ such that the inculsions $A\to D$ and $B\to D$ are equivariant.
\end{enumerate}
\end{definition}

\begin{exs}
\begin{enumerate}
\item The algebra $C_0(M)$ is a $G$-algebra. For every $i\in I$, we have  $t^*(C_0(M))=C_0(U)=s^*(C_0(M))$; the action $\alpha $ is the identity. For $i\in I$, then at every $u \in U_i$ we associate the identity map ($\C\to \C$).  In a sense, this corresponds to the trivial representation.

\item More generally, let $Y\subset M$ is a locally closed saturated subset (\ie such that for every $\gamma\in G$ we have $t(\gamma)\in Y\iff s(\gamma)\in Y$). Then $C_0(Y)$ is a $H(\cF)$-algebra. In that case, For every $i\in I$, we have $t^{-1}(Y)=s^{-1}(Y)$ since $Y$ is saturated and  $t^*(C_0(Y))=C_0(t^{-1}(Y))=C_0(s^{-1}(Y))=s^*(C_0(Y))$. Again, the action $\alpha $ is the identity.

\end{enumerate}
\end{exs}

\subsubsection{Covariant representations and full crossed products}\label{sec:covreps}

Let us very briefly extend some constructions of \cite[\S 4 and 5]{AS1} to the more general case of our longitudinally smooth groupoid $G\gpd M$ with atlas $(U_i,q_i)_{i\in I}$. \begin{itemize}
\item When $f:N\to M$ is a smooth submersion of manifolds, we may define a Hilbert $C_0(M)$-module $\cE_f$ obtained by completion of the space $C_c(N;\Omega^{1/2}\ker(df))$ with respect to the $C_0(M)$ valued inner product defined by $\langle \xi,\eta\rangle(x)=\int_{z\in f^{-1}(x)}\overline{\xi(z)}\eta(z)$. This Hilbert module is endowed with an action of $C_0(N)$.
\item Let $i\in I$. We may then construct two Hilbert $C^*$-modules $\cE_{t_i}$ and $\cE_{s_i}$ over $C_0(M)$.
\item As $C_0(M)$ sits in the multiplier algebra of $C^*(G)$, every representation $\pi _G$ of $C^*(G)$ on a Hilbert space $\cH$ gives rise to a representation $\pi_M$ of $C_0(M)$. 
\item The representation $\pi $ is then characterized by $\pi_M$ and, for every $i\in I$, a unitary $V_i\in \cL(\cE_{s_i}\otimes_{C_0(M)}\cH,\cE_{t_i}\otimes_{C_0(M)}\cH)$ intertwining the representations of $C_0(U)$. It therefore defines a measurable family of unitaries $U_u:H_{s_i(u)}\to H_{t_i(u)}$. We require that $U_u$ only dpends on the class of $q_i(u)$ in $G$ (almost everywhere) and that (almost everywhere) it determines a representation of the groupoid $G$. See \cite[\S 5.2]{AS1} for the details.
\end{itemize}

Let $G$ act on a $C^*$-algebra $A$ and let $\pi_A$ be a representation of $A$ on a Hilbert space $\cH$. Using the morphism from $C_0(M)$ to  the multiplier algebra of $A$, we obtain a representation of $C_0(M)$ to $\cL(\cH)$. For every $i\in I$, as the image of $C_0(M)$ sits in the center of  the multiplier algebra of $A$, we have representations $\pi_A^{s_i}:s_i^*(A)\to \cL(\cE_{s_i}\otimes_{C_0(M)}\cH)$ and $\pi_A^{t_i}: t_i^*(A)\to \cL(\cE_{t_i}\otimes_{C_0(M)}\cH)$.

A \emph{covariant representation} of $G$ and $A$ is given by a representation of $\pi_G$ of $C^*(G)$ and a representation $\pi_A$ of $A$ in the same Hilbert space $\cH$ such that the two representations of $C_0(M)$ agree and, for every $i\in I$, the unitary $V_i$ intertwines $\pi_A^{s_i}\circ \alpha^i$ with $\pi_A^{t_i}$. 

Then the closed linear span of $\pi_A(a)\pi_G(x)$ where $a$ runs over $A$ and $x$ over $C^*(G)$ is a $*$-subalgebra of $\cL(\cH)$. 

\begin{definition}\label{dfn:fullcrosprod}
The \emph{full crossed product} $A\rtimes G$ is the completion of this linear span with respect to the supremum norm over all covariant representations.
\end{definition}

Using the ``regular representations'' on $L^2(G_x)$, one may also construct a natural reduced crossed product.

\subsubsection{Actions of a longitudinally smooth groupoid on Hilbert modules}

Let $G$, $(U_i)_{i\in I}$, $s_i,t_i$ be as above.

\label{section:actHilbmod}
Let $(A,\alpha)$ be a  $G$-algebra and $\cE$ a Hilbert module over  $A$. As usually, we may define an action of $G$ on $\cE$ by saying that it is just given by an action of $G$ on the $C^*$-agebra $\cK(\cE\oplus A)$ in such a way that the natural morphism $A\to \cK(\cE\oplus A)$ is equivariant.
\\
This amounts to giving, for any $i\in I$, an isomorphism $\tilde \alpha :\cE\otimes_A s_i^*A\to \cE\otimes_A t_i^*A$ of Banach spaces, which corresponds to a family of isomorphisms $\tilde\alpha_u:\cE_{s_i(u)}\to \cE_{t_i(u)}$. We need compatibility with $\alpha$ which means that for every $x\in A_{s_i(u)}$ and $\xi,\zeta\in \cE_{s_i(u)}$, we have $\tilde \alpha_u(\xi x)=\tilde \alpha_u(\xi ) \alpha_u( x)$ and $\alpha_u(\langle \xi|\zeta\rangle)=\langle \tilde \alpha_u(\xi ) |\tilde \alpha_u(\zeta)\rangle$.
\\
As above, we require that $\tilde\alpha_u$ only depends on the class of $u$ in $G$ and that the so defined $\tilde \alpha_\gamma:\cE_{s(\gamma)}\to \cE_{t(\gamma)}$  for $\gamma\in H(\cF)$ defines a morphism of groupoids, which means that $\tilde \alpha_{\gamma\gamma'}=\tilde \alpha_{\gamma}\tilde \alpha_{\gamma'}$.
\\
Note also that, given an action of $G$ on a $\cE$, we obtain for any $i\in I$,  an isomorphism of $C_0(U_i)$-algebras $\cK(\cE\otimes_A s_i^*A)\to \cK(\cE\otimes_A t_i^*A)$ and of their multipliers $\check \alpha_U :\cL(\cE\otimes_A s_i^*A)\to \cL(\cE\otimes_A t_i^*A)$.

\subsection{\texorpdfstring{$G$-equivariant $KK$-theory}{G-Equivariant KK-theory}}

Let $G\overset{t,s}\gpd M$ be a longitudinally smooth groupoid with an atlas $(U_i,q_i)_{i\in I}$. We put $s_i=s\circ q_i$ and $t_i=t\circ q_i$.

Here we use the apparatus developed in the previous sections to construct the left-hand side of the Baum-Connes conjecture in classical terms (\eg as in \cite{LeGall}). Namely we define the groups $KK_{G}(A,B)$ in \S \ref{sec:eqKasp}. The difficulty is to construct the Kasparov product; we do this in \S \ref{sec:Kasprod}.

\subsubsection{Equivariant Kasparov cycles}\label{sec:eqKasp}

We may of course define graded $G$ algebras, graded Hilbert modules...

In what follows, all algebras will be $\Z/2\Z$-graded and all commutators are graded ones.

Also, all the $C^*$-algebras and Hilbert $C^{\ast}$-modules that we will consider are supposed to be separable

Recall the following from \cite{Kasparov1981}:
\begin{itemize}
\item Let $A,B$ be graded $C^*$-algebras. An $(A, B)$-bimodule is a pair $(\cE,\pi_A)$ where $\cE$ is a $B$-Hilbert $C^{\ast}$-module and $\pi_A : A \to \cL(\cE)$ a representation which preserves the degree. For every $\xi \in \cE$ and $a \in A$ we denote $a\xi = \pi_A(a)(\xi)$.
\item A Kasparov $(A, B)$-bimodule is a triple $(\cE,\pi_A,F)$ where $(\cE,\pi_A)$ is an $(A, B)$-bimodule and $F\in \cL(\cE)$ is of degree $1$ (for the grading) and for all $a\in A$, the elements $[F,\pi_A(a)]$, $(F-F^*)\pi_A(a)$ and $(1-F^2)\pi_A(a)$ are all in $\cK(\cE)$.
\end{itemize}

\begin{definition}\label{KKcycles}
Let $(A,B)$ be $G$-algebras. A \emph{$G$-equivariant Kasparov $(A,B)$ bimodule} is a Kasparov $(A,B)$ bimodule $(\cE,\pi_A,F)$ such that:
\begin{enumerate}
\item $\cE$ is endowed with an action  of $G$ (\cf sect. \ref{section:actHilbmod}) and the representation $\pi_A:A\to \cL(\cE)=\cM(\cK(\cE))$ is $G$-equivariant (in the sense of def. \ref{dfn:equivt}.\ref{dfn:equivtMult})
\item for every $i\in I$ and $h\in C_0(U_i)$,  we have $(\check \alpha_i(F\otimes 1)-F\otimes 1)h\in \cK(\cE\otimes _At_i^*A)$.
\end{enumerate}
\end{definition}

\begin{itemize}
\item Two $G$-equivariant Kasparov bimodules $(\cE,\pi_A,F), (\cE',\pi'_A,F')$ are \emph{unitarily equivalent} if there exists a $G$-equivariant unitary $U \in \cL(\cE,\cE')$ of degree $0$ which satisfies $UFU^{\ast}=F'$ and $U\pi_A(a)U^{\ast}=\pi'_A(a)$ for all $a \in A$. 
\item Denote by $E_{G}(A,B)$ the set of equivalence classes of $G$-equivariant Kasparov bimodules. 
\item A homotopy in $E_{G}(A,B)$ is an element of $E_{G}(A,B[0,1])$. We define $KK_{G}(A,B)$ to be the set of homotopy classes of elements of $E_{G}(A,B)$. 
\item The direct sum of Kasparov bimodules induces an abelian group structure in $KK_{G}(A,B)$.
\item We define the unit element $1_A\in KK_{G}(A,A)$ as the class of $(A,\iota_A,0)$, where $\iota_A(a)=a \in \cK(A)$ for all $a \in A$ where the action of $G$ on the $C^*$-module $A$ is the action of $G$ on the $C^*$-algebra $A$.
\end{itemize}

\subsubsection{Kasparov's descent morphism}

Given an equivariant Hilbert $B$ module $\cE$, we may define the crossed product $\cE\rtimes {G}=\cE\otimes_BB\rtimes G$ - and the same for the reduced crossed product. If we have an equivariant action $A\to \cL(E)$, we naturally obtain an action $A\rtimes {G}\to \cL(\cE\rtimes {G})$.

Let $(\cE,F)$ is an equivariant Kasparov $(A,B)$ bimodule. Put $F\widehat \otimes 1\in \cL(\cE\otimes_BB\rtimes G)=\cL(\cE\rtimes {G})$. We check as in \cite{LeGall} that $(\cE\rtimes {G},F\widehat \otimes 1)$ is a Kasparov $(A\rtimes {G},B\rtimes {G})$ bimodule. This construction gives a well defined \emph{descent morphism} $j_{G}:KK_{G}(A,B)\to KK(A\rtimes {G},B\rtimes {G})$. 

In the same way we also obtain a \emph{reduced descent morphism}.

\subsubsection{Kasparov product - a general approach}\label{sec:Kasprod}

In order to define the Kasparov product in this equivariant context, we need first to understand the analogue of Kasparov's ``technical theorem'' (\cite[\S 3, Theorem 4]{Kasparov1981}). It turns out that, in a sense, the original theorem actually applies when formulated in a slightly different way. In turn, this formulation contains many equivariant generalizations.

We start by recalling Voiculescu's theorem on quasi-central approximate units (\cf \cite{Voiculescu, Arveson}).

\begin{lemma}\label{lemma:qc_approx_units}
Let $D_1$ be a $C^*$-algebra and $D_2\subset D_1$ be a closed essential two sided ideal. Let $h\in D_1 $ be a strictly positive element with $\|h\|\le 1$. Let $b\in D_1$ and let $K\subset \cM(D_2)$ be a (norm) compact subset such that $[h,k]\in D_1$ for all $k\in K$; let $\varepsilon >0$.  Let $f_0:[0,1]\to [0,1]$ be a continuous function such that $f(0)=0$. Then there exists $f:[0,1]\to [0,1]$, continuous and such that $f(0)=0$, and $f_0\le f$, $\|b-f(h)b\|< \varepsilon$ and $\|[f(h),k]\|<\varepsilon$ for all $k\in K$. \hfill$\square$
\end{lemma}

The following result is in fact proved by Kasparov in \cite[\S 3, Theorem 4]{Kasparov1981}. Formulated in this way, it further contains many generalizations of the Kasparov product \cite{Kasparov1988, BaajSkandalis, LeGall}. One immediately see that Higson's proof given in \cite{Higson1987} applies, so we omit it.

If $J$ is a closed two sided ideal in a $C^*$-algebra $B$, then $\cM(B;J)=\{x\in \cM(B);\ xB\subset J\}$.

\begin{thm} (\cf \cite[\S 3, Theorem 4]{Kasparov1981}).\label{thm:KaspM}
Let $D_1$ be a separable graded $C^*$-algebra and $D_2$ a graded closed essential two sided ideal in $D_1$. Let $b\in\cM(D_1;D_2)_+$. Let also $A_1$ be a graded $C^*$-subalgebra of $D_1$ containing a strictly positive element of $D_1$ and such that $A_2=A_1\cap D_2$, contains a strictly positive element of $D_2$. Let $K\subset \cM(D_2)$ be a compact subset such that, for every $x\in A_1$ and every $k\in K$, we have $[x,k]\in D_1$. Then there exists $M\in \cM(A_1;A_2)^{(0)}$ such that $0\le M\le 1$, $(1-M)b\in D_2$ and $[M,K]\subset D_2$.\hfill$\square$
\end{thm}

One obtains easily a formulation which encodes many equivariant formulations of the product.

\begin{notation}
 Let $D$ be a separable graded $C^*$-algebra and $A\subset D$ a subalgebra containing a strictly positive element of $D^{(0)}$. 

Let $\cI$ denote the set of graded closed two-sided essential ideals $I$ of $D$ such that $I\cap A^{(0)}$ contains a strictly positive elements of $I$.

Let $D_1,D_2\in \cI$ such that $D_2\subset D_1$. Put $A_i=D_i\cap A$. Denote by $\bE_A(D_1,D_2)$ the set of $F\in \cM(A_2)^{(1)}$ such that:
\begin{center}
for all $x\in D_1$, we have $x(1-F^2)\in D_2$, $x(F-F^*)\in D_2$, $[x,F]\in D_2$.
\end{center}
In other words $(A_2,F)$ is a Kasparov $(A_1,A_2)$ bimodule and $(D_2,F)$
 is a Kasparov $(D_1,D_2)$ bimodule.
\end{notation}

\begin{thm}\label{thm:abstractKKprod}
Let $D_0,D_1,D_2\in \cI$ such that $D_2\subset D_1\subset D_0$. Let $F_1\in \bE_A(D_0,D_1)$ and $F_2\in \bE_A(D_1,D_2)$. Denote by $F_1\sharp F_2=\{F\in \bE_A(D_0,D_2); F-F_2\in \cM(A_1;A_2),\ [F,F_1]\in \cM(A_2)_++A_2\}.$
\begin{enumerate}
\item For every $F_1\in \bE_A(D_0,D_1)$ and $F_2\in \bE_A(D_1,D_2)$ the set $F_1\sharp F_2$ is non empty and path connected. 
\item The path connected component of $F\in F_1\sharp F_2$ in $\bE_A(D_0,D_2)$ only depends on the path connected components of $F_1\in  \bE_A(D_0,D_1)$ and of $F_2\in \bE_A(D_1,D_2)$.
\item (Associativity). Let $D_3\in \cI$ with $D_3\subset D_2$ and $F_3\in \bE_A(D_2,D_3)$. Let $F'_1\in F_1\sharp F_2$, $F'_2\in F_2\sharp F_3$. Then $F'_1\sharp F_3$ and $F_1\sharp F'_2$ are contained in the same path connected component of $\bE_A(D_0,D_3)$.
\end{enumerate}
\begin{proof}
The proof is exactly the same as in the ``classical'' case (\cf \cite{Kasparov1981, ConnesSkandalis, SkandalisJFA}).

For instance, to establish that $F_1\sharp F_2$ is non empty, we take $Q=C^*(D_0,F_1,F_2)$. Let $K$ be a compact subset of $Q$ generating $Q$  as a closed space and let $b$ be a strictly positive element of $Q\cap \cM(D_1,D_2)$.

Apply then Theorem \ref{thm:KaspM}, and put $F=M^{1/2}F_1+(1-M)^{1/2}F_2$.

If we start with paths $F_1^t\in \bE_A(A_0,A_1)$, $F_2^t\in \bE_A(A_1,A_2)$, we just take a bigger algebra: $Q=C^*(A_0,\{F_1^t,F_2^t;\ t\in [0,1]\})$.

The associativity is proved exactly as Lemma 22 in \cite{SkandalisJFA}.
\end{proof}
\end{thm}

We now introduce further notation in order to relate this theorem with equivariant $KK$-theory.

\begin{notation}
 Let $A,\bA$ be separable graded $C^*$-algebras. Let $\varphi,\psi:\cM(A)\to \cM(\bA)$ be two grading preserving strictly continuous morphisms. 

Let $\cJ$ denote the set of closed two-sided essential ideals $I$ of $A$ such that $\varphi(I)\bA=\psi(I)\bA$.

Let $A_1,A_2\in \cJ$ such that $A_2\subset A_1$. Put $\bA_i=\varphi(A_i)\bA$. Denote by $\bE_{\varphi,\psi}(A_1,A_2)$ the set of $F\in \cM(A_2)^{(1)}$ such that:
\begin{enumerate}
\item  for all $x\in A_1$, we have $x(1-F^2)\in A_2$, $x(F-F^*)\in A_2$, $[x,F]\in A_2$, in other words $(A_2,F)$ is a Kasparov $(A_1,A_2)$ bimodule;
\item $(\varphi-\psi)(F)\in \cM(\bA_1;\bA_2)$ (``equivariance property'').
\end{enumerate}
\end{notation}

As an immediate consequence of Theorem \ref{thm:abstractKKprod} we have:

\begin{cor}\label{cor:abstractKKprod}
\begin{enumerate}
\item For every $F_1\in \bE_{\varphi,\psi}(A_0,A_1)$ and $F_2\in \bE_{\varphi,\psi}(A_1,A_2)$ the set $F_1\sharp F_2$ is non empty and path connected. 
\item The path connected component of $F\in F_1\sharp F_2$ only depends on the path connected components of $F_1\in  \bE_{\varphi,\psi}(A_0,A_1)$ and of $F_2\in \bE_{\varphi,\psi}(A_1,A_2)$.
\item (Associativity). Let $A_3\in \cI$ with $A_3\subset A_2$ and $F_3\in \bE_{\varphi,\psi}(A_2,A_3)$. Let $F'_1\in F_1\sharp F_2$, $F'_2\in F_2\sharp F_3$. Then $F'_1\sharp F_3$ and $F_1\sharp F'_2$ are contained in the same path connected component of $\bE_{\varphi,\psi}(A_0,A_3)$.
\end{enumerate}
\begin{proof}

Let $\chi:\cM(A)\to \cM(A\oplus M_2(\bA))$ be the morphism $x\mapsto x\oplus
\begin{pmatrix}
\varphi(x)&0\\
0&\psi(x)
\end{pmatrix}$ and put $D=\chi(A)+(0\oplus M_2(\bA))\subset \cM(A\oplus M_2(\bA))$.

Let $A_1,A_2\in \cJ$ such that $A_2\subset A_1$.  Put $\bA_i=\varphi(A_i)\bA$ and  $D_i=\chi(A_i)+(0\oplus M_2(\bA_i))\subset D$.

We obviously have $\bE_{\varphi,\psi}(A_1,A_2)=\bE_{\chi(A)}(D_1,D_2)$. Therefore, Theorem \ref{thm:abstractKKprod} immediately applies.
\end{proof}
\end{cor}

\begin{exs}
 It is very easy to apply this abstract theorem (Corollary \ref{cor:abstractKKprod}) to many equivariant situations.\begin{enumerate}\renewcommand\theenumii{\arabic{enumii}}
\renewcommand\labelenumii{\rm {\theenumii}.}
\item (\cf \cite{Kasparov1988}) If a second countable locally compact group $G$ acts on separable $C^*$-algebras $A$ and $B$.

An equivariant Kasparov $(A,B)$ bimodule is then a pair $(\cE,F)$ where:\begin{enumerate}
\item  $\cE$ is an $(A,B)$-equivariant Hilbert bimodule; 
\item $F\in \bE_{\varphi,\psi}(A_1,A_2)$ where we have put \begin{itemize}
\item $A_2=\cK(\cE)$ and $A_1=A+A_2$,
\item $\bA_i=C_0(G;A)$;
\item $\varphi,\psi:A_i\to C_b(G;A_i)\subset \cM(\bA_i)$ defined by $\varphi(a)(g)=a$ and $\psi(a)(g)=g.a$.
\end{itemize}
\end{enumerate}

\item (\cf \cite{BaajSkandalis}) Exactly in the same way, if $S$ is a separable Hopf algebra, given a $C^*$-algebra $A$ with an action $\alpha:A\to \cM(A\otimes S)$ of $S$, we just put $\bA=A\otimes S$ and let $\varphi:a\mapsto a\otimes 1$ and $\psi=\alpha$.

\item  (\cf \cite{LeGall}) If $G\overset{s,t}\rightrightarrows G^{(0)}$ is a second countable locally compact groupoid,  given a $C^*$-algebra $A$ with an action $\alpha:s^*A\to t^*A$ of $G$, we just put $\bA=t^*A$ and let $\varphi:a\mapsto t^*a\in\cM(t^*A)$ and $\psi(a)=\alpha(s^*a) $.
\end{enumerate}
\end{exs}

\subsubsection{\texorpdfstring{Kasparov product in $KK_{G}$}{Kasparov product in KK HF}}

Let $G$ be a longitudinally smooth groupoid with atlas $(U_i)_{i\in I}$. We assume that $I$ is countable.

Let $A_0,A_1,A_2$ be $G$-algebras. Let $(\cE_1,F_1)$ and $(\cE_2,F_2)$ be equivariant $(A_0,A_1)$ and $(A_1,A_2)$ cycles. Put $\cE=\cE_1\widehat \otimes _{A_1}\cE_2$. Put $F'_1=F_1\widehat\otimes 1$ and let $F'_2$ be an $F_2$ connection. Put $\check A_2=\cK(\cE)$, $\check A_1=\cK(\cE_1)\widehat \otimes 1+\check A_2$ and $\check A_0=A_0+\check A_1$ (where we denoted by $A_0$ its image in $\cL(\cE)$).

The algebras $\check A_i$ are $G$ algebras, the inclusions $\check A_2\subset \check A_1\subset \check A_0$ are equivariant and the pairs $(\check A_1,F'_1)$ and $(\check A_2,F'_2)$ are equivariant $(\check A_0,\check A_1)$ and $(\check A_1,\check A_2)$ cycles.

Put $(U,q)=\coprod _{j\in I}(U_j,q_j)$. Put $\hat t=t\circ q$ and $\hat s=s\circ q$.

The action of $G$ on $\check A_0$ gives a map $\alpha:\hat s^*(\check A_0)\to \hat t^*(\check A_0)$. Put $\bA_i=\hat t^*(\check A_i)$. Let $\varphi:\check A_0\to \cM(\bA_0)$ be the natural map $\hat t^*$: (defined by $\varphi(x)_{u}=x_{\hat t(u)}$ for all $u\in U$). Let $\psi:\check A_0\to \cM(\bA_0)$ be the composition of $\alpha $ with the map $\hat s^*$. In other words, $\psi(x)_u=\alpha_u(x_{\hat s(u)})$.

Let also $q\in C_0(U)$ be a strictly positive function.

The equivariance condition means exactly that $(\varphi-\psi)(F'_i)\in \cM(\bA_{i-1};\bA_{i})$.

We thus may apply Theorem \ref{thm:abstractKKprod} and obtain the existence of the Kasparov product in $KK_{G}$ with the usual properties:

\begin{thm}
There is a well defined bilinear product $$KK_{G}(A_0,A_1)\times KK_{G}(A_1,A_2)\to KK_{G}(A_0,A_2)$$ which is natural in all $A_i$'s and associative. The element $1$ acts as a unit element. 

Moreover, the Kasparov product is compatible with the descent morphisms.\hfill$\square$
\end{thm}

\section{A Baum-Connes conjecture for the telescopic algebra}
\label{sec:BCtelescope}

In this section, we construct the Baum-Connes map for telescopic algebra of a nicely decomposable singular foliation. 

\subsection{An abstract construction}

\subsubsection{Setting of the problem}


Let $\cF$ be a nicely decomposable foliation. We keep the notation of section \ref{subsec:telescope}. We put $\cG'_k=(\cG_k)_{|\Omega_{k-1}\cap W_k}$.

There is a priori a ``left hand side'' for the $K$-theory of the Lie groupoid $C^*$-algebras $C^{\ast}(\cG_k)$ and $C^{\ast}(\cG'_k)$. In order to construct a left hand side and a Baum-Connes map for $C^*(\cT(q,j))$, we first wish to understand the morphisms and mapping cones associated with the morphisms $j_k$ and  $q_k$ at the left hand side level.

Let us first note that the morphism $j_k$ is just the inclusion $\cG'_k\subset \cG_k$ of the restriction of $\cG_k$ to the (saturated) open subset $\Omega_{k-1}\cap W_k$ of $W_k$. The mapping cone of such a morphism is just the $C^*$-algebra of a Lie groupoid (restriction of $\cG_k\times [0,1)$ to the open subset $(\Omega_{k-1}\cap W_k)\times [0,1)\cup W_k\times (0,1)$). We may then very easily construct a left hand side for it.

On the other hand, the morphism $q_k$ corresponds to a groupoid homomorphism, which is the identity at the level of objects ($W_k$) and a surjective submersion at the level of arrows. The corresponding map at the level of left hand sides is not as easy. Let us also note that, even knowing the map $(q_k)_*^{top}$ at the level of $K_*^{top}$, we need more in order to construct the left hand side for the mapping cone: this morphism only gives a short exact sequence $$0\to \coker (q_k)_*^{top}\to K_*^{top}(\cF)\to  \ker (q_k)_*^{top}\to 0$$ which is not sufficient in order to determine the group $K_*^{top}(\cF)$ that we are looking for.

In order to understand the $K$-theory of this mapping cone, one needs in fact to construct $(q_k)_* ^{top}$ as a $KK$-element. To do so, we need to write explicitly the topological $K$-groups as $K$-groups of $C^*$-algebras and the Baum-Connes maps as $KK$-elements. To that end we will assume that:

\begin{enumerate}\renewcommand\theenumi{\roman{enumi}}
\renewcommand\labelenumi{(\rm {\theenumi})}

\item  The Lie groupoids $\cG_k$ are Hausdorff. 

\item \label{condition:ii}The classifying spaces for proper actions of these groupoids are smooth manifolds. This is always the case when the groupoids $\cG_k$ are given by (connected) Lie group actions - or are Morita equivalent to those. This is indeed the case in most of singularity height one examples \ref{exs:basic} above - in fact, also in the examples of higher height given in section \ref{sec:exlengthk}.
\end{enumerate}

It turns out that condition (\ref{condition:ii}) can be somewhat bypassed (thanks to the Baum-Douglas presentation of $K^*_{top}$ (see \cite{BaumDouglas1, BaumDouglas2, BaumConnes, Tu})). We will discuss this in appendix \ref{app:Enotmanifold}.

\subsubsection{Submersions of Lie groupoids and left hand sides}
\label{subsubsec:Subm}
Let us recall some facts about the Baum-Connes map for groupoids (\cf \cite{BaumConnes, Tu}).

Let $G$ be a a Lie groupoid. If the classifying space for proper actions is a manifold $M$, then there is no inductive limit to be taken, and replacing if necessary $M$ by the total space of the vector bundle $(\ker dp)^*$, we may assume that the equivariant submersion $p:M\to G^{(0)}$ is $K$-oriented and then the left hand side is $K_*(C_0(M)\rtimes G)$ and the Baum-Connes map is just the wrong-way functoriality element $\widehat{p_!}\in KK(C_0(M)\rtimes G,C^*(G))$ constructed in \cite{ConnesSkandalis, HilsumSkandalis}. 

In Le Gall's equivariant $KK_G$ theory and terminology (\cite{LeGall}, see also Kasparov \cite{Kasparov1988}), the Baum-Connes assembly map is the element $\widehat{p_!}=j_G(p_!)$, where $p_!$ is the element of $KK_G(C_0(M),C_0(G^{(0)}))$ associated with the $G$ equivariant $K$-oriented smooth map $p$.

Before we proceed and construct a ``left hand side'' for the telescopic algebra, we examine the case of a morphism $\pi:G_0\to G_1$ of Hausdorff Lie groupoids $G_i\overset{t_i,s_i}\gpd G_i^{(0)}$ ($i=0,1$). We assume that $\pi$ is a submersion and that it is an inclusion of an open subset $\pi:G_0^{(0)}\subset G_1^{(1)}$ at the level of units.

Let $p_i:M_i\to G_i^{(0)}$ be smooth manifolds which are classifying spaces for proper actions for $G_i$. We will assume further that the $p_i$'s are $K$-oriented submersions and that the dimensions of the fibers are even.
\begin{itemize}
\item Let $W=M_0\times _{p_0,t_1}  G_1$. The groupoid $G_0$ acts properly on $W$; we thus obtain a Hausdorff locally compact quotient $W/G_0=M_0\times _{G_0}  G_1$. Note that $x\mapsto (x,\pi(p_0(x))$ defines a continous map from $M_0$ to $W$ and therefore $M_0\to M_0\times _{G_0}  G_1$.
\item The groupoid $G_1$ acts properly on the quotient space $M_0\times _{G_0}  G_1$. Since $M_1$ is universal, we obtain a $G_1$-equivariant map $M_0\times _{G_0}  G_1\to M_1$. Whence, by composition we have a $G_0$ equivariant map $q:M_0\to M_1$. As $p_1\circ q=p_0$, we obtain a morphism of proper groupoids $$q:M_0\rtimes G_0\to M_1\rtimes G_1$$ 
\item The map $q$ is naturally $K$-oriented, so it induces an element $q_!\in KK_{G_0}(C_0(M_0),C_0(M_1))$. Applying the descent map $j_{G_0}$ we obtain an element $$\widehat{q_!} = \tilde \pi_*(j_{G_0}(q_!)) \text{ in } KK(C_0(M_0)\rtimes G_0,C_0(M_1)\rtimes G_1)$$ where $\tilde \pi$ is the morphism $C_0(M_1)\rtimes G_0\to C_0(M_1)\rtimes G_1$ induced by the morphism $\pi$.
\end{itemize}

\begin{prop}\label{prop:lhspi}
The morphism $\pi:C^*(G_0)\to C^*(G_1)$ corresponds at the level of left hand side to the element $\widehat{q_!}$. More precisely, we have $\pi_{*}(\widehat{(p_0)_!}) =\widehat{q_!}\otimes \widehat{(p_1)_!}$.
\end{prop}
\begin{proof}
The morphism $p_1$ being $G_1$ equivariant, is also $G_0$-equivariant, (where $G_0$ acts through the morphism $\pi$). It gives rise to an element $\widecheck {(p_1)_!}\in KK(C_0(M_1)\rtimes G_0,C^*(G_0))$. The elements $\widecheck {(p_1)_!}$ and $\widehat {(p_1)_!}$ correspond to each other via the morphism $\pi:G_0\to G_1$, \ie we have $\pi_{*}(\widecheck {(p_1)_!})=\tilde \pi^*(\widehat {(p_1)_!})$.
Now, 
$$\begin{array}{cclc}
 \widehat{q_!}\otimes \widehat{(p_1)_!}&=&\tilde \pi_*(j_{G_0}(q_!))\otimes j_{G_1}((p_1)_!)\\
 &=&j_{G_0}(q_!)\otimes \tilde \pi^*(\widehat {(p_1)_!})&{(\clubsuit)}\\
 &=&j_{G_0}(q_!)\otimes \pi_{*}(\widecheck {(p_1)_!})\\
& =&\pi_{*}\big(j_{G_0}(q_!)\otimes j_{G_0}((p_1)_!)\big)&{(\clubsuit)}\\
& =&\pi_{*}(j_{G_0}(q_!\otimes (p_1)_!))&{(\clubsuit)}\\
&=&\pi_{*}(j_{G_0}((p_0)_!))&{(\diamondsuit)}
\end{array}$$
\begin{itemize}
\item Equalities $(\clubsuit)$ follow from the functoriality properties of the Kasparov product (\cf \cite{Kasparov1981, Kasparov1988, LeGall}).
\item Equality $(\diamondsuit)$ follows from the wrong way functoriality (\cf \cite{ConnesSkandalis, HilsumSkandalis}). Note that, since the groupoid $M_0\rtimes G_0$ is proper, the $\gamma$ obstruction appearing in this computation in \cite{HilsumSkandalis} vanishes.
\qedhere
\end{itemize}
\end{proof}

\subsubsection{Abstract left hand sides for mapping cones}\label{subsec:BCabstract}

Next, we wish to construct in a natural way the left hand side for the mapping cone of the morphism $\pi_{C^*}:C^*(G_0)\to C^*(G_1)$. We saw in prop. \ref{prop:lhspi} that the left hand side of $\pi$ is an element in $KK(C_0(M_0)\rtimes G_0,C_0(M_1)\rtimes G_1)$. The left hand side of the cone of $\pi$ should be a kind of ``mapping cone of this $KK$-element''. 

In this section, we abstractly construct this mapping cone up to $KK$-equivalence. We will give an explicit description of this left hand side (section \ref{subsec:lhscones}) and of the Baum-Connes map (section \ref{section:BCmap}) below.

Recall that a $KK$-element $x\in KK(A,B)$ can be given as a composition $$x=[f]^{-1}\otimes [g]\eqno (\spadesuit)$$
of a morphism $g:D\to B$ with an element $[f]^{-1}$ which is the $KK$-inverse of a morphism $f:D\to A$ which is invertible in $KK$-theory (\cf \cite[App. A]{VLHO}). We may then wish to define (up to $KK$-equivalence) the cone of $x$ as being the cone of $g$. 

Next, in order to understand the Baum-Connes map, we should construct a $KK$-element associated with a map between mapping cones. Let us state the following probably classical and quite obvious lemma.

\begin{lemma}\label{lem:trivlemma}
 Let $f_i:A_i\to B_i$ be morphisms of $C^*$-algebras ($i=0$ or $1$). Denote by $p_i:C_{f_i}\to A_i$ and $j_i:B_i(0,1)\to C_{f_i}$ the natural maps ($p_i(a_i,\phi)=a_i$ and $j_i(\phi)=(0,\phi)$). Let $x\in KK(A_0,A_1)$ and $y\in KK(B_0,B_1)$ satisfy $(f_1)_*(x)=f_0^*(y)$\begin{enumerate}
\item There exists $z\in KK(C_{f_0},C_{f_1})$ such that $(p_1)_*(z)=p_0^*(x)$ and $(j_1)_*(Sy)=j_0^*(z)$ where $Sy\in KK(B_0(0,1),B_1(0,1))$ is deduced from $y$.
\item If $x$ and $y$ are invertible, then so is $z$.
\end{enumerate}
\end{lemma}
\begin{proof}
\begin{enumerate}
\item Note that $z$ is not a priori unique. To construct it, one needs in fact to be more specific. Fix Kasparov bimodules $(E_A,F_A)$ representing $x$ and $(E_B,F_B)$ representing $y$; a Kasparov $(A_0,B_1[0,1])$ bimodule  $(E',F')$ realizing a homotopy between $(E_A\otimes _{A_1}B_1,F_A\otimes 1)$ and $f_0^*(E_B,F_B)$ gives rise to a Kasparov $(A_0,Z_{f_1})$, where $Z_{f_1}=\{(a_1,\phi)\in A_1\otimes B_1[0,1];\ f_1(a_1)=\phi(0)\}$ is the mapping cylinder of $f_1$, which can be glued with $(E_B,F_B)[0,1)$ to give rise to the desired element in $KK(C_{f_0},C_{f_1})$.
\item By (a) applied to $x^{-1}$ and $y^{-1}$, there exists $z'\in KK(C_{f_1},C_{f_0})$ such that $(p_1)_*(z')=p_1^*(x^{-1})$ and $(j_0)_*(Sy^{-1})=j_1^*(z')$. The Kasparov products $u_0=z\otimes z'$ and $u_1=z'\otimes z$ are elements in $KK(C_{f_i},C_{f_i})$ such that $(p_i)_*(1-u_i)=0$ and $j_i^*(1-u_i)=0$. From the first equality and the mapping cone exact sequence, it follows that there exists $d_i\in KK(C_{f_i},B_i(0,1))$ such that $1-u_i=(j_i)_*(d)$, and it follows that $(1-u_i)^2=(j_i)_*(d)\otimes (1-u_i)=d\otimes j_i^*(1-u_i)=0$, whence $u_i$ is invertible.
\qedhere
\end{enumerate}
\end{proof}
 
\begin{remark}\label{rem:commutK}
Note also that we have a diagram
\[
\xymatrix{
K_i(A_0) \ar[r]\ar[d]^{\otimes x}  &K_i(B_0) \ar[r]\ar[d]^{\otimes y}  &K_{1-i}(C_{f_0}) \ar[r]\ar[d]^{\otimes z}&K_{1-i}(A_0) \ar[r]\ar[d]^{\otimes x}  &K_{1-i}(B_0)\ar[d]^{\otimes y}  \\
K_i(A_1) \ar[r] &K_i(B_1) \ar[r] &K_{1-i}(C_{f_1}) \ar[r]&K_{1-i}(A_1) \ar[r]  &K_{1-i}(B_1)
}
\]
where the lines are exact (Puppe sequences) and the squares commute. It follows that if $x$ and $y$ induce isomorphisms in $K$-theory, then the same holds for $z$.
\end{remark}

\begin{remarks}\label{rem:trivrem}
\begin{enumerate}
\item It follows easily from this construction that given an element $x\in KK(A,B)$ the mapping cone $C_g$ does not depend on the decomposition $(\spadesuit)$ up to $KK$-equivalence. 
\item An alternative (and equivalent) way to construct the $K$-theory of the mapping cone of the $KK$-element $x$ is to write $x$ as an extension $0\to SB\otimes \cK\to D\to A\to 0$ and define this $K$-theory as being $K_*(D)$.
\item One can also define the $KK$-theory of this mapping cone as a relative $KK$-group \cite[Remark 3.7.c)]{Sk}. 
\end{enumerate}
\end{remarks}

\subsection{Baum-Connes map for mapping cones of submersions of Lie groupoids}\label{subsec:lhscones}

Let us come back to our morphism $\pi :G_0\to G_1$ of Hausdorff Lie groupoids which is assumed to be a submersion and an open inclusion at the level of objects. We assume that the classifying spaces for proper maps of $G_i$ are manifolds $M_i$. In section \ref{subsubsec:Subm}, we explained how to construct an equivariant map $q:M_0\to M_1$ that can be assumed to be a smooth submersion (up to replacing $M_0$ by a homotopy equivalent manifold).

As a consequence of Lemma \ref{lem:trivlemma} and Proposition \ref{prop:lhspi}, we see that, in order to construct the left hand side we need to give an explicit construction of the wrong-way functoriality element $\widetilde{\pi}_{\ast}(j_{G_0})(q_{!})\in KK(C^*(M_0\rtimes G_0),C^*(M_1\rtimes G_1))$. Here, using a double deformation longitudinally smooth groupoid we will give a groupoid $\bH$ which is a family over $[0,1]\times [0,1]$, whose vertical lines $\{i\}\times [0,1]$ can be interpreted as the Baum-Connes maps for the groupoid $G_i$ and whose horizontal lines $[0,1]\times \{0\}$ and $[0,1]\times \{1\}$ are respectively $q!$ and $[\pi]$.

We then may define the left hand side of $\pi$ as the groupoid $\bH$ restricted to $[0,1)\times \{0\}$ and construct the Baum-Connes map using the groupoid $\bH$ (restricted to $[0,1)\times [0,1]$).

In order to have a ``ready to glue'' groupoid, in view of  the case of the telescopic algebra (section \ref{subsec:BCtelescope}), we are lead to perform a slightly more complicated construction.




\subsubsection{Some ``classical'' constructions with groupoids}

Before proceeding to explain this construction, we recall some constructions based on Lie groupoids that we will use.

\paragraph{Pull-back groupoid.} Let $G\overset {t,s}\gpd G^{(0)}$ be a Lie groupoid, $M$ a smooth manifold and $q:M\to G^{(0)}$ a smooth submersion. The pull back groupoid $G_q^q$ is a subgroupoid $$G_q^q=\{(x,\gamma,y)\in M\times G\times M;\  q(x)=t(\gamma)\ \hbox{and}\ q(y)=s(\gamma)\}$$ of the product groupoid of $G$ with the pair groupoid $M\times M$. As $q$ is supposed to be a submersion, $G_q^q$ is a Lie groupoid (actually, a transversality condition suffices). If $q(M)$ meets all the $G$-orbits, the groupoids $G$ and $G_q^q$ are canonically Morita equivalent.

\paragraph{Actions on spaces.} Recall  that an action of a groupoid $G\overset {t,s}\gpd G^{(0)}$ on a space $X$ is given by a map $p:X\to G^{(0)}$ and the action $G\times_{s,p}X\to X$ denoted by $(\gamma,x)\mapsto \gamma.x$ with the requirements $p(\gamma.x)=t(\gamma),\ \gamma.(\gamma'.x)=(\gamma\gamma').x$ and $u.x=x$ if $u=p(x)$.

\paragraph{Semi-direct product.} If a groupoid $G$ acts on a space $X$, we may form the semi-direct product groupoid $X\rtimes G$: 
\begin{itemize}
\item as a set  $X\rtimes G=X\times _{t}G=\{(x,\gamma)\in X\times G; \ t(\Gamma)=x\}$;
\item $(X\rtimes G)^{(0)}=X$; we have $t(x,\gamma)=x$ and $s(x,\gamma)=\gamma^{-1}.x$;
\item the elements $(x,\gamma)$ and $(y,\gamma')$ are composable if $x=\gamma y$; the composition is then  $(x,\gamma)(y,\gamma')=(x,\gamma \gamma')$.
\end{itemize}
When $p$ is a submersion, $X\rtimes G$ is a Lie groupoid: it is the closed subgroupoid $\{(x,\gamma,y)\in G_p^p;\ x=\gamma.y\}$ of $G_p^p$.

\paragraph{Actions on groupoids.} This construction can be generalized: If $X\overset {t_X,s_X}\gpd X^{(0)}$ is a groupoid, we say that the action is by groupoid automorphisms (\cf \cite{RBrown}) if $G$ acts on $X^{(0)}$ through a map $p_0:X^{(0)}\to G^{(0)}$, we have $p=p_0\circ t_X=p_0\circ s_X$ and $\gamma.(xy)=(\gamma .x)(\gamma.y)$. There is a semi-direct product construction in this generalized setting.

\paragraph{Deformation to the normal cone.} The adiabatic deformation of a Lie groupoid $G$ with Lie algebroid $\gG$ was defined by Alain Connes in the particular case of the pair groupoid (\cf \cite{ConnesNCG}) and generalized by various authors (\cite{HilsumSkandalis, MonthPie, NWX}...). This is based to the notion of deformation to the normal cone that we briefly recall  (see also \cite{PCR, DS1}...).

Let $X$ be a submanifold of a manifold $Y$. Denote by $N_X^Y$ the total space of the normal bundle to $X$ in $Y$. There is a natural way to put a manifold structure to $Y\times \R^*\cup N_X^Y\times \{0\}$; denote this manifold by $\DNC(Y,X)$. 

The map $p:\DNC(Y,X)\to \R$ defined by $p(y,t)=t$ for $(y,t)\in Y\times \R^*$ and $p(\xi,0)=0$ for $\xi\in N_X^Y$ is a smooth submersion. For $J\subset \R$, we put  $\DNC_J(Y,X)=p^{-1}(J)$.

This construction is functorial: Given a commutative diagram of smooth maps $$\xymatrix{X\ar@{^{(}->}[r]\ar[d]_{f_X}&Y\ar[d]^{f_Y}\\X'\ar@{^{(}->}[r]&Y'}$$ where the horizontal arrows are inclusions of submanifolds, we naturally obtain a smooth map $\DNC(f):\DNC(Y,X)\to \DNC(Y',X')$. If $f_Y$ is a submersion and $X=X'\times _{Y'}Y$ then $\DNC(f)$ is a submersion.

\paragraph{Double deformations to the normal cone}

Let $Z$ be a smooth manifold, $Y$ a (locally) closed submanifold of $Z$ and $X$ a (locally) closed submanifold of $Y$. Then $\DNC(Y,X)$ is a (locally) closed submanifold of $\DNC(Z,X)$. 

Put then $$\DNC^2(Z,Y,X)=\DNC(\DNC(Z,X),\DNC(Y,X))$$

We have a submersion  $p_2:\DNC^2(Z,Y,X)\to \R^2$.  For every subset $L$ of $\R^2$, we put $$\DNC_L^2(Z,Y,X)=p_2^{-1}(L)$$

 By definition of the deformation to the normal cone $\DNC^2_{\R\times \R^*}(Z,Y,X)=\DNC(Z,X)\times \R^*$.
 
 By functoriality of the $\DNC$ construction,  $\DNC^2_{\R^*\times \R}(Z,Y,X)=\DNC(Z\times \R^*,Y\times \R^*)\simeq\DNC(Z,Y)\times \R^*$.

\paragraph{Deformation groupoids, adiabatic groupoids.}

From naturality, it follows that if $Y$ is a Lie groupoid and $X$ is a Lie subgroupoid of $Y$, then $\DNC(Y,X)$ is naturally endowed with a Lie groupoid structure - with objects $\DNC(Y^{(0)},X^{(0)})$, and target and source maps $\DNC(t)$ and $\DNC(s)$. Of course, if in the above diagram all the maps are groupoid morphisms, then $\DNC(f)$ is a morphism of groupoids too.

The \emph{adiabatic groupoid} of a Lie groupoid $G$ is just  $G_{ad}=\DNC_{[0,1)}(G,G^{(0)})=\gG\times \{0\}\cup G\times (0,1)$ (with base manifold $G^{(0)} \times [0,1)$). Note that the normal bundle $N_{G^{(0)}}^G$ is, by definition, the Lie algebroid $\gG$ of $G$.

It follows that, $Z$ is a Lie groupoid, $Y$ is a Lie subgroupoid of $Z$ and $X$ a Lie subgroupoid of $Y$, then $\DNC^2(Z,Y,X)$ is a Lie groupoid.

\subsubsection{The Baum-Connes map of a Lie groupoid via deformation groupoids}

Let $G$ be a Lie groupoid and let $M$ be a smooth manifold on which $G$ acts via a smooth onto submersion $p:M\to G^{(0)}$. We will not assume that $p$ is $K$ oriented but rather consider the total space of $(\ker dp)^*$. Note that, if $M$ is a classifying space for proper actions of $G$, then $(\ker dp)^* \to G^{(0)}$ is also a classifying space of proper actions, and it moreover carries a canonical $K$-orientation. So we can replace $M$ with $(\ker dp)^*$.

Put then $\Gamma_p=\DNC(G_p^p,M\rtimes G)$. As $p$ is supposed to be a surjective submersion, the groupoid $G_p^p$ is Morita equivalent to $G$. There is a canonical Morita equivalence bimodule $\cE$ of the $C^*$-algebras $C^*(G_p^p)$ and $C^*(G)$.

We have an exact sequence of $C^*$-algebras:
$$0\to C^*(G_p^p\times (0,1])\longrightarrow C^*((\Gamma_p)_{[0,1]})\overset{ev_0}{\longrightarrow}C^*(\ker (dp)\rtimes G)\to 0.$$
Note that $C^*(G_p^p\times (0,1])$ is contractible. 
It follows that $ev_0$ is invertible in $E$-theory.

We may then observe the diagram:
$$\xymatrix{C_0((\ker dp)^*)\rtimes G&C^*((\Gamma_p)_{[0,1]})\ar[l]_{\qquad ev_0}\ar[r]^{ev_1}&C^*(G_p^p)\ar@{-}[r]^{\quad\cE}&C^*(G)}$$
We thus obtain an element $\mu_M =[ev_0]^{-1}\otimes [ev_1]\otimes [\cE]\in E(C_0((\ker dp)^*)\rtimes G,C^*(G))$. Note that this $E$-theory coincides with $KK$-theory if the action of $G$ on $M$ is assumed to be amenable - and in particular, if it is proper.

If $M$ is the classifying space for proper algebras, the morphism on $K$-groups defined by $\mu_M$ is the Baum-Connes map.

\subsubsection{A double deformation construction}

Let now $G_0$ and $G_1$ be two Lie groupoids and let $\pih :G_0\to G_1$ be a groupoid morphism which is a smooth submersion whose restriction $\pih^{(0)}:G_0\to G_1$ is the inclusion of an open subset. Let $M_i$ be manifolds with actions of $G_i$. We assume that the maps $p_i:M_i\to G_i^{(0)}$ defining these actions are smooth submersions. Let also $q:M_0\to M_1$ be a smooth submersion which is equivariant, \ie $q(\gamma .x)=\pih(\gamma) q(x)$ for every $(x,\gamma)\in M_0\times _sG_0$. In other words, we assume that we have a morphism of semi-direct products $\widehat{\pih}:M_0\rtimes G_0\to M_1\rtimes G_1$ defined by $\widehat{\pih} (x,\gamma)=(q(x),\pih(\gamma))$. 

\medskip The groupoid $G_0$ acts on the open subspace $M'_1=q(M_0)$ of $M_1$ through the morphism $\pih$: just put $\gamma.q(x)=q(\gamma.x)=\pi(\gamma).q(x)$ for $x\in M_0$ and $\gamma\in G_0$ with $s(x)=p_0(x)=p_1(q(x))$.

\medskip We have inclusions of groupoids $M_0\rtimes G_0\subset (M_1'\rtimes G_0)_q^q\subset (G_0)_{p_0}^{p_0}.$
Indeed, $M_0\rtimes G_0=\{(x,\gamma,y)\in (G_0)_{p_0}^{p_0};\ x=\gamma .y\}$ and $(M_1'\rtimes G_0)_q^q=\{(x,\gamma,y)\in (G_0)_{p_0}^{p_0};\ q(x)=q(\gamma .y)\}$.

Let then $\cH_0$ be the double deformation Lie groupoid:
 $$\cH_0=\DNC^2\big((G_0)_{p_0}^{p_0},(M_1'\rtimes G_0)_q^q,M_0\rtimes G_0\big).$$
The groupoid $\cH_0$ is a family of groupoids indexed by $\R^2$. For every locally closed subset $L$ of $\R^2$, we may form the locally compact groupoid $(\cH_0)_L$.

 

Let $q':M_0\coprod M_1\to M_1$ be the map which coincides with $q$ on $M_0$ and the identity on $M_1$ and $p':M_0\coprod M_1\to G_1^{(0)}, p'=p_1\circ q'$. Define the groupoid  $$\cH_1=\DNC((G_1)_{p'}^{p'}\times \R^*,(M_1\rtimes G_1)_{q'}^{q'}\times \R^*)\simeq \DNC((G_1)_{p'}^{p'},(M_1\rtimes G_1)_{q'}^{q'})\times \R^*$$  (with objects $(M_0\coprod M_1)\times \R^*\times \R$).

For every locally closed subset $Y\subset \R^*\times \R$ we denote by $(\cH_1)_Y$ the restriction of $\cH_1$ to its saturated subset $(M_0\coprod M_1)\times Y$.

\subsubsection{A longitudinally smooth groupoid}

Note that $(M_1'\rtimes G_0)_q^q=\{(x,\gamma,y)\in (G_0)_p^p;\ q(x)=\pi(\gamma)q(y)\}$. In other words, $(M_1'\rtimes G_0)_q^q$ is the fibered product $(G_0)_{p_0}^{p_0}\times_{(G_1)_{p_0}^{p_0}}(M_1\rtimes G_1)_q^q$.

We therefore have a commutative diagram: $$\xymatrix{(M_1'\rtimes G_0)_q^q\ar@{^{(}->}[r]\ar[d]&(G_0)_{p_0}^{p_0}\ar[d]\\(M_1\rtimes G_1)_q^q\ar@{^{(}->}[r]&(G_1)_{p_0}^{p_0}}$$
which gives rise to a morphism $$\DNC((G_0)_{p_0}^{p_0},(M_1'\rtimes G_0)_q^q)\to \DNC((G_1)_{p_0}^{p_0},(M_1\rtimes G_1)_q^q)$$ which is a groupoid morphism, a submersion and the identity at the level of objects ($M_0\times \R$). We thus obtain a a morphism of groupoids $$\pi:(\cH_0)_{\R^*\times \R}\to \cH_1$$ which is a submersion. At the level of objects it is the inclusion $M_0\times \R^*\times \R\to (M_0\coprod M_1)\times \R^*\times \R$.



Let $Z_0=[0,1)\times [0,1/2]$, $Q=\{(u,v)\in \R^2;\ 0\le u\le v\le 1/2\}$ and $Z_1= Z_0\setminus Q=\{(u,v)\in Z_0;\ u>v\}$.

We may then construct a longitudinally smooth groupoid $\bH=(\cH_0)_Q\cup (\cH_1)_{Z_1}$ with atlas formed by $(\cH_0)_{Z_0}$ and $(\cH_1)_{Z_1}$, using the morphism $\pi$ in order to map $(\cH_0)_{Z_1}$ to $(\cH_1)_{Z_1}$. We have $\bH^{(0)}=M_0\times Z_0\coprod M_1\times Z_1$.

In the same way as above, for every locally closed subset $Y\subset Z_0$ we denote by $\bH_Y$ the restriction of $\bH$ to its saturated subset $M_0\times Y\cup M_1\times (Y\cap Z_1)$.

\begin{remarks}\label{remark:gluing}
\begin{enumerate}
\item It is worth noticing that the groupoid $\bH$ only depends on $\pih :G_0\to G_1$, the (proper) actions of $G_i$ on $M_i$ and the submersion $q$. Also, the restriction $\bH_{\{1/2\}\times [0,1/2]}$ is nothing else than $(M_0\times _{p_0}M_0)_{ad}\rtimes G_0$ (restricted to $[0,1/2]$). It does not depend on $G_1,M_1,q$.
\item Note that $\cH_{0,0}$ is isomorphic to the direct sum of vector bundles $\ker dq\oplus  q^*(\ker dp_1)$.
\end{enumerate}
\end{remarks}

\subsubsection{Baum-Connes map for a mapping cone}\label{section:BCmap}

Set $F_0=[0,1)\times \{0\}\cup \{0\}\times [0,1/2]$.

Note that, since the action of $G_i$ on $M_i$ is proper, the groupoid $\bH_{F_0}=(\ker dp_0)\rtimes G_0\times [0,1/2]\cup ((\ker dp_1)\rtimes G_1)_{q'}^{q'}\times (0,1)$ is amenable and we have a semi-split exact sequence $$0\to C^*(\bH_{Z_0\setminus F_0})\longrightarrow C^*(\bH)\overset {\sigma_0}\longrightarrow C^*(\bH_{F_0})\to 0.$$

\begin{prop}
The homomorphism $\sigma_0$ is invertible in $KK$-theory.
\end{prop}
\begin{proof}
We have a semi-split exact sequence $$0\to C^*(\bH_{Z_1\setminus F_0})\longrightarrow C^*(\bH_{Z_0\setminus F_0}) \longrightarrow C^*(\bH_{Q\setminus F_0})\to 0.$$
Note that the groupoid $\bH$ is constant over the sets $Z_1\setminus F_0$ and $Q\setminus F_0$: \begin{itemize}
\item $\bH_{(u,v)}=(G_1)^{p_1\circ q'}_{p_1\circ q'}$ for $(u,v)\in Z_1\setminus F_0$
\item $\bH_{(u,v)}=(G_0)_{p_0}^{p_0}$ for $(u,v)\in Q\setminus F_0$.
\end{itemize}

The sets $Z_1\setminus F_0=\{(u,v);\ 0<v<u< 1 \ \hbox{and}\ v\le 1/2\}$ and $Q\setminus F_0=\{(u,v);\ 0<u\le v\le 1/2\}$ are contractible (more precisely, their one point compactification contracts to this point) and it follows that the $C^*$-algebras $C^*(\bH_{Z_1\setminus F_0})$ and $C^*(\bH_{Q\setminus F_0})$ are contractible. It follows that $C^*(\bH_{Z_0\setminus F_0})$ is $K$-contractible (it is actually contractible). We deduce that $[\sigma_0]$ is a $KK$-equivalence.
\end{proof}

Set also $F_1=\{1/2\}\times [1/2,1)$. One sees that $\bH_{F_1}$ is isomorphic to the groupoid $\cC_{\pih}=G_0\times \{0\}\cup G_1\times (0,1)$ pulled back by $q'':M_0\times [0,1)\coprod M_1\times (0,1)\to G_1^{(0)}\times [0,1)$ (recall that $G_0^{(0)}$ is an open subset of $G_1^{(0)}$).

\begin{cor}
The algebra $C^*(\bH_{F_1})$ is canonically Morita equivalent to the mapping cone of $h_{C^*}:C^*(G_0)\to C^*(G_1)$. 
\end{cor}

Denote by $\cE$ the Morita $C^*(\bH_{F_1}),C^*(\cC_{h})$ bimodule and $[\cE]$ its $KK$-class. Let $[\sigma_1]:C^*(\bH)\to C^*(\bH_{F_1})$ be the evaluation. 

\begin{definition} 
Assume further that the manifolds $M_i$ are classifying spaces for the proper actions of $G_i$. With the notation above, the topological $K$-theory of the groupoid $\cC_{\pih}$ is $K_*(C^*(\bH_{F_0}))$ and the Baum-Connes morphism is the composition $[\sigma_0]^{-1}\otimes [\sigma_1]\otimes [\cE]$.
\end{definition}

\subsubsection{Justifying why this is a Baum-Connes map}

Let us explain why this is a ``good'' definition.

First of all, for $v\in [0,1/2]$, the $K$-theory of the $C^*$-algebra $C^*(\bH_{(0,v)})=C^*(\ker p_0\rtimes G_0)=C_0((\ker p_0)^*)\rtimes G_0$ is the left hand side for $G_0$.

Also, for $u\in (0,1)$, the   $C^*$-algebra $C^*(\bH_{(u,0)})=C^*((\ker p_1\rtimes G_1)_{q'}^{q'})$ is Morita equivalent to $C^*(\ker p_1\rtimes G_1)=C_0((\ker p_1)^*)\rtimes G_1$ whose $K$-theory is the left hand side for $G_1$. 

We may then write a diagram:

$$\xymatrix{0\ar[r] &C^*(\bH_{Z_1\cap F_0})\ar[r] &C^*(\bH_{F_0})\ar[r]&C^*(\bH_{Q\cap F_0})\ar[r]&0\\
0\ar[r] &C^*(\bH_{Z_1})\ar[r]\ar[u]^{\sigma_{1,0}}\ar[d]_{\sigma_{1,1}} &C^*(\bH)\ar[r]\ar[u]^{\sigma_{0}}\ar[d]_{\sigma_{1}}&C^*(\bH_{Q})\ar[r]\ar[u]^{\sigma_{0,0}}\ar[d]_{\sigma_{0,1}}&0\\
0\ar[r] &C^*(\bH_{Z_1\cap F_1})\ar[r]\ar@{-}[d]^{\cE_1} &C^*(\bH_{F_1})\ar[r]\ar@{-}[d]^{\cE}&C^*(\bH_{Q\cap F_1})\ar[r]\ar@{-}[d]^{\cE_0}&0\\
0\ar[r] &C^*(G_1)(0,1)\ar[r]&\cC_{h_{C^*}}\ar[r] &C^*(G_0)\ar[r]&0
}$$
In this diagram all sequences are semi split, the morphisms $\sigma_0,\sigma_{i,0}$ are $KK$-equivalences and the compositions $[\sigma_{i,0}]^{-1}\otimes [\sigma_{i,1}]\otimes [\cE_i]$ are indeed the Baum-Connes maps for $G_1\times (0,1)$ and $G_0$.

It follows also that the class in $KK^1(C^*(\bH_{Q\cap F_0}),C^*(\bH_{Z_1\cap F_0}))$  for the first sequence corresponds to the class of $[h_{C^*}]\in KK(C^*(G_0),C^*(G_1))$.

From the discussion in section \ref{subsec:BCabstract}, it follows that the $K$-theory of $C^*(\bH_{F_0})$ and the morphism is indeed the right one, and that the composition $[\sigma_0]^{-1}\otimes [\sigma_1]\otimes [\cE]$ is indeed a Baum-Connes map.

\begin{remark}
The groupoid $\bH_{F_0}$ is a semi-direct product $\Lambda\rtimes \cC_{\pih}$, where $\Lambda $ is a groupoid obtained by gluing $\DNC_{[0,1)}(M_0,(\ker dp_1)_{q'}^{q'})$ with $\ker dp_0\times [0,1/2]$.

One may give a generalized notion of proper algebras on a longitudinally smooth groupoid $G$ by saying that $G^{(0)}$ is an increasing union $\bigcup \Omega_k$ of saturated open subsets such that the restriction of $G$ to $\Omega_k\setminus \Omega_{k-1}$ is Hausdorff. We may say that an action of $G$ on an algebra $A$ is proper if its restriction to each $\Omega_k\setminus \Omega_{k-1}$ is proper.

In this generalized sense, the $\cC_{\pih}$-algebra  $C^*(\Lambda)$ is a proper  $\cC_{\pih}$-algebra. Its restriction to $ {Q\cap F_0}$ is indeed a  proper $G_0$ algebra and  its restriction to ${Z_1\cap F_0}$ is a proper $G_1\times (0,1)$-algebra. 

It may be interesting to look for a way to say that the $C^*(\DNC_{[0,1)}(M_0,(M_1)_{q'}^{q'}))$ is somewhat a universal proper algebra.
\end{remark}

\subsection{Baum-Connes map for the telescopic algebra}\label{subsec:BCtelescope}

Since a mapping telescope is a mapping cylinder which, in turn, is a mapping cone (\cf remark \ref{rem:telescope=cone}) we can just proceed and construct the left hand side for the telescopic algebra - and therefore for the foliation one.

We are given a nicely decomposable foliation $(M,\cF)$, a decomposition given by an increasing sequence $\Omega_k$ of saturated sets - we put $Y_k=\Omega _k\setminus \Omega_0$, a sequence of Lie groupoids $\cG_k\gpd W_k\subset \Omega _k$ such that $Y_k\subset W_k$ and $W_k\cap \Omega_{k-1} \subset W_{k-1}$; we put $\cG'_k=(\cG_k)_{|W_k\cap \Omega_{k-1}}$ and assume that we have a groupoid morphism which is a submersion $\pi_k:\cG'_k\to \cG_{k-1}$.

We further assume that we have submersions of manifolds $p_k:M_k\to W_k$ which are classifying spaces for proper actions of $\cG_k$. For $k\ge 1$, the restriction $p_k^{-1}(\Omega_{k-1})$ of $M_k$ is a classifying space for $\cG'_k$ but we may need to modify it: we choose a classifying space given by a submersion $p'_k:M'_k\to \Omega_{k-1} \subset W_{k-1}$ in such a way that the maps $q_k:M_k'\to M_{k-1}$ and $\hat q_k:M_k'\to M_{k}$ are submersions.

We then construct the classifying groupoids \begin{itemize}
\item $\bH_k$ associated to the morphism $\pi_k:\cG'_k\to \cG_{k-1}$ and the submersion $q_k:M_k'\to M_{k-1}$ of classifying spaces;
\item  $\widehat \bH_k$ associated to the morphism $j_k:\cG'_k\to \cG_{k}$ and the submersion $\hat q_k:M_k'\to M_{k}$ of classifying spaces
\end{itemize}
We then glue the groupoids $\bH_k$ and $\widehat \bH_k$ in their common part $(\bH_k)_{\{1/2\}\times [0,1/2]}=(\widehat \bH_k)_{\{1/2\}\times [0,1/2]}$ (\cf remark \ref{remark:gluing}) and obtain a groupoid $\widetilde \bH_k$.

For a locally closed part $Y$ of $Z_0=[0,1)]\times [0,1/2]$ we $(\widetilde \bH_k)_Y=(\bH_k)_Y\cup (\widehat \bH_k)_Y $.

Recall that $Q=\{(u,v);\ 0\le u\le v\le 1/2\}$ and $Z_1=Z_0\setminus Q$.

We define diffeomorphisms $\vartheta_k :Z_1\to [0,1]\times (k-1,k)$ by setting $\vartheta (u,v)=(2v,k-\frac{u-v}{1-v})$ and $\hat \vartheta_k :Z_1\to [0,1]\times (k,k+1)$ by setting $\vartheta (u,v)=(2v,k+\frac{u-v}{1-v})$.

Thanks to this diffeomorphism, we obtain identifications of \begin{itemize}
\item $\Theta_k:(\bH_k)_{Z_1}\overset{\sim}\longrightarrow \Big(\DNC((\cG_{k-1})_{p_{k-1}}^{p_{k-1}},M_{k-1}\rtimes\cG_{k-1})_{[0,1]} \Big)_{q_k'}^{q_k'}\times (k-1,k)$ where $q'_k:M_{k-1}\coprod M'_k\to M_{k-1}$ is the identity on $M_{k-1}$ and $q_k$ on $M'_k$;
\item $\widehat \Theta_k:(\widehat \bH_k)_{Z_1}\overset{\sim}\longrightarrow \Big(\DNC((\cG_{k})_{p_{k}}^{p_{k}},M_{k}\rtimes\cG_{k})_{[0,1]} \Big)_{\hat q_k'}^{\hat q_k'}\times (k,k+1)$  where $\hat q'_k:M_{k}\coprod M'_k\to M_{k}$ is the identity on $M_k$ and $\hat q_k$ on $M'_k$.
\end{itemize}

Define $q''_k:M_k\coprod M'_k\coprod M'_{k+1}\to M_k$ to be the map coinciding with the identity over $M_k$, $\hat q_k$ on $M'_k$ and $q_{k+1}$ on $M'_{k+1}$ - with the convention $M'_0=\emptyset$ and, if $n\ne +\infty$, $M'_{n+1}=\emptyset$.

\begin{definition} We define the \emph{adiabatic telescopic groupoid} $\bG$ to be the union  $\bigcup _{k=1}^n(\widetilde \bH_k)_Q\times \{k\}$ with $\bigcup_{k=0}^n\Big(\DNC((\cG_{k})_{p_{k}}^{p_{k}},M_{k}\rtimes\cG_{k})_{[0,1]} \Big)_{q_k''}^{q_k''}\times (k,k+1)$. The gluing is obtained by mapping $\bH_k\to \bG$:  
\begin{itemize}
\item we map  $(\widetilde \bH_k)_Q$ to $\bG$ by the map  $\gamma \mapsto (\gamma,k)\in \bG$
\item Using $\Theta_{k+1}$ we map $(\bH_{k+1})_{Z_1}$ to $\bigcup_{k=0}^n\Big(\DNC((\cG_{k})_{p_{k}}^{p_{k}},M_{k}\rtimes\cG_{k})_{[0,1]} \Big)_{q_k'}^{q_k'}\times (k,k+1)$ which  is a subset of $\bigcup_{k=0}^n\Big(\DNC((\cG_{k})_{p_{k}}^{p_{k}},M_{k}\rtimes\cG_{k})_{[0,1]} \Big)_{q_k''}^{q_k''}\times (k,k+1) \subset \bG$;
\item Using $\widehat\Theta_k$ we map $(\widehat\bH_k)_{Z_1}$ to $((M_{k}\times _{p_{k}}M_{k})_{\overline{ad}}\rtimes \cG_{k})_{\hat q_k'}^{\hat q_k'}\times (k,k+1)$ which  is a subset of\\ $((M_{k}\times _{p_{k}}M_{k})_{\overline{ad}}\rtimes \cG_{k})_{q''_{k}}^{q''_{k}}\times (k,k+1)\subset \bG$.
\end{itemize}
We define the obvious map $\chi:\bG^{(0)}\to (0,n+1)$ (using the convention $+\infty+1=+\infty$ of course). Thanks to $\chi$, the (full) $C^*$-algebra $C^*(\bG)$ is a $C_0(0,n+1)$ algebra (with the convention $+\infty +1=+\infty$).
\end{definition}

Define a map $\xi :Z\to [0,1]$ by setting $\xi(u,v)=2\min(u,v)$ and let $\hat \xi :\bG^{(0)}\to [0,1]$ be defined as the composition $\bG^{(0)}\to Q\overset{\xi}\rightarrow [0,1]$ on $(\widetilde \bH_k)_Q\times \{k\}$ and letting $\hat \xi$ to be the parameter in the adiabatic deformation $(M_k\times _{p_k}M_k)_{\overline{ad}}$ on $\bigcup_{k=0}^n((M_k\times _{p_k}M_k)_{\overline{ad}}\rtimes \cG_k)_{q_k''}^{q_k''}\times (k,k+1)$.

We define then the subgroupoids $\bG_0$ and $\bG_1$ of $\bG$, restrictions of $\bG$ to the closed saturated set $\hat \xi^{-1}(\{i\})$.

We then   have:

\begin{prop}
\begin{enumerate}
\item The algebra $C^*(\bG_0)$ is nuclear.
\item The kernel of the evaluation $\rho_0:C^*(\bG)\to C^*(\bG_0)$ is $K$-contractible.
\item The algebra $C^*(\bG_1)$ is Morita equivalent to the telescopic algebra
\end{enumerate}
\begin{proof}
 \begin{enumerate}
\item In fact $C^*(\bG_0)$ sits in an exact sequence $$0\to \bigoplus_{k=0}^nC^*(((\ker p_k)^*\rtimes \cG_k)_{q_k'}^{q_k'}\times (0,1))\to C^*(\bG_0)\to \bigoplus_{k=1}^nC^*((\ker p'_k)^*\rtimes \cG'_k\times [0,1])\to 0$$ and the Lie groupoids $(\ker p_k)^*\rtimes \cG_k)_{q_k'}^{q_k'}$ and $(\ker p'_k)^*\rtimes \cG'_k$ are proper. It follows that $C^*(\bG_0)$ is in fact a type $I$ algebra.
\item We have a semi-split exact sequence $$ 0\to C_0((0,1]\times (0,1))\otimes B  \to \ker \rho_0\to C_0(Q\setminus F_0)\otimes B'\to 0 \to 0$$
where $B=\bigoplus_{k=0}^n C^*(((M_{k}\times _{p_{k}}M_{k})\rtimes \cG_{k})_{q''_{k}}^{q''_{k}}) $ and $B'=\bigoplus_{k=1}^nC^*((M'_{k}\times _{p'_{k}}M'_{k})\rtimes \cG'_{k})$. The algebras $C_0((0,1]\times (0,1))$ and $C_0(Q\setminus F_0)$ are contractible.
\item Actually the groupoid $\bG_1$ is Morita equivalent to the telescopic groupoid.
\qedhere
\end{enumerate}
\end{proof}
\end{prop}

\begin{definition}\label{dfn:BCmaptelescMF}
Let  $(M,\cF)$ be a nicely decomposable foliation. Assume that the classifying spaces of all the groupoids $\cG_k\gpd W_k$ involved in this decomposition are manifolds. With the above construction:\begin{itemize} 
\item we define the left hand side, \ie the topological $K$-theory (of this decomposition) to be the $K$-theory of $C^*(\bG_0)$;
\item we define the Baum-Connes map for the telescope to be the composition $[\rho_0]^{-1}\otimes [\rho_1]\otimes [\bE]$;
\item we define the Baum-Connes map for $(M,\cF)$ to be the Baum-Connes map for the telescope composed with the isomorphism $K_*(\cT)\to K_{*+1}(C^*(M,\cF))$.
\end{itemize}
\end{definition}

Let $\rho_1:C^*(\bG)\to C^*(\bG_1)$ be evaluation. The kernel of $\rho_1$ is a $C_0(0,n+1)$ algebra. It follows from the inductive limit construction that if $(\ker \rho_1)_{(k,k+1)}$ and $(\ker \rho_1)_k$ are $E$-contractible for all $k$, then so is $\ker \rho_1$. We thus obtain:

\begin{thm}\label{mainthm}
Let  $(M,\cF)$ be a nicely decomposable foliation such that the classifying spaces of all the groupoids $\cG_k\gpd W_k$ involved in this decomposition are manifolds. if the \emph{full} Baum-Connes conjecture holds for all of them, then the \emph{full} Baum-Connes map of definition \ref{dfn:BCmaptelescMF} is an isomorphism.
\end{thm}

\begin{cor}
Let  $(M,\cF)$ be a nicely decomposable foliation. If all the groupoids $\cG_k\gpd W_k$ involved in this decomposition are amenable and their classifying spaces are manifolds, then the Baum-Connes map is an isomorphism.
\end{cor}

\section{Two examples of foliations of singularity height one given by linear actions}\label{sec:computelinearactions}

In this section we compute the $K$-theory for two simple examples of  foliations of singularity height one. 

\subsection{\texorpdfstring{The $SO(3)$-action}{The SO(3)-action}}\label{sec:SO3}

In this section we consider the foliation $(\R^{3},\cF)$ defined by the action of $SO(3)$ on $\R^3$ (\cf example \ref{exs:basic}.\ref{examp:b}).

\subsubsection{Holonomy groupoid and exact sequences}

As discussed  in \ref{exs:basic}, $H(\cF)=SO(3)\times \{0\}\coprod \R_+^*\times S^2\times S^2$ and $\cF$ is nicely decomposable, in the sense of definition \ref{dfn:goodgpd} with $\Omega_0=\R^3\setminus \{0\}$ and $\cG_1=\R^3\rtimes SO(3)$.

Note that $SO(3)$ is compact and therefore amenable, so the reduced and full crossed product $C^{\ast}$-algebras coincide. The (full and reduced) $C^{\ast}$-algebra of $\cG_0$ is the crossed product $C_0(\R^3) \rtimes SO(3)$. 

Writing $\R^3\setminus \{0\}=\R_+^*\times S^2$, we find that $C^*(\cG|_{\Omega_0})=C_0(\R_+^*)\otimes (C(S^2)\rtimes SO(3))$ and
$C^*(M,\cF)|_{\Omega_0}=C_0(\R_+^*)\otimes \cK(L^2(S^2))$, 

Now figure \ref{fig:general} reads:
\begin{figure}[H]
\[
\xymatrix{
0 \ar[r] & C_0(\R_+^*)\otimes (C(S^2)\rtimes SO(3)) \ar[d]^{\widehat q} \ar[r]^i &C_0(\R^3) \rtimes SO(3) \ar[r] \ar[d]^{\pi} & C^{\ast}(SO(3)) \ar[r] \ar@{=}[d] & 0 &\ES\\
0 \ar[r] &         C_0(\R_{\ast}^{+}) \otimes \cK(L^2(S^2)) \ar[r]        & C^{\ast}(\R^3,\cF) \ar[r]               &  C^{\ast}(SO(3)) \ar[r]                   & 0 &\ES }
\]
\caption{Exact sequences for the $SO(3)$ action.\label{fig:SO3}}
\end{figure}
Here $\widehat q=\id_{C_0(\R_+^*)}\otimes q$ where $q:C(S^2) \rtimes SO(3)\to \cK(L^2(S^2))$ is obtained by integration along the fibers of the groupoid morphism $(t,s):S^2\rtimes SO(3)\to S^2\times S^2$.

\subsubsection{\texorpdfstring{Calculation of  $K$-theory with mapping cones}{Calculation of K-theory with mapping cones}}\label{sec:mapconSO3}

To describe the foliation $C^{\ast}$-algebra we give an interpretation of diagram \ref{fig:SO3} using mapping cones.

Denote by $\rho  : C^*(SO(3)) \to \cK(L^2(S^2))$ the natural representation of $SO(3)$ on $L^2(S^2)$. We thus have a diagram
\begin{figure}[H]
\[
\xymatrix{
C^{\ast}(SO(3)) \ar[rd]_{\rho} \ar[r]^j & C(S^2) \rtimes SO(3) \ar[d]^{q} \\
  & \cK(L^2(S^2)) 
}
\]
\caption{Mapping cones for the $SO(3)$ action.\label{diag:conesSO3}}
\end{figure}
where $j:C^{\ast}(SO(3))\to C(S^2) \rtimes SO(3)$ is the morphism induced by the unital inclusion $\C\to C(S^2)$.

Identify  $C_0(\R^3)$ with the mapping cone of $\C\to C(S^2)$. Taking crossed products by the action of $SO(3)$ and using the diagram in figure \ref{fig:SO3}, we find:
\begin{itemize}
\item The crossed product $C^{\ast}$-algebra $C_0(\R^3)\rtimes SO(3)$ in extension $(EC5)$ is the mapping cone $\cC_p$, where $p$ is the map $j : C^{\ast}(SO(3)) \to C(S^2)\rtimes SO(3)$

\item The foliation $C^{\ast}$-algebra $C^{\ast}(\cF)$ in extension $(EC6)$ is the mapping cone $\cC_{\rho}$.
\end{itemize}

To describe $C^{\ast}(\cF)$ it suffices to describe the representation $\rho : C^{\ast}(SO(3)) \to \cK(L^2(S^2))$. 

It follows from the Peter-Weyl theorem $C^{\ast}(SO(3)) = \oplus_{m\in \N} M_{2m+1}(\C)$ and $K_0(C^{\ast}(SO(3)))=\Z^{(\N)}$ (and $K_1(C^{\ast}(SO(3))) =\{0\}$).

In order to compute the map $\rho_*:K_0(C^{\ast}(SO(3)))\to \Z$, we have to understand how many times the representation $\sigma_{m}$ (of dimension $2m+1$) appears in $\rho $, \ie count the dimension of $\mathrm{Hom}_{SO(3)}(\sigma_m,\rho)$.

Since $S^2=SO(3)/S^1$, the representation $\rho$ is the representation $\mathrm{Ind}_{S^1}^{SO(3)}(\varepsilon)$ induced by the trivial representation $\varepsilon$ of $S^1$. Using the Frobenius reciprocity theorem we know $\dim(\mathrm{Hom}_{SO(3)}(\sigma_m,\rho) = \dim(\mathrm{Hom}_{S^1}(\sigma_m,\varepsilon)=1$.

It follows that the map $\rho_*:K_0(C^{\ast}(SO(3)))\to \Z$ is the map which maps each generator $[\sigma_m]$ of $K_0(C^{\ast}(SO(3)))$ to $1$.

We immediately deduce:

\begin{prop}\label{prop:KthSO3}
We have $K_0(C^{\ast}(\cF))=\ker \rho_*\simeq \Z^{(\N)}$ and $K_1(C^{\ast}(\cF))=0$.\hfill$\square$
\end{prop}

\begin{remark}
In the same way, one may easily compute $j_*:K_0(C^*(SO(3))\to K_0(C(S^2)\rtimes C^*(SO(3))$ and $K_*(C_0(\R^3) \rtimes SO(3))$. In fact, this is a really classical result that we recall briefly here: The algebra $C(S^2)\rtimes C^*(SO(3))$ is Morita equiavalent to $C^*(S^1)$ and the morphism $j_*:K_0(C^*(SO(3)))\to K_0(C^*(S^1))$ is the restriction morphism $R(SO(3))\to R(S^1)$ where $R(G)=K_0(C^*(G))$ is the representation ring of a compact group $G$ (see \cite{Rieffel, Julg}).

It follows that $j_*([\sigma_m])=\sum_{k=-m}^m[\chi_k]$ where the $(\chi_k)_{k\in \Z}$ are the characters of $S^1$. The morphism $j_*$ is therefore (split) injective, and we find $K_0(C_0(\R^3) \rtimes SO(3))=0$ and $K_1(C_0(\R^3) \rtimes SO(3))\simeq \Z^{(\N)}$.
\end{remark}

%
%
%

\subsection{\texorpdfstring{The $SL(2,\R)$-action}{The SL(2,R)-action}}\label{sec:SL2}

We consider the foliation on $\R^2$ induced by the action of $SL(2,\R)$. Recall the following:
\begin{enumerate}
\item $SL(2,\R)$ is not compact and not amenable, but it was shown by Kasparov in \cite{Kasparov1984} to be $KK$-amenable. 
\item Its maximal compact is $S^1$.
\item The action of $SL(2,\R)$ on $\R^2\setminus\{0\}$ is transitive and the stabilizer of the point $(1,0)$ is the set of matrices of the form 
$\begin{pmatrix}
 1&t\\0&1
\end{pmatrix}$.
 Whence the action groupoid $(\R^2 \setminus\{0\})\rtimes SL(2,\R)$ is Morita equivalent to the group $\R$. So the crossed product $C_0(\R^2\setminus\{0\})\rtimes SL(2,\R)$ is Morita equivalent to the group $C^*$-algebra $C^*(\R)$.
\item It follows as above from Lemma \ref{lem:holgpd} (see also \cite[Ex. 3.7]{AS1}) that the associated holonomy groupoid is $$H(\cF) = (\R^2\setminus\{0\} \times \R^2\setminus\{0\}) \bigcup \{0\} \times SL(2,\R)$$ 
\end{enumerate}
It follows that this foliation is nicely decomposable of singularity height $1$ with $\cG=\cG_0=\R^2 \rtimes SL(2,\R)$ (see remark \ref{rem:length1}). Here the diagram of figure \ref{fig:general} reads:

\begin{figure}[H]
\[
\xymatrix{
0 \ar[r] & (C_0(\R^2\setminus\{0\}))\rtimes SL(2,\R) \ar[r]^\iota \ar[d]^{\pi_1} & C_0(\R^2) \rtimes SL(2,\R) \ar[r] \ar[d]^{\pi}    & C^{\ast}(SL(2,\R)) \ar[r] \ar@{=}[d]  & 0\\
0 \ar[r] &    \cK(L^2(\R^2\setminus\{0\})) \ar[r]    & C^{\ast}(\cF) \ar[r]                              &  C^{\ast}(SL(2,\R)) \ar[r]        & 0   
}
\]
\caption{Exact sequences for the $SL(2,\R)$ action.\label{fig:SL2}}
\end{figure}
Recall that the $C^{\ast}$-algebras involved are full $C^{\ast}$-algebras.

\subsubsection{\texorpdfstring{Direct Calculation of  $K$-theory}{Direct Calculation of K-theory}}\label{sec:SL2direct}

The short exact sequence $$0 \to \cK(L^{2}(\R^2\setminus\{0\})) \to C^{\ast}(\cF) \to C^{\ast}(SL(2,\R)) \to 0$$ gives the 6-terms exact sequence:
\[
\xymatrix{
K_0(\cK(L^{2}(\R^2\setminus\{0\}))) \ar[r]   & K_0(C^{\ast}(\cF)) \ar[r] & K_0(C^{\ast}(SL(2,\R))) \ar[d] \\ 
K_1(C^{\ast}(SL(2,\R))) \ar[u]  & K_1(C^{\ast}(\cF)) \ar[l] & K_1(\cK(L^{2}(\R^2\setminus\{0\}))) \ar[l]
}
\]
We have $K_0(\cK(L^{2}(\R^2\setminus\{0\}))) = \Z$ and $K_1(\cK(L^{2}(\R^2\setminus\{0\}))) = 0$. On the other hand, using the Connes-Kasparov conjecture proved in \cite{Kasparov1984, Wassermann}, we have $K_1(C^{\ast}(SL(2,\R)))=0$. We conclude that $$K_{1}(C^{\ast}(\R^2,\cF))=0 \quad \text{and} \quad K_0(C^{\ast}(\R^2,\cF))=\Z \oplus K_0(C^{\ast}(SL(2,\R))) = \Z \oplus \Z^{(\Z)}$$

\subsubsection{\texorpdfstring{Calculation of  $K$-theory with mapping cones}{Calculation of  K-theory with mapping cones}}\label{sec:SL2mapcone}

Although the above construction is quite direct, it may be worth examining a construction following the general procedure of section \ref{sec:MappingConeGeneral} (Prop.  \ref{prop:Eequiv}):

To apply the mapping cones approach we gave in section \ref{sec:MappingConeGeneral}, we need the following result which follows from Kasparov \cite{Kasparov1984, Kasparov1988}.
\begin{prop}\label{prop:Kasparov}
Let $SL(2,\R)$ act on a $C^{\ast}$-algebra $A$ by automorphisms. The algebras $A\rtimes SL(2,\R)$ and $A \rtimes S^1$ are $KK$-equivalent.
\end{prop}
\begin{proof}
The Lie group $S^1$ is a maximal compact subgroup of $SL(2,\R)$. Note also that $SL(2,\R)/S^1$ is the Poincaré half plane and therefore admits a complex structure, hence an $SL(2,\R)$-invariant spin$^c$ structure. The result follows from \cite{Kasparov1984}.
\end{proof}

It follows in fact from \cite{Kasparov1984} that the exact sequences $$0\to C_0(\R^2\setminus\{0\}) \rtimes SL(2,\R)\to C_0(\R^2) \rtimes SL(2,\R)\to C^*(SL(2,\R))\to 0$$ and $$0\to C_0(\R^2\setminus\{0\}) \rtimes S^1\to C_0(\R^2) \rtimes S^1\to C^*(S^1)\to 0$$ are $KK$-equivalent. Note that\begin{itemize}
\item  $K_0(C^*(SL(2,\R)))=K_0(C^*(S^1))=\Z^{(\Z)}$, and $K_1=0$. 
\item since $S^1$ acts freely on $\R^2\setminus\{0\}$, with quotient $\R_+^*$, it follows that  $C_0(\R^2\setminus\{0\}) \rtimes S^1$ is Morita equivalent to $C_0(\R_+^*)$; also since $SL(2,\R)$ acts transitively on $\R^2\setminus\{0\})$ with stabilizers isomorphic to $\R$,  it follows that  $C_0(\R^2\setminus\{0\}) \rtimes SL(2,\R)$ is Morita equivalent to $C^*(\R)$.  It follows that $K_1(C(\R^2\setminus\{0\})\rtimes SL(2,\R))=K_1(C(\R^2\setminus\{0\})\rtimes S^1)=\Z$ and $K_0=0$.
\item Using the complex structure of $\R^2$, we have a Bott isomorphism of $K_*(C_0(\R^2) \rtimes S^1)$ with $K_*(C^*(S^1))$. It follows that $K_0(C(\R^2)\rtimes SL(2,\R))=K_0(C(\R^2)\rtimes S^1)=\Z^{(\Z)}$, and $K_1=0$. 
\end{itemize}

From this discussion, it follows that the morphism $\iota:C_0(\R^2\setminus\{0\}) \rtimes SL(2,\R)\to C_0(\R^2) \rtimes SL(2,\R)$ induces the $0$ map in $K$-theory, and so does the map $\pi:C_0(\R^2\setminus\{0\}) \rtimes SL(2,\R)\to \cK$.

\begin{remark}
Denoting by $(\chi_n)_{n\in \Z}$ the characters of $S^1$, the image of $[\chi_n]\in K_*(C_0(\R^2) \rtimes S^1)$  by  $C_0(\R^2) \rtimes S^1\to C^*(S^1)$ (evaluation at $0$) is $[\chi_n]-[\chi_{n+1}]$. This morphism is one to one and its image is the set of elements in $R(S^1)$ of dimension $0$.
\end{remark}


As the maps $\iota$ and $\pi $ induce the $0$ map in $K$-theory, we find as above from Prop. \ref{prop:Eequiv}.

\begin{prop}\label{prop:conesSL2}
Let $\cF$ be the foliation defined by the action of $SL(2,\R)$ on $\R^2$. We have $$K_0(\cF)\simeq K_0(C_0(\R^2)\rtimes SL^2(\R))\oplus K_0(\cK)\oplus K_1(C_0(\R^2\setminus\{0\})\rtimes SL^2(\R))\simeq \Z^{(\N)}\oplus \Z\oplus \Z$$ and $K_1(\cF)=0$.\hfill$\square$
\end{prop}

Note that we have a split short exact sequence $0\to K_0(C_0(\R^2)\rtimes SL^2(\R))\to K_0(C^*(SL^2(\R)))\to K_1(C_0(\R^2\setminus\{0\})\rtimes SL^2(\R))\to 0$, and thus the results of prop. \ref{prop:conesSL2} and the direct calculation (section \ref{sec:SL2direct}) are coherent.

\subsection{Generalizations}

The examples introduced above can be extended to the action of $SO(n)$ or $SL(n,\R)$ on $\R^n$. Let us discuss here a slightly more general situation which still gives singularity height one foliations.

\subsubsection{Subgroups of $SO(n)$}

Let $G$ be a connected closed subgroup of $SO(n)$. Assume that its action on $S^{n-1}$ is transitive, and let $H\subset G$ be the stabilizer of a point in $S^{n-1}$. Denote by $\cF$ the foliation of $\R^n$ associated with the action of $G$. Exactly as in the case of the action of $SO(3)\in \R^3$, we find that\begin{itemize}
\item $H(\cF)=G\times \{0\}\coprod \R_+^*\times S^{n-1}\times S^{n-1}$;
\item  $C^*(\R^n,\cF)$ is the mapping cone of the morphism $C^*(G)\to \cK(S^{n-1})$. 
\item The map $R(G)\to \Z$ corresponding to this morphism associates to a (virtual) representation $\sigma$ the (virtual) dimension of its $H$ fixed points. It is onto, and therefore $K_0(C^*(\R^n,\cF))=\Z^\N$ and $K_1(C^*(\R^n,\cF))=0$.
\end{itemize}

\subsubsection{\texorpdfstring{Subgroups of $GL_n$}{Subgroups of  GL(n)}}\label{sec:SLn}

Let now $G$ be a closed connected subgroup of $GL(n,\R)$. Assume that its action on $\R^n\setminus \{0\}$ is transitive, and let $H\subset G$ be the stabilizer of a nonzero point in $\R^n$. 

As for the case of $SL(2,\R)$ acting on $\R^2$, we have:

\begin{itemize}
\item The holonomy groupoid is $H(\cF) = (\R^n \setminus\{0\}) \times (\R^n \setminus\{0\}) \coprod G \times \{0\}$. 
\item We have an exact sequence of full $C^*$-algebras $0\to \cK(L^2(\R^n \setminus\{0\}))\to C^*(\R^n,\cF)\to C^*(G)\to 0$ and therefore an exact sequence:
$$0\to K_1(C^*(\R^n,\cF))\to K_1(C^*(G))\stackrel{\partial}{\longrightarrow} \Z\to K_0(C^*(\R^n,\cF))\to K_0(C^*(G))\to 0.$$
In order to try and compute the connecting map $\partial$, we may use the diagram of figure \ref{fig:general}. Note that the groupoid $(\R^n \setminus\{0\})\rtimes G$ is Morita equivalent to the group $H$. Following this diagram, the connecting map $\partial$ is the composition of the trivial representation of $H$ of with the connecting map $\partial ':K_1(C^*(G))\to K_0(C_0(\R^n \setminus\{0\})\rtimes G)\simeq K_0(C^*(H))$.
\end{itemize}

An example of this kind is of course $SL(n,\C)\subset SL(2n,\R)$.  The stabilizer group of a $z \in \C^n\setminus\{0\}$, say $z=(1,0,\ldots,0)$ is the group of matrices in $SL(n,\C)$ whose first row is $(1,0,\ldots,0)$. That is $\C^{n-1}\rtimes SL(n-1,\C)$.

Another example is given by $G=G_1\times \R_+^*$ where $G_1$ is a connected closed subgroup of $SO(n)$ whose action on $S^{n-1}$ is transitive, and $\R_+^*$ acts by similarities. Note that if $\cF_1$ is the foliation defined by the action of $G_1$, there is a natural action of $\R_+^*$ on $H(\cF_1)$ and $H(\cF)$ is a semi-direct product $H(\cF_1)\rtimes \R_+^*$; we find $C^*(\R^n,\cF)=C^*(\R^n,\cF_1)\rtimes \R_+^*$. Thanks to the Connes-Thom isomorphism, the algebras $C^*(\R^n,\cF)$ and $C^*(\R^n,\cF_1)$ have the same $K$-theory up to a shift of dimension.

\section{\texorpdfstring{Actions of $\R$ on manifolds}{actions of R on manifolds}}\label{sec:X}

We come to example \ref{exs:basic}.\ref{example:actionofR}). Let $M$ be a manifold endowed with a smooth action $\alpha$ of $\R$. Let $\cF$ be foliation associated with this action - \ie with the the groupoid $M\rtimes_\alpha \R$. We keep the notation of  example \ref{exs:basic}.\ref{example:actionofR}).

We showed that $\cF$ is nicely decomposable in the sense of definition \ref{dfn:goodgpd}. Here we compute the $K$-theory using an exact sequence. Note that in this example, in presence of periodic points, the groupoid $\cG_0$ is not always Hausdorff and its classifying space for proper actions is not a manifold, therefore theorem \ref{mainthm} does not apply directly.

From proposition \ref{prop:nicedecompoactionofR} we deduce that the groupoid $\cG_1'\gpd M$ which coincides with $H(\cF)$ on the complement of $W$ and with $W\times \R$ on $W$ is a (non necessarily Hausdorff) Lie groupoid and gives rise to the nice decomposition $(W\gpd W, \cG_1'\gpd M)$ of $\cF$. We will rather exploit this one in the computations below.

Put also $Y=M\setminus W$.

\subsection{Exact sequence of fixed points}

\begin{prop}
The $KK^1$-element associated with the exact sequence
\[\xymatrix{
0 \ar[r] & C_0(W)  \ar[r] \ar[r] & C^{\ast}(M,\cF)  \ar[r] & C^*(M,\cF)|_{Y}\ar[r]   & 0 &(EC1)}
\]
 is $0$.
\end{prop}
\begin{proof}
The corresponding exact sequence for the groupoid $\cG'_1$ givers rise to the following diagram:
\[
\xymatrix{
0 \ar[r] & C_0(W)  \ar[r] \ar[r] & C^{\ast}(M,\cF)  \ar[r] & C^*(M,\cF)|_{Y}\ar[r]  \ar@{=}[d] & 0 &(EC1)\\
0 \ar[r] & C_0(W\times\R) \ar[r] \ar[u]^{ev_0} & C^*(\cG'_1)\ar[r] \ar[u] & C^*(\cG_1')|_{Y}  \ar[r] & 0 &(EC2)}
\]
Denote by $z_1,z_2$ the $KK^1$ elements associated with exact sequences $(EC1)$ and $(EC2)$. We have $z_1=(ev_0)_*(z_2)$. But $(ev_0)$, which is the map induced by the inclusion $x\mapsto (x,0)$ from $W$ to $W\times \R$ is the $0$ element in $KK$, whence $z_1=0$ as claimed.
\end{proof}

We immediately deduce:

\begin{cor}
We have $K_*(C^*(M,\cF))=K_*(C_0(W))\oplus K_*(C^*(M,\cF)|_{Y})$.\hfill$\square$

\end{cor}

If all periodic points are in fact fixed, \ie if ${Per}(\alpha)=W$, then, using Connes' Thom isomorphism \cite{ConnesThom} this computation yields:

\begin{cor}\label{prop:compKX}
Assume that all the periodic points are in fact fixed. The $K$-theory group of $C^*(M,\cF)$ is \\ \indent $\hfill K_*(C^*(M,\cF)) = K_*(C_0(W)) \oplus K_*(C_{0}(Y)\rtimes\R) = K_*(C_0(W)) \oplus K_{1-*}(C_{0}(Y)). \hfill \square
$
\end{cor}

\begin{remark}
Corollary \ref{prop:compKX} can be interpreted by saying that, when there are no non trivial stably periodic points, the classifying space of proper actions of the holonomy groupoid is $W \coprod Y \times \R$. The associated assembly map is an isomorphism. 
\end{remark}


\subsection{The stably periodic points}

In presence of nontrivial stable periodic points, the complete computation of the $K$-theory is not so simple... Even in the regular case, this computation is quite hard. See \eg \cite{Torpe}.

As a consequence of proposition \ref{prop:nicedecompoactionofR} we find:

\begin{prop}
The set  $\widehat{Per}(\alpha) = Per(\alpha) \setminus W$ of \emph{nontrivial stably periodic points} is open.
\end{prop}
\begin{proof}
By prop. \ref{prop:nicedecompoactionofR}, the set $W$ is closed in $Per(\alpha)$, whence its complement is open in $Per(\alpha)$ - and therefore in $M$ since $Per(\alpha)$ is open.
\end{proof}

For $x\in \widehat{Per}(\alpha)$, let $p(x)\in \R_+$ be the infimum of the set of $t>0$ such that $(x,t)\in M\times \R$ is the trivial element in $H(\cF)$. By \cite{Debord2013} it follows that $p(x)>0$ and $(x,p(x))$ is  the trivial element in $H(\cF)$.

\begin{prop}
The map $p:\widehat{Per}(\alpha)\to \R_+$ is smooth.
\end{prop}
\begin{proof}
Since $(x,p(x))$ is the trivial element in $H(\cF)$ there exists an open neighborhood $U\subset \widehat{Per}(\alpha)$ of $x$ and a bounded (below and above) smooth function $f:U\to \R_+^*$ such that $(y,f(y))\in P$ for all $y\in U$ and $p(x)=f(x)$. For $y\in U$, since $U$ is a neighborhood of $y$, it follows that $f(y)$ is a multiple of $p(y)$.

\begin{itemize}
\item Assume $X(x)=0$. Let $m\in \R_+$ such that $f(y)\le m$ for all $y\in U$. Put  $V=\{x\in U;\ \forall t\in [0,m],\  \alpha_t(x)\in U\}$; by compactness of $[0,m]$ it is an open subset of $U$. Then by periodicity, $V$ is invariant by $\alpha_t,\ t\in \R$. For $y\in V$ and $t\in [0,T]$, as $f(\alpha_t(y))$ is a multiple of $p(\alpha_t(y))=p(y)$, itt follows by continuity of $f$  that $f(\alpha_t(y))=f(y)$. Replacing $X$ by $1/fX$, we get an action of $S^1$ on $V$.

Since $S^1$ is compact, Bochner's linearization theorem \cite{Bochner} says that in an open and $S^1$-equivariant neighborhood  $U'$ of $x$ the $S^1$-action is actually a linear representation of $S^1$, which is faithful since $f(x)=p(x)$. It follows that $p(y)=f(y)$ for all $y\in U'$.
\item Assume that $X(x)\ne 0$. Then $x$ is periodic of period $p(x)/k$ with $k\in \N^*$. Now choose a transversal $T$ at $x$; we get an action of $\Z/k\Z$ and applying Bochner's linearization theorem again, we conclude that $f(y)=p(y)$ in a neighborhood of $x$.\qedhere
\end{itemize}
\end{proof}

When restricting to $\widehat{Per}(\alpha)$, we may therefore replace $X$ by $\frac{1}{p}X$ and obtain an action of $S^1$. The foliation groupoid is then $W\coprod \widehat{Per}(\alpha)\rtimes S^1\coprod  (M\setminus Per(\alpha))\rtimes \R$.

\begin{remarks}
\begin{enumerate}
\item The building blocks of $C^*(M,\cF)$ are the algebras $C_0(W)$, $C_0(\widehat{Per}(\alpha))\rtimes S^1$ and $C_0(M\setminus Per(\alpha))\rtimes \R$. For each of them there is of course a left-hand side and a Baum-Connes map. Actually, since the first two are given by compact group actions, they are their own left hand side! The left hand side for  $C_0(M\setminus Per(\alpha))\rtimes \R$, given by Connes' Thom isomorphism is $(M\setminus Per(\alpha))\times \R$.
\item We already noticed that $K_*(C^*(M,\cF))=K_*(W)\oplus K_*( C^*(M,\cF)|_{Y})$.

To compute $K_*( C^*(M,\cF)|_{Y})$ we may use the exact sequence $$0\to C_0(\widehat{Per}(\alpha))\rtimes S^1\longrightarrow C^*(M,\cF)|_{Y}\longrightarrow C_0(M\setminus Per(\alpha))\rtimes \R\to 0\eqno(EC3)$$

In order to compute the connecting map of this sequence we note that we have a diagram:
\[
\xymatrix{
0 \ar[r] & C_0(\widehat{Per}(\alpha))\rtimes S^1  \ar[r] \ar[r] & C^{\ast}(\cF)|_{Y}  \ar[r] & C^*(M,\cF)|_{M\setminus Per(\alpha)}\ar[r]  \ar@{=}[d] & 0 &(EC3)\\
0 \ar[r] & C_0(\widehat{Per}(\alpha))\rtimes\R \ar[r] \ar[u]^{q} & C_0(Y)\rtimes\R\ar[r] \ar[u] & C_0(M\setminus Per(\alpha))\rtimes \R  \ar[r] & 0 &(EC4)}
\]
\end{enumerate}
Denote by $z_3,z_4$ the $KK^1$ elements associated with exact sequences $(EC3)$ and $(EC4)$. We have $z_3=q_*(z_4)=z_4\otimes [q]$. To compute $z_3$ we then remark:

\begin{itemize}
\item Through Connes' Thom isomorphism $$KK^1(C_0(M\setminus Per(\alpha))\rtimes \R,C_0(\widehat{Per}(\alpha))\rtimes\R)=KK^1(C_0(M\setminus Per(\alpha)),C_0(\widehat{Per}(\alpha)))$$ the element $[z_4]$ corresponds to the exact sequence of commutative algebras $$0 \to  C_0(\widehat{Per}(\alpha))  \longrightarrow C_0(Y)\longrightarrow C_0(M\setminus Per(\alpha))\to 0$$

\item Under the Takesaki-Takai isomorphism $C_0(\widehat{Per}(\alpha))\otimes \cK\simeq (C_0(\widehat{Per}(\alpha))\rtimes S^1)\rtimes \Z$ the element $[q]$ in
\begin{multline*}
KK(C_0(\widehat{Per}(\alpha))\rtimes \R,C_0(\widehat{Per}(\alpha))\rtimes S^1) = KK^1(C_0(\widehat{Per}(\alpha)),C_0(\widehat{Per}(\alpha))\rtimes S^1)\\ = KK^1((C_0(\widehat{Per}(\alpha))\rtimes S^1)\rtimes \Z,C_0(\widehat{Per}(\alpha))\rtimes S^1)
\end{multline*}
is the one associated the Pimsner Voiculescu exact sequence $$0\to B\otimes \cK \longrightarrow \gT \longrightarrow B\rtimes \Z\to0$$ (here $B= C_0(\widehat{Per}(\alpha))\rtimes S^1$).
\end{itemize}
\end{remarks}






\appendix
\bigskip
\section{When the classifying spaces are not manifolds}\label{appendix}
\label{app:Enotmanifold}

We finally explain how one should be able to to get rid of the assumption on the classifying spaces: we just assume that the foliation $\cF$ has a nice decomposition with Hausdorff Lie groupoids $\cG_i$ but  the classifying spaces $E_i$ for proper actions are not manifolds. 

In order to construct a left-hand side and a Baum-Connes map for $C^*(M,\cF)$, we just need to construct a left-hand side for a mapping cone of a morphism $\pi:G\to G'$ of Hausdorff Lie groupoids which is a submersion and the identity at the level of objects. As in the particular cases considered here, we then may construct left-hand sides for mapping tori and then of telescopic algebras.

In fact, given such a morphism $\pi:G\to G'$ we just have to show that:
\begin{enumerate}\renewcommand\theenumi{\roman{enumi}}
\renewcommand\labelenumi{(\rm {\theenumi})}
\item we may express the left-hand side of  $G$ and $G'$ as the $K$-theory of $C^*$-algebras $T$ and $T'$;
\item the Baum-Connes maps are given by elements $\mu\in KK(T,C^*(G))$ and $\mu'\in KK(T',C^*(G'))$;
\item  we may construct an element $x\in KK(T,T')$ such that $\pi_{*}(\mu)=x\otimes \mu'$.
\end{enumerate}
Then write $x=[f]^{-1}\otimes [g]$ where $f:D\to T$ and $g:D\to T'$ are morphisms with $f$ a $K$-equivalence. A left-hand side for $C_\pi$ is then the cone $C_{g}$ of $g$. As $f^*(\pi_*(\mu))=[f]\otimes \mu\otimes [\pi]=[g]\otimes \mu'$, me may construct an element $\tilde \mu\in KK(C_{g},C_\pi)$ as in Lemma \ref{lem:trivlemma} which defines the desired Baum-Connes map.

To do so, recall that if $G$ is a Hausdorff Lie groupoid, then the left-hand side for the Baum-Connes map can be described in the Baum-Douglas way (see \cite{BaumDouglas1, BaumDouglas2, BaumConnes, BaumConnesHigson, Tu1, Tu}):  there is an inductive limit of manifolds $(M_{k})_{k\in \N}$ with maps $h_{k}:M_{k}\to M_{k+1}$, forming a sequence of  approximations  of $E$. We may assume that the maps $q_k:M_k\to G^{(0)}$ are $K$-oriented in a $G$-equivariant way, and therefore so are the maps $h_{k}$. We also assume that the dimensions of all the $M_k$ are equal modulo $2$. Then the ``left-hand side'' $K_*^{top}(G)$ is the inductive limit $\varinjlim_{k}(K_*(C_0(M_k)\rtimes G),(h_k)!)$.

The Baum-Connes map on the image of $K_*(C_0(M_k)\rtimes G)$ is given by the element $(q_k)!$.

Put $A_k=C_0(M_k)\rtimes G$.

The same construction is then given for the groupoid $G'$, yielding proper $G'$-manifolds $M'_k$, maps $h'_k:M'_k\to M'_{k+1}$, algebras $A'_k=C_0(M'_k)\rtimes G'$, \emph{etc.} 

We may (and will) also assume that $h_{k+1}(M_{k})/G$ is relatively compact in $M_{k+1}/G$. As in section \ref{subsubsec:Subm}, let $\Gamma=\ker \pi$. As $G'$ acts properly on $M_{k+1}/\Gamma$ and by the relative compactness assumption, we may embed $h_{k}(M_{k-1})/\Gamma)$ in a manifold approximating the classifying space for proper actions $E'$ of $G'$. Using a subsequence of the $M'_k$ we may assume that we are given equivariant smooth maps $\ell_k:M_k\to M'_k$. Up to taking again a subsequence, we may further assume that the maps $h'_k\circ\ell_{k}$ and $\ell_{k+1}\circ h_k$ are homotopic (where $h'_k:M_{k}'\to M'_{k+1}$). Note that  the maps $\ell_k$ are automatically $K$-oriented, and thus we obtain $KK$-elements $(\ell_k)!\in KK(A_k,A'_k)=KK(C_0(M_k)\rtimes G,C_0(M'_k)\rtimes G')$ satisfying $(\ell_{k})!\otimes (h'_k)!=(h_k)!\otimes(\ell_{k+1})!$. 

Now, using \cite[Appendix]{VLHO}, we find (explicit) algebras $D_k$ and morphisms $f_k:D_k\to A_k$ which are $K$-equivalences and $g_k:D_k\to A'_k$ such that $(\ell_k)!=[f_k]^{-1}\otimes [g_k]$.

Put then $x_k=[f_k]\otimes (h_k)!\otimes [f_{k+1}]^{-1}\in KK(D_k,D_{k+1})$. We find 
\begin{eqnarray*}
(g_{k+1})_*(x_k)&=&[f_k]\otimes (h_k)!\otimes [f_{k+1}]^{-1}\otimes [g_{k+1}]\\
&=&[f_k]\otimes (h_k)!\otimes (\ell_{k+1})!=[f_k]\otimes (\ell_{k})!\otimes (h'_k)!\\
&=&[g_k]\otimes (h'_k)!
\end{eqnarray*}

As explained in section \ref{subsec:lhscones},  using precise homotopies between Kasparov bimodules representing these elements, we may then construct elements $y_k\in KK(C_{g_k},C_{g_{k+1}})$.

Note also, that we have the equalities $\pi_*((q_k)_!)=(\ell_k)!\otimes (q'_k)!\in KK(A_k,C^*(G'))$ as in proposition \ref{prop:lhspi}, yielding an element $z_k\in KK(C_{g_k},C_{\pi })$.

In order to construct the left-hand side for the mapping cone we need to make the following assumption - which could be true in general:

\begin{assumption}\label{ass:homotopy}
We assume that the homotopies used in the constructions of $y_k$ and $z_k$ are well matching, so that we have the equality $y_k\otimes z_{k+1}=z_k$.
\end{assumption}

One may then construct, for each $k$, $C^*$-algebras $B_k$ and $B'_k$, morphisms $u_k:B_k\to D_{k}$ and $u'_k:B_k\to A'_{k}$ which are $KK$-equivalences and $v_k:B_k\to D_{k+1}$ and $v'_k:B_k\to D_{k+1}$ such that  $x_k=[u_k]^{-1}\otimes [v_k]$ and $(h'_k)!=[u'_k]^{-1}\otimes [v'_k]$ (using \cite[Appendix]{VLHO}).

left-hand sides for $G$ and $G'$ (up to a shift of dimension by $1$) are then the infinite telescopic algebras $T=T(v,u)$ and $T'=T(v',u')$. 
 
These algebras are mapping tori $\cT(\check u,\check v)$ and $\cT(\check u',\check v')$ where $$\check u,\check v:\widecheck B=\bigoplus_{k=1}^{+\infty} B_k\to \widecheck D=\bigoplus_{k=0}^{+\infty} D_k\ \ \hbox{and} \ \ \check u',\check v':\widecheck B'=\bigoplus_{k=1}^{+\infty} B'_k\to \widecheck A'=\bigoplus_{k=0}^{+\infty} A'_k$$ are the maps given by $$\check u(x_1\ldots,x_k,\ldots)=(0,u_1(x_1),\ldots, u_k(x_k),\ldots)\ \ \hbox{and}\ \ \check v(x_1\ldots,x_k,\ldots)=(v_1(x_1),\ldots, v_k(x_k),\ldots),$$
and analogous formulae for $\check u'$ and $\check v'$.

The families of $(q_k)!$ and $(q'_k)!$ give elements $\check q!\in KK(\widecheck D,C^*(G))$ and $\check q'!\in KK(\widecheck A',C^*(G'))$.

The homotopy between $[\check u]\otimes q!$ and $[\check v]\otimes q!$ (\resp $[\check u']\otimes q'!$ and $[\check v']\otimes q'!$) gives rise to the element $\mu_G\in KK(T,C^*(G))$ and $\mu'_G\in KK(T',C^*(G'))$.

We may now do the same construction at the mapping cone level: writing $y_k=[\alpha_k]^{-1}\otimes [\beta_k]$ where $\alpha_k:V_k\to C_{g_k}$ and $\beta_k:V_k\to C_{g_{k+1}}$ are morphisms we may consider the infinite telescope $T(\beta,\alpha)=\cT(\check \alpha,\check \beta)$ as a left-hand side $K_{*,top}(\pi)$ for $C_\pi$. The element $\check z$ defined by the $z_k$'s gives an element of $KK(\bigoplus C_{g_k},C_\pi)$; a homotopy between $[\check \alpha] \otimes \check z$ and $[\check \beta] \otimes \check z$ (based on our assumption) gives rise to the Baum-Connes element $\mu_\pi\in KK(T(\beta,\alpha),C_\pi)$ and thus a morphism $\mu_\pi:K_{*,top}(\pi)\to K_*(C_\pi)$.

\begin{remark}
One may push a little further the above calculations. Indeed one needs to check that we have an exact sequence $$\xymatrix{K_{*,top}(G)\ar[rr]^{\pi_*}&&K_{*,top}(G')\ar[ld]\\ &K_{*,top}(\pi)\ar[lu]}$$ compatible with the mapping cone exact sequence. It then follows that if $G$ and $G'$ satisfy the (full version of the) Baum-Connes conjecture, then so does $C_\pi$.
\end{remark}


\begin{thebibliography}{10}

\bibitem{AbrahamRobbin}
R.~Abraham and J.~Robbin.
\newblock {\em Transversal mappings and flows}.
\newblock An appendix by Al Kelley. W. A. Benjamin, Inc., New York-Amsterdam,
  1967.

\bibitem{AH}
I.~Androulidakis and N.~Higson.
\newblock In preparation.

\bibitem{AS1}
I.~Androulidakis and G.~Skandalis.
\newblock The holonomy groupoid of a singular foliation.
\newblock {\em J. Reine Angew. Math.}, 626:1--37, 2009.

\bibitem{AS3}
I.~Androulidakis and G.~Skandalis.
\newblock The analytic index of elliptic pseudodifferential operators on a
  singular foliation.
\newblock {\em J. K-Theory}, 8(3):363--385, 2011.

\bibitem{AS2}
I.~Androulidakis and G.~Skandalis.
\newblock Pseudodifferential calculus on a singular foliation.
\newblock {\em J. Noncommut. Geom.}, 5(1):125--152, 2011.

\bibitem{AZ1}
I.~{Androulidakis} and M.~{Zambon}.
\newblock {Smoothness of holonomy covers for singular foliations and essential
  isotropy.}
\newblock {\em {Math. Z.}}, 275(3-4):921--951, 2013.

\bibitem{Arveson}
W.~B. Arveson.
\newblock Notes on extensions of ${C}^{\ast}$-algebras.
\newblock {\em Duke Math. J.}, 44:329--355, 1977.

\bibitem{BaajSkandalis}
S.~Baaj and G.~Skandalis.
\newblock {$C^\ast$}-alg\`ebres de {H}opf et th\'eorie de {K}asparov
  \'equivariante.
\newblock {\em $K$-Theory}, 2(6):683--721, 1989.

\bibitem{BaumConnes}
P.~Baum and A.~Connes.
\newblock Geometric {$K$}-theory for {L}ie groups and foliations.
\newblock {\em Enseign. Math. (2)}, 46(1-2):3--42, 2000.

\bibitem{BaumConnesHigson}
P.~Baum, A.~Connes, and N.~Higson.
\newblock Classifying space for proper actions and {$K$}-theory of group
  {$C^\ast$}-algebras.
\newblock In {\em {$C^\ast$}-algebras: 1943--1993 ({S}an {A}ntonio, {TX},
  1993)}, volume 167 of {\em Contemp. Math.}, pages 240--291. Amer. Math. Soc.,
  Providence, RI, 1994.

\bibitem{BaumDouglas2}
P.~Baum and R.~G. Douglas.
\newblock Index theory, bordism, and {$K$}-homology.
\newblock In {\em Operator algebras and {$K$}-theory ({S}an {F}rancisco,
  {C}alif., 1981)}, volume~10 of {\em Contemp. Math.}, pages 1--31. Amer. Math.
  Soc., Providence, R.I., 1982.

\bibitem{BaumDouglas1}
P.~Baum and R.~G. Douglas.
\newblock {$K$} homology and index theory.
\newblock In {\em Operator algebras and applications, {P}art {I} ({K}ingston,
  {O}nt., 1980)}, volume~38 of {\em Proc. Sympos. Pure Math.}, pages 117--173.
  Amer. Math. Soc., Providence, R.I., 1982.

\bibitem{Bochner}
S.~{Bochner}.
\newblock {Compact groups of differentiable transformations.}
\newblock {\em {Ann. Math. (2)}}, 46:372--381, 1945.

\bibitem{RBrown}
R.~Brown.
\newblock Groupoids as coefficients.
\newblock {\em Proc. London Math. Soc. (3)}, 25:413--426, 1972.

\bibitem{PCR}
P.~Carrillo~Rouse.
\newblock A {S}chwartz type algebra for the tangent groupoid.
\newblock In {\em {$K$}-theory and noncommutative geometry}, EMS Ser. Congr.
  Rep., pages 181--199. Eur. Math. Soc., Z\"urich, 2008.

\bibitem{ConnesThom}
A.~Connes.
\newblock An analogue of the {T}hom isomorphism for crossed products of a
  {$C^{\ast} $}-algebra by an action of {${\bf R}$}.
\newblock {\em Adv. in Math.}, 39(1):31--55, 1981.

\bibitem{ConnesNCG}
A.~Connes.
\newblock {\em Noncommutative geometry}.
\newblock Academic Press Inc., San Diego, CA, 1994.

\bibitem{ConnesSkandalis}
A.~Connes and G.~Skandalis.
\newblock The longitudinal index theorem for foliations.
\newblock {\em Publ. Res. Inst. Math. Sci.}, 20(6):1139--1183, 1984.

\bibitem{DebordJDG}
C.~Debord.
\newblock Holonomy groupoids of singular foliations.
\newblock {\em J. Diff. Geom.}, 58(3):467--500, 2001.

\bibitem{Debord2013}
C.~Debord.
\newblock Longitudinal smoothness of the holonomy groupoid.
\newblock {\em C. R. Math. Acad. Sci. Paris}, 351(15-16):613--616, 2013.

\bibitem{DS1}
C.~Debord and G.~Skandalis.
\newblock Adiabatic groupoid, crossed product by {$\mathbb{R}_+^\ast$} and
  pseudodifferential calculus.
\newblock {\em Adv. Math.}, 257:66--91, 2014.

\bibitem{Higson1987}
N.~Higson.
\newblock On a technical theorem of {K}asparov.
\newblock {\em J. Funct. Anal.}, 73(1):107--112, 1987.

\bibitem{HLS}
N.~Higson, V.~Lafforgue, and G.~Skandalis.
\newblock Counterexamples to the {B}aum-{C}onnes conjecture.
\newblock {\em Geom. Funct. Anal.}, 12(2):330--354, 2002.

\bibitem{HilsumSkandalis}
M.~Hilsum and G.~Skandalis.
\newblock Morphismes {$K$}-orient\'es d'espaces de feuilles et fonctorialit\'e
  en th\'eorie de {K}asparov (d'apr\`es une conjecture d'{A}. {C}onnes).
\newblock {\em Ann. Sci. \'Ecole Norm. Sup. (4)}, 20(3):325--390, 1987.

\bibitem{Julg}
P.~Julg.
\newblock Induction holomorphe pour le produit crois\'e d'une {$C^{\ast}
  $}-alg\`ebre par un groupe de {L}ie compact.
\newblock {\em C. R. Acad. Sci. Paris S\'er. I Math.}, 294(5):193--196, 1982.

\bibitem{Kasparov1981}
G.~G. Kasparov.
\newblock The operator {$K$}-functor and extensions of {${C}^*$}-algebras.
\newblock {\em Math. U.S.S.R. Izv}, 16:513--572, 1981.

\bibitem{Kasparov1984}
G.~G. Kasparov.
\newblock Lorentz groups: {$K$}-theory of unitary representations and crossed
  products.
\newblock {\em Dokl. Akad. Nauk SSSR}, 275(3):541--545, 1984.

\bibitem{Kasparov1988}
G.~G. Kasparov.
\newblock Equivariant {$KK$}-theory and the {N}ovikov conjecture.
\newblock {\em Invent. Math.}, 91(1):147--201, 1988.

\bibitem{VLHO}
V.~Lafforgue.
\newblock {$K$}-th\'eorie bivariante pour les alg\`ebres de {B}anach,
  groupo\"\i des et conjecture de {B}aum-{C}onnes. {A}vec un appendice
  d'{H}erv\'e {O}yono-{O}yono.
\newblock {\em J. Inst. Math. Jussieu}, 6(3):415--451, 2007.

\bibitem{LeGall}
P.-Y. {Le Gall}.
\newblock {Equivariant Kasparov theory and groupoids. I. (Th\'eorie de Kasparov
  \'equivariante et groupo\"{\i}des. I.)}.
\newblock {\em {$K$-Theory}}, 16(4):361--390, 1999.

\bibitem{MonthPie}
B.~Monthubert and F.~Pierrot.
\newblock Indice analytique et groupo\"\i des de {L}ie.
\newblock {\em C. R. Acad. Sci. Paris S\'er. I Math.}, 325(2):193--198, 1997.

\bibitem{NWX}
V.~Nistor, A.~Weinstein, and P.~Xu.
\newblock Pseudodifferential operators on differential groupoids.
\newblock {\em Pacific J. Math.}, 189(1):117--152, 1999.

\bibitem{Rieffel}
M.~A. Rieffel.
\newblock Strong {M}orita equivalence of certain transformation group
  {$C\sp*$}-algebras.
\newblock {\em Math. Ann.}, 222(1):7--22, 1976.

\bibitem{Rosenberg-Schochet}
J.~Rosenberg and C.~Schochet.
\newblock The {K}\"unneth theorem and the universal coefficient theorem for
  {K}asparov's generalized {$K$}-functor.
\newblock {\em Duke Math. J.}, 55(2):431--474, 1987.

\bibitem{Sk}
G.~Skandalis.
\newblock On the strong ext bifuctor.
\newblock Preprint, Queen's University, 1984.

\bibitem{SkandalisJFA}
G.~Skandalis.
\newblock Some remarks on {K}asparov theory.
\newblock {\em J. Funct. Anal.}, 56:337--347, 1984.

\bibitem{Stefan}
P.~Stefan.
\newblock Accessible sets, orbits, and foliations with singularities.
\newblock {\em Proc. London Math. Soc. (3)}, 29:699--713, 1974.

\bibitem{Sussmann}
H.~J. Sussmann.
\newblock Orbits of families of vector fields and integrability of
  distributions.
\newblock {\em Trans. Amer. Math. Soc.}, 180:171--188, 1973.

\bibitem{Torpe}
A.~M. Torpe.
\newblock {$K$}-theory for the leaf space of foliations by {R}eeb components.
\newblock {\em J. Funct. Anal.}, 61(1):15--71, 1985.

\bibitem{Tu1}
J.-L. Tu.
\newblock La conjecture de {B}aum-{C}onnes pour les feuilletages moyennables.
\newblock {\em $K$-Theory}, 17(3):215--264, 1999.

\bibitem{Tu}
J.-L. Tu.
\newblock The {B}aum-{C}onnes conjecture for groupoids.
\newblock In {\em {$C^*$}-algebras ({M}\"unster, 1999)}, pages 227--242.
  Springer, Berlin, 2000.

\bibitem{Voiculescu}
D.-V. Voiculescu.
\newblock On the existence of quasicentral approximate units relative to normed
  ideals. part i.
\newblock {\em J. Funct. Anal.}, 91(1):1--36, June 1990.

\bibitem{Wassermann}
A.~Wassermann.
\newblock Une d\'emonstration de la conjecture de {C}onnes-{K}asparov pour les
  groupes de {L}ie lin\'eaires connexes r\'eductifs.
\newblock {\em C. R. Acad. Sci. Paris S\'er. I Math.}, 304(18):559--562, 1987.

\end{thebibliography}

\end{document}